\pdfoutput=1
\documentclass[11pt]{article}
\usepackage{amssymb, amsmath, amsthm}
\usepackage[numbers]{natbib}

\usepackage{amsfonts,mathtools,mathrsfs,mathtools,color,enumitem}
\usepackage[colorlinks]{hyperref}
\usepackage{graphicx, subcaption}
\graphicspath{ {./images/} }
\usepackage{multicol}
\usepackage{float}
\newtheorem{definition}{Definition}
\newtheorem{theorem}{Theorem}
\newtheorem*{theorem*}{Theorem}
\newtheorem{lemma}{Lemma}
\newtheorem{corollary}{Corollary}

\newtheorem{assumption}{Assumption}
\newtheorem{conjecture}{Conjecture}
\newtheorem{proposition}{Proposition}

\renewcommand{\P}{\mathcal{P}}
\newcommand{\R}{\mathbb{R}} 
\newcommand{\Z}{\mathcal{Z}} 
\newcommand{\N}{\mathbb{N}}
\newcommand{\M}{\mathcal{M}}

\renewcommand{\part}[2]{\frac{\partial #1}{\partial #2}}

\newcommand{\A}{\mathcal{A}}
\newcommand{\B}{\mathcal{B}}

\newcommand{\Q}{\mathcal{Q}}

\newcommand{\lnorm}{\left|\left|}
\newcommand{\rnorm}{\right|\right|}

\newcommand{\ad}{\mathrm{ad}}

\newcommand{\dg}{\textsf{dg}}
\newcommand{\frA}{\mathfrak{A}}
\newcommand{\frB}{\mathfrak{B}}

\newcommand{\stirlingtwo}[2]{\genfrac{\lbrace}{\rbrace}{0pt}{}{#1}{#2}}

\DeclareMathOperator\arcsinh{arcsinh}

\providecommand{\keywords}[1]{\textbf{\textit{Keywords---}} #1}

\let\oldFootnote\footnote
\newcommand\nextToken\relax

\renewcommand\footnote[1]{%
    \oldFootnote{#1}\futurelet\nextToken\isFootnote}

\newcommand\isFootnote{%
    \ifx\footnote\nextToken\textsuperscript{,}\fi}

\title{A Symplectic Analysis of Alternating Mirror Descent}
\author{Jonas Katona~\thanks{Department of Applied \& Computational Mathematics, Yale University. Email: \texttt{jonas.katona@yale.edu}.} \and Xiuyuan Wang~\thanks{Department of Computer Science, Yale University. Email: \texttt{xiuyuan.wang@yale.edu}.} \and Andre Wibisono\thanks{Department of Computer Science, Yale University. Email: \texttt{andre.wibisono@yale.edu}.}}
\begin{document}

\maketitle

\begin{abstract}    
    Motivated by understanding the behavior of the Alternating Mirror Descent (AMD) algorithm for bilinear zero-sum games, we study the discretization of continuous-time Hamiltonian flow via the symplectic Euler method. We provide a framework for analysis using results from Hamiltonian dynamics, Lie algebra, and symplectic numerical integrators, with an emphasis on the existence and properties of a conserved quantity, the modified Hamiltonian (MH), for the symplectic Euler method. 
    We compute the MH in closed-form when the original Hamiltonian is a quadratic function, and show that it generally differs from the other conserved quantity  known previously in that case. We derive new error bounds on the MH when truncated at orders in the stepsize in terms of the number of iterations, $K$, and use these bounds to show an improved $\mathcal{O}(K^{1/5})$ total regret bound and an $\mathcal{O}(K^{-4/5})$ duality gap of the average iterates for AMD. Finally, we propose a conjecture which, if true, would imply that the total regret for AMD scales as $\mathcal{O}\left(K^{\varepsilon}\right)$ and the duality gap of the average iterates as $\mathcal{O}\left(K^{-1+\varepsilon}\right)$ for any $\varepsilon>0$, and we can take $\varepsilon=0$ upon certain convergence conditions for the MH.
\end{abstract}


\section{Introduction}\label{intro}
 
In its original physics setting, Hamiltonian flow is used to describe the dynamics of physical systems where the total energy---i.e., the Hamiltonian function $H$---is conserved. 
Perhaps after some perturbations or upon a change of variables, Hamiltonian flows also appear in other scientific fields as dynamical models of systems, even when these systems do not describe physical processes. Some examples which motivate our work are as follows. In game theory, Hamiltonian structure appears in the dynamics of multi-agent systems~\citep{hofbauer1996evolutionary,bailey2019multi}; this perspective has been useful for understanding emergent phenomena such as chaotic behavior in the systems~\citep{sato2002chaos, piliouras2021constants}, and for designing more efficient multi-agent learning algorithms~\citep{balduzzi2018mechanics}. In online learning, Hamiltonian structure appears in the continuous-time view of best-response dynamics such as fictitious play~\citep{ostrovski2011piecewise,van2011hamiltonian}.

Hamiltonian flows also appear in many problems motivated by computer science. In optimization, Hamiltonian dynamics (with friction or damping) appear in the continuous-time view of accelerated gradient methods; this Hamiltonian perspective has been useful for deriving new optimization algorithms and extending the concept of accelerated methods to more general optimization problems, \citep[see, e.g.,][]{su2016differential, wibisono2016variational, betancourt2018symplectic,wilson2019accelerating, diakonikolas2021generalized, wilson2021lyapunov, duruisseaux2022accelerated}. A Hamiltonian flow perspective has also been used for deriving new descent methods for optimization~\citep{o2019hamiltonian, teel2019first, de2023improving, wang2024hamiltonian}. In sampling or probabilistic applications such as Bayesian inference, Hamiltonian dynamics appear as a crucial component of the Hamiltonian Monte Carlo (HMC) algorithm~\citep{Duane87,neal2011mcmc}, which is a practical sampling algorithm with good empirical performance, and which is the default method in probabilistic programming packages such as Stan~\citep{carpenter2017stan} and PyMC3~\citep{salvatier2016probabilistic}.
The theoretical guarantees of HMC are strongly tied to conservation of energy in the discretization scheme used for implementing the Hamiltonian dynamics, see e.g.~\cite{HG14, Betancourt16, bou2018geometric, MV18, lee2018convergence, CV19, betancourt2019convergence, CDWY20, lee2020logsmooth, kook2022sampling}.

Hence, the accurate simulation of Hamiltonian flows is of great importance for applications across game theory, learning theory, and computer science at large. Unless one can explicitly integrate the Hamiltonian flow for the Hamiltonian in question (which is rare), one generally resorts to discretizing the flow by subdividing the time interval and computing iterative approximations to the ODE solution at those times. For these \textit{numerical integrators}, we generally speak of a stepsize $\eta>0$, such that as $\eta\rightarrow 0$, the numerical solution approaches the true solution \citep{Atkinson}. 

Since the true Hamiltonian flow preserves $H$---the Hamiltonian function or ``energy''---one might also expect a numerical integrator to conserve a quantity approximating $H$ for sufficiently small $\eta$, which we call a \textit{modified Hamiltonian} $\widetilde{H}_{\eta}$, such that as $\eta \to 0$, the modified Hamiltonian recovers the true Hamiltonian: $\widetilde{H}_{\eta} \to H$, in some notion of convergence. Analogous to how knowing $H$ and its level sets are useful for geometric analyses of Hamiltonian flow, knowing $\widetilde{H}_{\eta}$ and its level sets would be similarly useful for analyzing the numerical integrator and its trajectories. However, a modified Hamiltonian $\widetilde{H}_{\eta}$ does not exist for many discrete-time approximations of Hamiltonian flow \citep[Chapter IX]{Hairer2006}---two counterexamples are given later in Section \ref{Sec:HamFlowDisc}. Thus, guaranteeing that the numerical solution matches the behavior of the continuous-time dynamics can be difficult and subtle, especially when trying to develop such guarantees rigorously for algorithmic applications.

Motivated by the above problem, this paper studies the following question: How can one implement Hamiltonian flows as discrete-time algorithms with provable energy conservation errors? Symplectic integration is the study of how to accurately discretize Hamiltonian flows as discrete-time algorithms and is a long-studied field with many classical results that are designed precisely to answer our first question \citep[e.g.,][]{Channell1990,Yoshida1990,Yoshida1993,Hairer2006}. However, existing results are not sufficiently algorithmic for many applications of interest in algorithmic game theory and computer science. For instance, while all symplectic discretizations {\em in principle} conserve a modified Hamiltonian $\widetilde{H}_{\eta}$, unfortunately, $\widetilde{H}_{\eta}$ can only be described as a formal power series and there is no guarantee on its actual convergence or existence in general. Furthermore, since the series for $\widetilde{H}_{\eta}$ can only be summed up to a finite number of terms in practice, it remains unanswered as to how using the truncated forms of the series for $\widetilde{H}_{\eta}$ induces an error in their conservation from an algorithmic perspective without assuming a priori that the numerical solution remains within a compact set \citep[Chapter IX]{Hairer2006}. This paper reviews classical results which partially answer the aforementioned questions and reveals new ones that refine these answers.

This study also sets the ground for a framework of a rigorous mathematical machinery that can be used for deriving algorithmic guarantees and new algorithms in game theory, learning theory, and computer science applications which relate to Hamiltonian dynamics. One particular example comes from algorithmic game theory. The joint behavior of two agents in a bilinear zero-sum game using greedy strategies in continuous time can be described via a Hamiltonian flow, and different discretization methods of the Hamiltonian flow correspond to different strategies for the two players in the game \citep{hofbauer1996evolutionary,bailey2019multi}. 
Out of these strategies, we focus on the {\em Alternating Mirror Descent} (AMD) algorithm for constrained zero-sum games~\citep{Wibisono2022}, which we show is related by duality to the discretization of a Hamiltonian flow that preserves a modified Hamiltonian $\widetilde{H}_{\eta}$. The conservation of $\widetilde{H}_{\eta}$ and the near-conservation of its truncated forms have direct implications on regret guarantees of the AMD algorithm, which we describe further in Section~\ref{Regret_Analysis}. We use several key properties of $\widetilde{H}_{\eta}$ and its order-by-order truncations in $\eta$ to generalize both the analyses done in~\cite{pmlr-v125-bailey20a} for the unconstrained setting of AMD with quadratic regularizers and the regret analysis of AMD in~\cite{Wibisono2022}.

\paragraph{Contributions.} We address the questions above and set the framework for relating algorithms in game theory to symplectic integrators by showing the following:
\begin{enumerate}
    \item Under a computable linear transformation and a Legendre transform, for any choice of Legendre-type regularizers in the bilinear zero-sum game setting and usual domain assumptions, the trajectories of the symplectic Euler method applied to a Hamiltonian system with a suitable Hamiltonian function transform into the trajectories of AMD. Thus, AMD inherits the properties which symplectic Euler has: A conserved quantity expressible via order-by-order corrections in $\eta$ and a continuous-time flow which interpolates its discrete-time trajectory.
    \item Order-by-order approximations in the stepsize, $\eta>0$, to the modified Hamiltonian $\widetilde{H}_{\eta}$ for symplectic Euler---and hence, the associated conserved quantity for AMD as well---are approximately conserved by the trajectories of symplectic Euler at least up to order $3$ in $\eta$; furthermore, we conjecture this statement holds for all orders in $\eta$, as we formalize in Conjecture \ref{Conjecture_MH_General}.
    \item Under some smoothness conditions and usual domain assumptions, AMD has an $\mathcal{O}(K^{1/5})$ total regret and an $\mathcal{O}(K^{-4/5})$ duality gap of the average iterates, where $K$ is the number of iterations of AMD. If the conjecture introduced in Section \ref{primaldual} is true in its entirety, then these can be improved to $\mathcal{O}\left(K^{\varepsilon}\right)$ and $\mathcal{O}\left(K^{-1+\varepsilon}\right)$, respectively, for any $\varepsilon>0$. 
    Even more strongly, if $\widetilde{H}_{\eta}$ converges absolutely, then the total regret is bounded for all $K$, i.e., aforementioned complexities become $\mathcal{O}\left(1\right)$ and $\mathcal{O}\left(K^{-1}\right)$, respectively. In particular, our results generalize those derived in \cite{Wibisono2022}, in which the authors show an $\mathcal{O}(K^{1/3})$ total regret and an $\mathcal{O}(K^{-2/3})$ duality gap of the average iterates for AMD.
\end{enumerate}
Furthermore, we compute and analyze the closed-form of $\widetilde{H}_{\eta}$ when the original Hamiltonian $H$ is quadratic. We find that $\widetilde{H}_{\eta}$ is \textit{different} from the conserved quantity found previously in the literature~\citep[e.g.,][]{pmlr-v125-bailey20a}.

\paragraph{Roadmap.}
The structure of the rest of the paper is as follows. In Section \ref{Sec:HamFlowDisc}, we review some basic definitions and properties of continuous-time Hamiltonian flow, state some foundational assumptions that will hold throughout the paper, and introduce the foundations of symplectic integrators---and in particular, the symplectic Euler method---alongside connections between symplectic integrators and algorithms used throughout algorithmic game theory and optimization. 
In Section \ref{Section_Settings}, we introduce and review the bilinear zero-sum games and how to measure the performance of algorithms used in their design, especially in the context of AMD and how it is defined. 
In Section \ref{primaldual}, we show how the AMD and symplectic Euler algorithms are related to one another via duality and a linear transformation in the same way as skew-gradient flow and Hamiltonian flow, respectively.  
In Section \ref{MH_more}, we define the modified Hamiltonian $\widetilde{H}_{\eta}$ for the symplectic Euler method via the Baker-Campbell-Hausdorff series, present several of its known characteristics, and introduce a conjecture describing how well approximations to $\widetilde{H}_{\eta}$ are conserved.
In Section \ref{Quadratic}, we present an example of a case where we can compute $\widetilde{H}_{\eta}$ in closed-form---the multivariate quadratic case, under some mild restrictions---and discuss some properties which can be derived about the trajectories of symplectic Euler in that case using $\widetilde{H}_{\eta}$. 
In Section \ref{Regret_Analysis}, we use various properties of $\widetilde{H}_{\eta}$ introduced in prior sections to derive the algorithmic performance guarantees mentioned above in our contributions. 
Finally, in Section \ref{Section_Conclude} we conclude by discussing further implications of this paper and further questions to study.

\section{Preliminaries}

\subsection{Notations}\label{note}
Let $\N = \{1,2,\dots\}$ be natural numbers, $\N_0 = \N \cup \{0\}$ be the non-negative integers, and $\Z = \{\cdots, -2, -1, 0, 1, 2, \cdots\}$ be the integers.
Let $\P, \Q \subseteq \R^d$, so $\P \times \Q \subseteq \R^{2d}$.
We denote an element of $\P \times \Q$ by $(p,q)$ where $p \in \P$ and $q \in \Q$.

For $n \in \N_0$, let $C^n(\P \times \Q)$ denote the space of $n$-times continuously differentiable functions $F \colon \P \times \Q \to \R$.
Let $C^\infty(\P \times \Q)$ denote the space of infinitely-differentiable functions $F \colon \P \times \Q \to \R$.

Given any differentiable $F:\P \times \Q\rightarrow\R$, we denote its gradient vector by
$$\nabla F(p,q) = \begin{pmatrix}
    \nabla_p F(p,q) \\
    \nabla_q F(p,q)
\end{pmatrix} \in \R^{2d}$$
where $\nabla_p F(p,q) = \frac{\partial}{\partial p} F(p,q) \in \R^d$ is the partial derivative in the $p$ coordinate, and $\nabla_q F(p,q) = \frac{\partial}{\partial q} F(p,q) \in \R^d$ is the partial derivative in the $q$ coordinate, both represented as column vectors.

For column vectors $u,v \in \R^d$, we denote their $\ell_2$-inner product by $u^\top v$.

Given a function  $f: \mathbb{R}^d \to \mathbb{R}$, we call $f$ \textit{$L$-smooth of order $k$} if $f$ is $k$-times continuously differentiable and there exists a constant $L>0$ such that for all $x\in \mathbb{R}^d$,
$
||\nabla^k f(x)||_{\textrm{op}} \leq L,
$
where $||\cdot||_{\textrm{op}}$ is the operator norm. Concretely, this means
$
|\nabla^k f(x)[v^{\otimes k}]| \leq L,
$
for all $v$ with $||v||_2 = 1$.

\subsection{Convex optimization}
Let $W$ be a subset of $\R^{d}$ with non-empty interior. A function $\Psi:W\rightarrow\R$ is a \textit{Legendre-type function} if $\Psi$ is strictly convex, $\Psi\in C^{1}\left(W\right)$, $\nabla \Psi:W\rightarrow\R^{d}$ is injective, and $\lnorm\nabla\Psi\left(w\right)\rnorm\rightarrow\infty$ as $w$ approaches $\partial W$ (the boundary of $W$) \citep[Section 26]{Rockafellar1996}. For any Legendre-type function $\Psi:W\rightarrow\R$, we may also define its convex conjugate $\Psi^{*}:\nabla\Psi\left(W\right)\rightarrow\R$ as follows:
\begin{gather}
    \Psi^{*}\left(p\right)\coloneqq \sup_{w\in W}\left\{w^{\top}p-\Psi\left(w\right)\right\}\quad \forall p\in\nabla\Psi\left(W\right).\nonumber
\end{gather}
The gradient of the convex conjugate of $\Psi$ is the inverse map of the gradient of $\Psi$---i.e., $\nabla\Psi^{*}=\left(\nabla\Psi\right)^{-1}$---see \cite{Ghaoui} or Theorem 26.5 from \cite{Rockafellar1996}.

The \textit{Bregman divergence} of a Legendre-type function $\Psi:W\rightarrow\R$ is the function $D_{\Psi}:W\times W\rightarrow\R$ defined as  
\begin{gather}
    D_{\Psi}(w,\widetilde{w})=\Psi(w)-\Psi(\widetilde{w})-\nabla\Psi(\widetilde{w})^{\top}(w-\widetilde{w})\nonumber
\end{gather}
for $w,\widetilde{w}\in W$. Since $\Psi$ is strictly convex, $D_{\Psi}(w,\widetilde{w})\geq 0$ and $D_{\Psi}(w,\widetilde{w})=0$ if and only if $w=\widetilde{w}$. 

The Bregman divergence can recover standard measures of ``distance,'' e.g., the Euclidean distance (which is symmetric) when $\Psi$ is the quadratic function, or the Kullback-Leibler divergence (which is asymmetric) when $\Psi$ is the negative entropy function on the simplex.

\subsection{Modes of convergence}\label{odes}

The following definitions and properties are taken from \cite{Ross2013}. Let $\left\{a_{k}\right\}_{k\in\N_{0}}$ be a sequence of real or complex numbers. We say that the series $S\coloneqq \sum_{k=0}^{\infty}{a_{k}}$ converges if and only if for all $\varepsilon>0$, there exists an $N=N\left(\varepsilon\right)\in\N_{0}$ and $L\in\R$ such that for all $n\geq N$, $\left|\sum_{k=0}^{n}{a_{k}}-L\right|<\varepsilon$. We say $S$ converges conditionally if $S$ converges but $\lim_{N\rightarrow\infty}\sum_{k=0}^{N}{\left|a_{k}\right|}=\infty$. Finally, $S$ converges absolutely if both $S$ and $\sum_{k=0}^{\infty}{\left|a_{k}\right|}$ converge.

Throughout this paper, we note how absolute convergence allows one to rearrange the order of a series without affecting its convergence or limit. For an absolutely convergent series of real or complex numbers, summing that series' terms in any order will still result in the same value. More precisely, if $S$ converges absolutely, then for any permutation $\sigma:\N_{>0}\rightarrow\N_{>0}$, we have $S=\sum_{k=0}^{\infty}{a_{\sigma\left(k\right)}}$. The Riemann rearrangement theorem further states that the converse of the former statement is true, although this theorem does not necessarily hold for sequences outside of the reals or complex numbers. 

\section{Symplectic Euler discretization of Hamiltonian flow}
\label{Sec:HamFlowDisc}

For now, let $\P, \Q \subseteq \R^d$ be closed sets with non-empty interiors, and denote $\Z = \P \times \Q$ the {\em phase space}.
We write an element $z \in \Z$ of the phase space as $z = (p,q)$ where $p \in \P$ and $q \in \Q$. 

Suppose we are given a {\em Hamiltonian function} $H \colon \Z \to \R$, which is an arbitrary differentiable function. Let $\Omega = \begin{pmatrix}
    0 & -I_d \\ I_d & 0
\end{pmatrix}$ be the {\em symplectic matrix},
where $I_d \in \R^{d \times d}$ is the identity matrix. Recall that the {\bf Hamiltonian flow} generated by Hamiltonian $H$ (with respect to the symplectic matrix $\Omega$) is the flow of the differential equation:
\begin{align}\label{Eq:HF}
    \dot z(t) = \Omega \nabla H(z(t))
\end{align}
starting from any $z(0) \in \Z$.
Here, $\dot z(t) = \frac{d}{dt} z(t)$ is the velocity, which is the time derivative of the position $z$.

We define the {\bf energy} of the flow at time $t$ to be the value of the Hamiltonian function $H(z(t))$.
One notable feature of the Hamiltonian flow is the conservation of energy:
For all $t \ge 0$,
\begin{align*}
    H(z(t)) = H(z(0)).
\end{align*}
Indeed, we can compute:
\begin{align*}
    \tfrac{d}{dt} H(z(t)) =\left(\nabla H(z(t))\right)^{\top}\dot z(t) = \left(\nabla H(z(t))\right)^{\top}\Omega \, \nabla H(z(t)) = 0,
\end{align*}
where the last equality holds since $\Omega = -\Omega^\top$ is skew-symmetric.
The Hamiltonian flow also conserves the symplectic form, although we do not use this property explicitly in this paper.

Where necessary, we will make the following separability assumption:

\begin{assumption}\label{Assm:Separable}
    The Hamiltonian function $H$ is separable, i.e.\ there exist differentiable functions $F \colon \P \to \R$ and $G \colon \Q \to \R$ such that
    \begin{align*}
        H(p,q) = F(p) + G(q)
    \end{align*}
    for all $(p,q) \in \Z$.
    We define the extensions $\widetilde{F} \colon \Z \to \R$ and $\widetilde{G} \colon \Z \to \R$ by $\widetilde{F}(p,q) = F(p)$ and $\widetilde{G}(p,q) = G(q)$ so that we can also write $H(z) = \widetilde{F}(z) + \widetilde{G}(z)$ for all $z \in \Z$.
\end{assumption} 

For the sake of simplicity, in the remainder of the paper, we drop the tildes from $\widetilde{F}$ and $\widetilde{G}$, as it will be clear from context whether the functions $F$ and $G$ or their extensions are being used. When the Hamiltonian is separable as $H=F+G$, we can write the Hamiltonian flow \eqref{Eq:HF} in terms of the coordinates of $z(t) = (p(t), q(t))$ as:
\begin{subequations}\label{Eq:HFComp}
    \begin{align}
        \dot{p}(t) &=-\nabla_{q}H(p(t),q(t)) = -\nabla G(q(t)) \\
        \quad \dot{q}(t) &=\nabla_{p}H(p(t),q(t)) = \nabla F(p(t)).
    \end{align}
\end{subequations}

To simulate the trajectories of continuous-time Hamiltonian flow under Hamiltonian $H$, we can discretize the Hamiltonian flow using a numerical integrator. Two of the simplest numerical integrators known in numerical analysis literature are the forward (explicit) and backward (implicit) Euler methods \citep{Atkinson}. For applications in min-max game theory, these integrators when applied to the Hamiltonian flow \eqref{Eq:HFComp} correspond to when the two players follow simultaneous mirror descent and simultaneous proximal mirror descent, respectively \citep{Wibisono2022}. 
For forward Euler applied to the Hamiltonian flow \eqref{Eq:HFComp}, we have
\begin{gather}
    p_{k+1} =p_{k}-\eta\nabla G\left(q_{k}\right), \quad q_{k+1} =q_{k}+\eta\nabla F\left(p_{k}\right),\label{forward}
\end{gather}
while for backward Euler applied to the Hamiltonian flow \eqref{Eq:HFComp}, we have
\begin{gather}
    p_{k+1} =p_{k}-\eta\nabla G\left(q_{k+1}\right), \quad q_{k+1} =q_{k}+\eta\nabla F\left(p_{k+1}\right).\label{backward}
\end{gather}
As shown in Appendix B of \cite{Wibisono2022}, when $H$ is convex, $H$ increases monotonically for forward Euler and decreases monotonically for backward Euler. In fact, the increase or decrease in $H$, respectively, is exponentially fast when $H$ is strongly convex. Hence, to achieve some sort of energy conservation in discrete time, a more sophisticated discretization scheme is required. 

A \textbf{symplectic integrator} is a numerical integrator $F_{\eta}:\Z\rightarrow\Z$---in the sense that $F_{\eta}\left(p_{k},q_{k}\right)=\left(p_{k+1},q_{k+1}\right)$---for Hamiltonian flow which preserves the symplectic form, i.e., $\left(JF_{\eta}\right)\Omega\left(JF_{\eta}\right)^{\top}=\Omega$ for any $\eta>0$ on all of $\Z$, where $JF_{\eta}$ is the Jacobian matrix of $F_{\eta}$. More generally, a change of coordinates of the Hamiltonian phase space which preserves the symplectic form is called a \textit{canonical transformation} (see, e.g., Chapter 2 in Part I of \cite{Lazutkin1993} or Section 5.3 in \cite{Jose1998}), and hence, all symplectic integrators are canonical transformations. In this paper, we do not use this symplectic conservation property directly as we are concerned with the energy conservation property explained below. 

Under certain conditions, a symplectic integrator conserves a modified Hamiltonian $\widetilde{H}_{\eta}$ expressible as a formal power series in $\eta$ and whose level sets approximate those of $H$ \citep{Yoshida1990,Yoshida1993,Hairer2006}. Moreover, the Hamiltonian flow generated by $\widetilde{H}_{\eta}$ exactly interpolates the discrete-time trajectories of the symplectic integrator. Some examples of symplectic integrators include the symplectic Euler, leapfrog, and Störmer–Verlet methods---\citep[Chapter VI]{Hairer2006} gives a comprehensive overview of different symplectic integrators and how to derive new ones of arbitrarily high accuracy. Generally speaking, the modified Hamiltonian $\widetilde{H}_{\eta}$ of a given symplectic integrator might not converge (see, e.g., \cite{Field2003} or Chapter IX in \cite{Hairer2006}), in which case one must approximate $\widetilde{H}_{\eta}$ by taking the series up to some order in $\eta$. But even then, the symplectic integrator could still preserve other conserved quantities, including invariants of the original Hamiltonian flow \citep{Alsallami2018,ohsawa2023preservation}. 

By the discussion above, the forward and backward Euler methods are non-symplectic methods. The simplest symplectic integrator is the symplectic Euler method~\citep{Channell1990,Yoshida1993}, 
which applies the explicit Euler method for updating one component of the Hamiltonian flow~\eqref{Eq:HFComp}, and the implicit Euler method for updating the other.
Concretely, the {\bf symplectic Euler method} with step size $\eta > 0$ is defined as follows: At each iteration $k \in \N_0$, from the current iterate $(p_k, q_k) \in \Z$, we compute the next iterate by the updates
\begin{subequations}\label{init}
\begin{align}
        p_{k+1} &=p_{k}-\eta \nabla_{q}H\left(p_{k+1},q_{k}\right)= p_{k}-\eta\nabla G\left(q_{k}\right) \label{Eq:SymEuler-p} \\ 
        q_{k+1} &=q_{k}+\eta\nabla_{p}H\left(p_{k+1},q_{k}\right)= q_{k}+\eta\nabla F\left(p_{k+1}\right), \label{Eq:SymEuler-q}
\end{align}
\end{subequations}
since we assumed that the Hamiltonian $H = F+G$ is separable (Assumption~\ref{Assm:Separable}). We will define the modified Hamiltonian $\widetilde{H}_{\eta}$ for symplectic Euler and introduce some of its properties in Section \ref{MH_more}.

For the game-theoretic application, the symplectic Euler updates \eqref{init} correspond to the Alternating Mirror Descent algorithm in the dual space after a linear transformation; see Section~\ref{primaldual}. Before characterizing this correspondence, we review the necessary background on game theory and mirror descent in Section \ref{Section_Settings} below.    

\section{Alternating Mirror Descent}\label{Section_Settings}

In this paper, we are also concerned with the problem of finding the Nash equilibrium of zero-sum games with a bilinear payoff matrix $A \in \R^{d\times d}$ for $d \in \N$ in \textit{constrained} strategy spaces $\mathcal{A}$ and $\mathcal{B}$, which are both closed and convex subsets of $\R^{d}$: 
\begin{gather}
    \min_{a\in\mathcal{A}} \,{\max_{b\in\mathcal{B}} \; a^{\top}Ab} \label{ineqq}.
\end{gather}
A Nash equilibrium of the bilinear zero-sum game above is a pair of points $\left(a^{*},b^{*}\right)\in\A\times\B$ which satisfies the saddle-point inequality:
\begin{gather}
    \left(a^{*}\right)^{\top}Ab\leq \left(a^{*}\right)^{\top}Ab^{*}\leq a^{\top}Ab^{*}, \qquad \text{ for all } ~ a \in \mathcal{A}, b \in \mathcal{B}.\label{ineq}
\end{gather}
By von Neumann's min-max theorem \citep{vNeumann1928min-max}, $\left(a^{*},b^{*}\right)$ always exists, provided that $\mathcal{A},\mathcal{B}$ are compact. 
We can measure the convergence to a Nash equilibrium $\left(a^{*},b^{*}\right)\in\mathcal{A}\times\mathcal{B}$ via the \textbf{duality gap} function $\dg:\mathcal{A}\times\mathcal{B}\rightarrow\R_{\geq 0}$ given by 
\begin{gather}
    \dg\left(a,b\right)=\max_{\widetilde{b}\in\mathcal{B}}{a^{\top}A\widetilde{b}}-\min_{\widetilde{a}\in\mathcal{P}}{\widetilde{a}^{\top}Ab}\qquad \text{ for }\left(a,b\right)\in\mathcal{A}\times\mathcal{B}.
\end{gather}
One can check that $\dg\left(a,b\right) \ge 0$ for all $(a,b) \in \A \times \B$, and $\dg\left(a^*,b^*\right)=0$ if and only if $\left(a^{*},b^{*}\right)$ is a Nash equilibrium. 
Hence, we can use the duality gap to measure the performance of an algorithm to find a Nash equilibrium in min-max games. 
We also define $\overline{\dg}_{K}$ to be the duality gap of the average iterates up to $K$ iterations, i.e., for a sequence of strategies $(a_0,b_0), \dots, (a_{K-1}, b_{K-1}) \in \A \times \B$, 
$$\overline{\dg}_{K}\coloneqq \dg\left(\frac{1}{K}\sum_{k=0}^{K-1}{a_{k}}, \; \frac{1}{K}\sum_{k=0}^{K-1}{b_{k}}\right).$$  

A natural strategy is for each player to follow a greedy method, such as gradient descent, to optFimize their own objective function.
In the constrained setting when $\mathcal{A},\mathcal{B}\subsetneq\R^{d}$, it is natural for each player to follow a constrained greedy method instead of gradient descent, especially considering how the iterations of simultaneous gradient descent diverge from their initial value \citep{gidel2019diverge}. Hence, we consider the case where each player follows the \textit{mirror descent algorithm} \citep{Nemirovsky1983-vo} to minimize their loss functions over $\mathcal{A}$ and $\mathcal{B}$, respectively. This gives rise to an algorithm called \textbf{Alternating Mirror Descent (AMD)}. 

Consider regularizer functions $\alpha:\mathcal{A}\rightarrow\R$ and $\beta:\mathcal{B}\rightarrow\R$ to keep the strategies of each player in $\A$ and $\B$, respectively, where $\alpha$ and $\beta$ are each Legendre-type functions.
Let $D_{\alpha}:\A\times\A\rightarrow\R$ and $D_{\beta}:\B\times\B\rightarrow\R$ are the Bregman divergences of $\alpha$ and $\beta$, respectively. Then, the AMD algorithm consists of the following updates from a position $\left(a_{k},b_{k}\right) \in \A \times \B$ at iteration (or turn) $k \in \N_0$, to the next position $(a_{k+1}, b_{k+1}) \in \A \times \B$ given by:
\begin{subequations}\label{losses}
\begin{align}
       a_{k+1}&=\textrm{arg}\min_{a\in\mathcal{A}}\left\{a^{\top}Ab_{k}+\frac{1}{\eta}D_{\alpha}\left(a,a_{k}\right)\right\},\label{lossesa}
    \\
    b_{k+1}&=\textrm{arg}\min_{b\in\mathcal{B}}\left\{-a_{k+1}^{\top}Ab+\frac{1}{\eta}D_{\beta}\left(b,b_{k}\right)\right\},\label{lossesb}
\end{align}
\end{subequations}
where $\eta>0$ is the stepsize. 
Note, the first player updates their position $a$ based on $b_{k}$, and the second player updates their position $b$ based on the updated $a_{k+1}$.
This is the algorithm that was studied in~\cite{Wibisono2022}.
In the unconstrained setting with $\alpha, \beta$ being quadratic functions, AMD recovers the Alternating Gradient Descent (AGD) algorithm, which was studied in~\cite{pmlr-v125-bailey20a} and will be revisited later in Sections \ref{Quadratic} and \ref{Regret_Analysis} of this paper.

A relevant quantity is the \textbf{total regret} $R_{K}$ of the two players after $K$ iterations under AMD. We have the following definition of $R_{K}$ taken from \cite{Wibisono2022, pmlr-v125-bailey20a}: 

\begin{definition}
    The \textit{total regret} of both players after $K$ iterations of AMD is the best cumulative regret of both players in hindsight:
    \begin{gather}
        R_{K}\coloneqq \max_{(a,b)\in\mathcal{A}\times\mathcal{B}}\left\{R_{1,K}(a)+R_{2,K}(b)\right\},\nonumber
    \end{gather}
    where $R_{1,K}$ and $R_{2,k}$ are the regrets of the first and second player, respectively, defined with respect to static actions $a\in\mathcal{A}$ and $b\in\mathcal{B}$ as follows:
    \begin{align*}
        R_{1,K}(a) &\coloneqq \sum_{k=0}^{K-1}{\left(\frac{a_{k}+a_{k+1}}{2}\right)^{\top}Ab_{k}}-\sum_{k=0}^{K-1}{a^{\top}Ab_{k}}, \\
        R_{2,K}(b) &\coloneqq \sum_{k=0}^{K-1}{a_{k+1}Ab}-\sum_{k=0}^{K-1}{a_{k+1}^{\top}A\left(\frac{b_{k}+b_{k+1}}{2}\right)}.
    \end{align*}
\end{definition}
 
The motivation for the above definition is that, from iteration $k$ to iteration $k+1$ of the algorithm, there are two half-steps that happen: The first player updates from $a_{k}$ to $a_{k+1}$ while the second is at $b_{k}$, and then the second player updates from $b_{k}$ to $b_{k+1}$ while the first is at $a_{k+1}$. Thus, while the first player is at $a_{k}$ and updates to $a_{k+1}$, they observe the second player while the latter is at $b_{k}$; similarly, while the first player is at $b_{k}$ and updates to $b_{k+1}$, they observe the second player while the latter is at $a_{k+1}$.

The total regret $R_{K}$ is directly related to $\overline{\dg}_{K}$. Hence, $R_{K}$ can likewise be used to measure the performance of AMD in the context of min-max games. More precisely, we have the following relation.

\begin{lemma}[{\cite[Lemma~3.1]{Wibisono2022}}]\label{dgK_and_RK}
    Along AMD from any $(a_0, b_0) \in \A \times \B$, for any $K\geq 1$,
    \begin{gather}
        \overline{\dg}_{K}=\frac{1}{K}R_{K}-\frac{1}{2K}\left(a_{0}^{\top}Ab_{0}-a_{K}^{\top}Ab_{K}\right).\label{dG}
    \end{gather}
\end{lemma}

We note that for the constrained setting where $\A$ and $\B$ are bounded, the last term $\\\left(a_{0}^{\top}Ab_{0}-a_{K}^{\top}Ab_{K}\right)$ in \eqref{dG} is uniformly bounded in terms of the diameter of the sets, in which case Lemma \ref{dgK_and_RK} implies that $R_{K}$ alone determines the rate of decay of $\overline{\dg}_{K}$. Hence, the total regret $R_{K}$ can likewise be used to measure the performance of AMD in the context of min-max games.

Along the AMD algorithm~\eqref{losses}, the average iterates $\left(\frac{1}{K}\sum_{k=1}^{K}{a_{k}}, \; \frac{1}{K}\sum_{k=1}^{K}{b_{k}}\right)$ converges to the Nash equilibrium, as implied by when their duality gap vanishes: $\overline{\dg}_{K} \to 0$ as $K \to \infty$.
Classically, it is known that the duality gap $\overline{\dg}_K$ vanishes at a rate of $\mathcal{O}\left(K^{-1/2}\right)$ for the simultaneous mirror descent algorithm when the stepsize is $\eta=\Theta\left(K^{-1/2}\right)$.
In contrast, for AMD, the duality gap $\overline{\dg}_K$ was shown to vanish at a rate of $\mathcal{O}\left(K^{-2/3}\right)$~\citep[Corollary~3.3]{Wibisono2022} when the stepsize is $\eta=\Theta\left(K^{-1/3}\right)$. We will show how to improve the latter bound in Section \ref{Regret_Analysis}.

\section{Alternating Mirror Descent and Symplectic Euler Method} \label{primaldual}

Let $\frA\coloneqq \nabla\alpha\left(\A\right) \subseteq \R^d$ and $\frB\coloneqq \nabla\beta\left(\B\right) \subseteq \R^d$ be the images of $\mathcal{A}$ and $\mathcal{B}$ under $\nabla\alpha$ and $\nabla\beta$, respectively, which we refer to as the {\em dual spaces}. Let $f\coloneqq \alpha^{*} \colon \frA\rightarrow\R$ and $g\coloneqq \beta^{*} \colon\frB \rightarrow\R$ be the convex conjugates \citep[Section 26]{Rockafellar1996} of $\alpha$ and $\beta$, respectively. 
The following lemma describes how to translate between AMD and an equivalent algorithm in the dual spaces of $\A$ and $\B$:

\begin{lemma}
    Let $x_k \coloneqq  \nabla \alpha(a_k) \in \frA$ and $y_k \coloneqq  \nabla \beta(b_k) \in \frB$ be the dual variables of $a_k \in \A$, $b_k \in \B$, where $\A,\B\subseteq\R^d$ are closed and convex.
    If $(a_k, b_k)$ evolves following the AMD algorithm, then $\left(x_{k},y_{k}\right)$ evolves via the following update:
\begin{subequations}\label{DAMD}
    \begin{align}
    x_{k+1}&=x_{k}-\eta A\nabla g\left(y_{k}\right),\\
    y_{k+1}&=y_{k}+\eta A^{\top}\nabla f\left(x_{k+1}\right).
\end{align}
\end{subequations}
\end{lemma}

\begin{proof}
    Consider the primal updates in $\A$ given by \eqref{lossesa}, and define the mapping
    \begin{gather*}
        \varphi_k(a):=a^{\top}Ab_k+\frac{1}{\eta}D_{\alpha}\left(a,a_k\right)\qquad \text{ for }a\in\mathcal{A}.
    \end{gather*}
    Since $\alpha$ is Legendre-type on $\A$ and $D_{\alpha}\left(\cdot,a_k\right)$ consists of $\alpha$ plus an affine term, $D_{\alpha}\left(\cdot,a_k\right)$ is also Legendre-type. Hence, $\varphi_k$ is strictly convex and essentially smooth, whence it follows that $\varphi_k$ has a unique minimizer $a_{k+1}$ that lies in the relative interior of $\A$ \cite[Sections 26 and 27]{Rockafellar1996}.
    
    The first-order KKT condition for minimizing $\varphi_k$ over the closed convex set $\A$ gives
    \begin{gather*}
        0\in \nabla\varphi_k\left(a_{k+1}\right)+N_{\A}\left(a_{k+1}\right),
    \end{gather*}
    where $N_{\A}\left(a_{k+1}\right)$ is the normal cone of the convex set $\A$ at the point $a_{k+1}$ \citep[Definition~2.7 and Theorem~27.4]{Rockafellar1996}. Because $a_{k+1}$ lies in the relative interior of $\A$, we have $N_{\A}\left(a_{k+1}\right)=\{0\}$, whence it follows that $\nabla\varphi_k\left(a_{k+1}\right)=0$. 
    
    Finally, since $\nabla_{a}D_{\alpha}\left(a,a_{k}\right)=\nabla\alpha(a)-\nabla\alpha\left(a_{k}\right)$, we have
    \begin{gather}
        \nabla\varphi_k\left(a_{k+1}\right)=Ab_{k}+\frac{1}{\eta}\left(\nabla\alpha\left(a_{k+1}\right)-\nabla\alpha\left(a_{k}\right)\right)=0.\label{aequation}
    \end{gather}
    By an analogous argument applied to the updates in $\B$ given by \eqref{lossesb}, we also have
    \begin{gather}
        -A^{\top}a_{k+1}+\frac{1}{\eta}\left(\nabla\beta\left(b_{k+1}\right)-\nabla\beta\left(b_{k}\right)\right)=0.\label{bequation}
    \end{gather}
    Thus, by \eqref{aequation} and \eqref{bequation}, we can write the optimality condition for the AMD algorithm \eqref{losses} as
\begin{gather}\label{todual}
\nabla\alpha\left(a_{k+1}\right)=\nabla\alpha\left(a_{k}\right)-\eta Ab_{k}, \quad \nabla\beta\left(b_{k+1}\right)=\nabla\beta\left(b_{k}\right)+\eta A^{\top}a_{k+1}.
\end{gather}
Since $\nabla\phi^{*}=\left(\nabla\phi\right)^{-1}$ for any Legendre-type function $\phi:\R^{d}\rightarrow\R$ \citep[Theorem 26.5]{Rockafellar1996}, we can transform between $\left(a_k,b_k\right)$ and $\left(x_k,y_k\right)$ as follows:
\begin{subequations}\label{xayb}
\begin{align}
        x_k=\nabla\alpha\left(a_k\right)\quad&\iff\quad a_k=\nabla\alpha^{*}\left(x_k\right)=\nabla f\left(x_k\right),
    \\
    y_k=\nabla\beta\left(b_k\right)\quad&\iff\quad b_k=\nabla\beta^{*}\left(y_k\right)=\nabla g\left(y_k\right).
\end{align}
\end{subequations}
Substituting \eqref{xayb} into \eqref{todual} gives \eqref{DAMD}.
\end{proof}

We call the update~\eqref{DAMD} the \textbf{Dual Alternating Mirror Descent (DAMD)} algorithm. DAMD can be interpreted as a version of \textit{Natural Gradient Descent (NGD)} \cite{Raskutti2015,pmlr-v130-gunasekar21a} on the dual manifold $\frA\times\frB$, where the geometry is defined by the block-diagonal metric tensor
\begin{align*}
    G\left(x,y\right)&\coloneqq 
\begin{pmatrix}
    \nabla^2 f(x) & 0 \\ 
    0 & \nabla^2 g(y)
\end{pmatrix}
\qquad \text{ for }\left(x,y\right)\in\frA\times\frB,
\end{align*}
and the update directions are rotated by the skew-symmetric matrix $\Omega_{A}$, where
\begin{align}
    \Omega_{M}&\coloneqq \begin{pmatrix}
    0 & -M \\ M^\top & 0
\end{pmatrix}\label{OmegaA}
\end{align}
for any matrix $M\in\R^{d\times d}$. In fact, the equivalence between NGD and mirror descent---as established by \cite{Raskutti2015}---is structurally analogous to the equivalence we establish between DAMD and AMD, respectively. Moreover, when $f$ and $g$ are quadratic functions, DAMD coincides with the \textit{Alternating Gradient-descent-ascent (AltGDA)} algorithm in min-max optimization and game theory \citep{9048943,pmlr-v125-bailey20a,pmlr-v151-zhang22e}. 

We now explain the relationship between AMD and the symplectic Euler discretization of a Hamiltonian flow.
We recall that we can decompose any real square matrix into a product of real symmetric matrices. 
We state this fact as the following lemma, and refer the reader to~\citep{Bosch} for the proof and details on how to compute such a decomposition using the Jordan canonical form.

\begin{lemma}[{\citep[Theorem~2]{Bosch}}]\label{decomp}
    Every real matrix $A \in \R^{d \times d}$ can be written as a product $A = UV$ of real symmetric matrices $U,V \in \R^{d \times d}$.
\end{lemma}

Lemma \ref{decomp} allows us to transform between DAMD and symplectic Euler, as summarized below.

\begin{theorem}\label{Thm_Pfwd}
    Let $A=UV$ be the symmetric decomposition of the payoff matrix $A$, and let $\P,\Q\subseteq\R^d$. 
    Define the pushforward coordinates and maps
    \begin{subequations}\label{transforms}
        \begin{align}
    x_k&\coloneqq Up_k, \hspace{1.26cm} y_k\coloneqq Vq_k, \label{2.1_eq_NewCord}\\
F\left(p\right)&\coloneqq f\left(Up\right), \quad G\left(q\right)\coloneqq g\left(Vq\right)\label{2.1_eq_PfwdMap} 
\end{align}
    \end{subequations}
for all $k\in\N_{0}$ and all $(p,q)\in\Z$. Then, provided the domains satisfy $U\P=\frA$ and $V\Q=\frB$, the iterations $(p_k, q_k)$ of the symplectic Euler discretization \eqref{init} for the Hamiltonian flow generated by $H(p, q)=F(p)+G(q)$ push forward to the iterations $(x_k, y_k)$ of DAMD in~\eqref{DAMD}.
\end{theorem}
\begin{proof}
    By the chain rule and the definitions above, \begin{subequations}\label{chainrule}
        \begin{align}
    \nabla F(p) &= U \nabla f(Up)= U \nabla f(x), \label{toppp}\\
        \nabla G(q) &= V \nabla g(Vq)= V \nabla g(y).\label{botttom}
\end{align}
    \end{subequations}
    Then, starting from the symplectic Euler update~\eqref{Eq:SymEuler-p} after multiplying both sides by $U$ and using the chain rule \eqref{toppp}, we can write:
    \begin{align*}
        x_{k+1} \,=\, Up_{k+1} 
        \,&=\, Up_k - \eta U\nabla G(q_k) \\
        &=\, x_{k} - \eta UV\nabla g\left(y_{k}\right)\\
        &=\, x_{k}-\eta A\nabla g\left(y_{k}\right).
    \end{align*}
    Similarly, multiplying both sides of~\eqref{Eq:SymEuler-p} by $V$, using the chain rule \eqref{botttom}, and using the fact that $U$ and $V$ are symmetric, we get:
    \begin{align*}
        y_{k+1} \,=\, Vq_{k+1} 
        \,&=\, Vq_k - \eta V\nabla F(p_{k+1}) \\
        &=\, y_{k} - \eta VU\nabla f\left(x_{k+1}\right)\\
        &=\, y_{k}-\eta V^{T}U^{T}\nabla f\left(x_{k+1}\right)\\
        &=\, y_{k}-\eta A^{T}\nabla f\left(x_{k+1}\right).
    \end{align*}
    This concludes the proof.
\end{proof}
Hence, Theorem \ref{Thm_Pfwd} defines a pushforward mapping from the $(p,q)$-coordinates for symplectic Euler to the $(x,y)$-coordinates for DAMD. However, when $A$ is not invertible, at least one of $U$ or $V$ must not be invertible, in which case an inverse transformation from $(x,y)$ to $(p,q)$ need not exist. This will not pose a problem for analyses in the following sections, since these are done with respect to $(p,q)$, $F$, and $G$, and thus can always be translated to $(x,y)$, $f$, and $g$, respectively.

\paragraph{Relation between skew-gradient flow and Hamiltonian flow.}
The equivalence between DAMD and symplectic Euler in discrete time is algebraically analogous with the equivalence between skew-gradient flow \citep[e.g.,][]{Abernethy,Wibisono2022} and Hamiltonian flow in continuous time, respectively, in the sense that both are related via transformations in \eqref{transforms}. Skew-gradient flow can be thought of as a generalization of Hamiltonian flow in which the symplectic matrix $\Omega$ is replaced by the skew-symmetric matrix $\Omega_{A}$, as defined earlier in $\eqref{OmegaA}$.

Note that $\Omega_{I_{d}}=\Omega$.
Furthermore, when $A$ is an involutory matrix, i.e., $A^{2} = I_{d}$, then $\Omega_A$ preserves the symplectic form. Namely, we have
$$\Omega_A \Omega \Omega_A^\top =\begin{pmatrix}
    -A & 0 \\ 0 & -A^{\top}
\end{pmatrix}\begin{pmatrix}
    0 & A \\ -A^{\top} & 0
\end{pmatrix}=\begin{pmatrix}
    0 & -A^{2} \\ \left(A^{\top}\right)^{2} & 0
\end{pmatrix}=\begin{pmatrix}
    0 & -I_{d} \\ I_{d} & 0
\end{pmatrix}=
\Omega.$$ However, for min-max games, the payoff matrix $A$ can be arbitrary.

In skew-gradient flow, the phase space variables $\zeta=(x,y)\in U\P\times V\Q$ obey the differential equation 
\begin{align}
    \dot \zeta(t) = \Omega_{A}\nabla \mathsf{E}(\zeta(t))\nonumber
\end{align}
for some differentiable function $\mathsf{E}:U\P\times V\Q\rightarrow\R$. Via analogous steps as for showing the relationship between DAMD and symplectic Euler as above, the variables $(x,y)$ for skew-gradient flow are related to the variables $(p,q)$ for Hamiltonian flow via the same relationships \eqref{transforms} at any time: $x=Up$, $y=Vq$, and $H\left(p,q\right)\coloneqq \mathsf{E}\left(Up,Vq\right)$. Indeed, starting from Hamiltonian flow \eqref{Eq:HFComp}, multiplying both sides by $\begin{pmatrix} U & 0 \\ 0 & V \end{pmatrix}$, using the chain rule \eqref{chainrule}, and noting that $U$ and $V$ are symmetric, we have
\begin{subequations}\label{skew}
\begin{align}
    \dot \zeta &=\begin{pmatrix} \dot x \\ \dot y \end{pmatrix}=\begin{pmatrix} U & 0 \\ 0 & V \end{pmatrix}\begin{pmatrix} \dot p \\ \dot q \end{pmatrix}=\begin{pmatrix} U & 0 \\ 0 & V \end{pmatrix}\begin{pmatrix} -\nabla G\left(q\right) \\ \nabla F\left(p\right) \end{pmatrix}=\begin{pmatrix} U & 0 \\ 0 & V \end{pmatrix}\begin{pmatrix} V & 0 \\ 0 & U \end{pmatrix}\begin{pmatrix} -\nabla g\left(y\right) \\ \nabla f\left(x\right) \end{pmatrix}\\
    &=\begin{pmatrix} UV & 0 \\ 0 & VU \end{pmatrix}\begin{pmatrix} -\nabla g\left(y\right) \\ \nabla f\left(x\right) \end{pmatrix}=\begin{pmatrix} A & 0 \\ 0 & A^{\top} \end{pmatrix}\begin{pmatrix} -\nabla g\left(y\right) \\ \nabla f\left(x\right) \end{pmatrix}=\Omega_{A}\nabla \mathsf{E}(\zeta(t)).
\end{align}
\end{subequations}
Hence, one can derive the dynamics for any skew-gradient flow given the dynamics of a suitable Hamiltonian flow. Furthermore, when $A$ is invertible, we can reverse the steps above in \eqref{skew} and the converse is also true. 

In summary, the discrete-time dynamics of AMD and DAMD are related by duality; furthermore, the dynamics of DAMD and symplectic Euler are the same up to a linear transformation, thereby highlighting the equivalence between these three algorithms. The following diagram visualizes the relationships:
\begin{align*}
    \fbox{$\substack{\textrm{symplectic Euler (discrete time)}\\\textrm{Hamiltonian flow (continuous time)}}$}\qquad &\substack{\xLeftarrow{A\textrm{ is invertible}} \\ \xRightarrow{\hphantom{A\textrm{ is invertible}}}}\qquad\fbox{$\substack{\textrm{alternating mirror descent (discrete time)}\\\textrm{skew-gradient flow (continuous time)}}$} 
\end{align*}
This allows one to apply a framework of analysis in the study of symplectic numerical integrators to analyze AMD. The link between symplectic numerical integrators and AMD has previously been identified in \citep{Wibisono2022}, but in this work, we develop and use the links between these fields more explicitly.

\section{The modified Hamiltonian}\label{MH_more}

Consider a separable Hamiltonian $H \colon \Z \to \R$, such that $H=F+G$ for some differentiable functions $F \colon \P \to \R$ and $G \colon \Q \to \R$. To ensure that the domains for the Hamiltonian flow and the game-theoretic setting from Section \ref{Section_Settings} are compatible, $\P$ and $\Q$ should be chosen such that $U\P=\mathfrak{A}$ and $V\Q=\mathfrak{B}$.

The symplectic Euler method \eqref{init} is a symplectic integrator which conserves a modified Hamiltonian $\widetilde{H}_{\eta}$---perturbed in order-by-order corrections in $\eta$ from $H$---provided that this perturbation series converges (c.f. \citealp[Sections IX.3.1–IX.3.2]{Hairer2006}; \citealp{Yoshida1993}). Thus, the iterations of symplectic Euler lie on level sets of $\widetilde{H}_{\eta}$, thereby introducing a geometric analysis to predict dynamical properties of the iterates (e.g., boundedness) and inform analyses of both DAMD and AMD by studying these level sets. But to understand and derive $\widetilde{H}_{\eta}$ for symplectic Euler, we first need to define Poisson brackets.

\paragraph{Poisson brackets.} Let $C^{n}\left(\Z\right)$ be the space of $n$-times continuously differentiable real-valued functions over $\Z$ for any $n\in\N_{0}$ (see Section \ref{note} for more details). For any $n\in\N$, we denote the {\em Poisson bracket} 
$$\left\{\cdot,\cdot\right\}:C^{n}\left(\Z\right)\times C^{n}\left(\Z\right)\rightarrow C^{n-1}\left(\Z\right)$$ 
of two functions $\varphi,\psi\in C^{n}\left(\Z\right)$ by, for all $(p,q) \in \Z$:
\begin{align*}
    \{\varphi,\psi\}(p,q) = -\nabla_{p}\varphi(p,q)^\top\nabla_{q}\psi(p,q) + \nabla_{q}\varphi(p,q)^\top \nabla_{p}\psi(p,q).
\end{align*}
For ease of notation, we suppress the dependence on the argument $(p,q)$, and write the above as:
\begin{align*}
    \{\varphi,\psi\} = -\nabla_{p}\varphi^\top\nabla_{q}\psi+\nabla_{q}\varphi^\top\nabla_{p}\psi.
\end{align*}
Poisson brackets can be iterated, and the evaluation starts from the innermost bracket. For example, assuming that $F\in C^{2}\left(\P\right)$ and $G\in C^{2}\left(\Q\right)$, we have $\{\{F, G\}, G\}=\nabla_{q}G^{\top} \nabla_{p}^{2}F\nabla_{q}G$.

Some later definitions involve {\bf Iterated Poisson Brackets (IPBs)}, which we denote using the following compact notation for any $N \in \N$ and $r_{1},\dots,r_{N},s_{1},\dots,s_{N} \in \N_0$:
\begin{gather}
    \left\{F^{s_{1}}G^{r_{1}}\cdots F^{s_{N}}G^{r_{N}}\right\}\label{por} =\{\cdots\underbrace{\left\{\left\{F,\dots,F\right\},F\right\}}_{s_{1}},\underbrace{\left.\left.\left.G\right\},\cdots G\right\},G\right\}}_{r_{1}},\cdots\underbrace{\left.\left.\left.F\right\},\cdots F\right\},F\right\}}_{s_{N}},\underbrace{\left.\left.\left.\left.G\right\},\cdots\right\}G\right\},G\right\}}_{r_{N}}\nonumber
\end{gather}
For notational ease, we take the convention that single-function brackets return the function itself---i.e. $\{F\}\coloneqq F$ and $\{G\}\coloneqq G$.

$N$ refers to the \textit{rank} of the IPB $\left\{F^{s_{1}}G^{r_{1}}\cdots F^{s_{N}}G^{r_{N}}\right\}$.
For a review on Poisson brackets, including how to exponentiate Poisson brackets as operators (which is important for deriving the modified Hamiltonian $\widetilde{H}_{\eta}$), please refer to Appendix \ref{Liealgebra}.

\paragraph{Formal definition.} The dynamics of the symplectic Euler algorithm \eqref{init} can be interpolated by a continuous-time Hamiltonian flow according to an infinite series called a \textbf{modified Hamiltonian} \citep{Field2003, Alsallami2018} or \textbf{shadow Hamiltonian} \citep{Engle2005}, $\widetilde{H}_{\eta}$, when that series converges. By using the Dynkin form of the Baker-Campbell-Hausdorff (BCH) formula (see, e.g., Section III.4 in \cite{Hairer2006}), $\widetilde{H}_{\eta}$ can be expressed as the following formal series in $\eta$ and IPBs of $F$ and $G$: 

\begin{definition}
    The \textit{modified Hamiltonian} $\widetilde{H}_{\eta}$ associated with the symplectic Euler method \eqref{init} is defined by the following iterated summation:
    \begin{gather}
        \widetilde{H}_{\eta}\left(p,q\right)=\sum_{n=1}^{\infty}{\frac{\left(-1\right)^{n-1}}{n}{\sum_{r_{1}+s_{1}>0 \cdots r_{n}+s_{n}>0}{\frac{\eta^{r_{1}+\dots+r_{n}+s_{1}+\dots+s_{n}-1}\left\{G^{r_{1}}F^{s_{1}}G^{r_{2}}F^{s_{2}}\cdots G^{r_{n}}F^{s_{n}}\right\}(p,q)}{\left(r_{1}+\dots+r_{n}+s_{1}+\dots+s_{n}\right)\prod_{i=1}^{n}{r_{i}!s_{i}!}}}}}.\label{dynkin}
    \end{gather}
\end{definition}

See Appendix \ref{cons_section} for the derivation of the series \eqref{dynkin}. Here, ``formal'' means that the convergence of \eqref{dynkin}, which we call the $\widetilde{H}_{\eta}$ series or Dynkin series, is not guaranteed. Due to its cumbersome structure, it is difficult to sum the series \eqref{dynkin} directly and derive a closed-form. To simplify computations, especially when trying to derive a closed-form for $\widetilde{H}_{\eta}$ in Section \ref{Quadratic}, we propose equivalent integral formulae for $\widetilde{H}_{\eta}$ stated as Corollary \ref{Corollary_SymplecticEuler_ConservedIntegral} and Theorem \ref{BCH_integral} in Appendix \ref{IntegralThm1Proof}.

We remind the reader that the series \eqref{dynkin} is the particular form of the modified Hamiltonian $\widetilde{H}_{\eta}$ for symplectic Euler. The modified Hamiltonian $\widetilde{H}_{\eta}$ will have a different form for other symplectic numerical integrators. For more details on how to derive $\widetilde{H}_{\eta}$ for other symplectic numerical integrators and examples of what $\widetilde{H}_{\eta}$ looks like in those cases (e.g., for leapfrog integrator), see \citep{Yoshida1993, Skeel} or \cite[Chapters IX and XV]{Hairer2006}.

\paragraph{Convergence of the modified Hamiltonian.} Appendix \ref{Section_AbsCvg}, whose content is omitted from the main body of the paper, presents preliminary information and results on how to determine when the modified Hamiltonian, expressed using the series \eqref{dynkin}, converges for a given $F$ and $G$. Even then, there are only a few known choices of $F$ and $G$ such that $\widetilde{H}_{\eta}$ converges. For a detailed review of these cases, please refer to Appendix~\ref{master_list}.

\subsection{The interpolating Hamiltonian flow}\label{interpolant}

If the series~\eqref{dynkin} defining $\tilde H_\eta$ converges, then the iterations of symplectic Euler \eqref{init} conserve $\widetilde{H}_{\eta}$. i.e., $\widetilde{H}_{\eta}$ is constant along the iterations of the algorithm \eqref{init}. This follows from how the Hamiltonian flow \eqref{Eq:HFComp} generated by $\widetilde{H}_{\eta}$, whose ODE is known as the ``modified equation'' \citep{Skeel,Hairer2006}, interpolates the iterates of the symplectic Euler \eqref{init}. These properties can be summarized in the theorem below:

\begin{theorem}\label{conserved_Hamiltonian}
    Let $F\in C^{\infty}\left(\P\right)$, $G\in C^{\infty}\left(\Q\right)$, and $k\in\N_{0}$. {\bf Assume} that the modified Hamiltonian~\eqref{dynkin} converges when $0 < \eta < \eta_{\max}$ for some $\eta_{\max}$ possibly a function of $\left(p_{k},q_{k}\right)$, $F$, and $G$. Let $(p_{k+1}, q_{k+1})$ be one iteration of the symplectic Euler method~\eqref{init} from $(p_k, q_k)$. Then, for all $\eta\in\left(0,\eta_{\max}\right)$, if we solve the initial-value problem
    \begin{align}
    \dot{\widetilde{z}}(t) &=\Omega\nabla\widetilde{H}_{\eta}\left(\widetilde{z}\left(t\right)\right)\nonumber
\end{align}
for $\widetilde{z}(t)\in\Z$ with initial conditions $\widetilde{z}\left(0\right)=\left(p_{k},q_{k}\right)$, then $\widetilde{z}\left(\eta\right)=\left(p_{k+1},q_{k+1}\right)$. In particular, $\widetilde{H}_{\eta}$ is conserved by one step of the symplectic Euler algorithm \eqref{init} with stepsize $\eta$ applied to $\left(p_{k},q_{k}\right)$:
    \begin{gather}
        \widetilde{H}_{\eta}\left(p_{k+1},q_{k+1}\right)=\widetilde{H}_{\eta}\left(p_{k},q_{k}\right).\nonumber
    \end{gather}
\end{theorem}

For the sake of completeness, we provide a new proof of Theorem \ref{conserved_Hamiltonian} in Appendix~\ref{cons_section} after going through an informal proof sketch modeled after the expositions in \citep{Field2003,Alsallami2018}. 
However, this is a known classical result with rigorous proofs found in several texts, e.g.,\ \cite[Chapters III and IX]{Hairer2006}. 

It might seem as if Theorem \ref{conserved_Hamiltonian} implies that $\widetilde{H}_{\eta}(p_k, q_k)$ is conserved for any number of iterations $k$ of symplectic Euler starting from $(p_0, q_0)$ for a fixed stepsize $\eta>0$. However, the maximum stepsize $\eta_{\max}$ in Theorem \ref{conserved_Hamiltonian} is dependent on whether or not the series \eqref{dynkin} for $\widetilde{H}_{\eta}$ converges when evaluated at $(p_k, q_k)$ for all $\eta$ up to $\eta_{\max}$ at each $k$. Thus, $\eta_{\max}$ might approach $0$ as $k$ grows. We can preclude this by assuming that $\widetilde{H}_{\eta}$ converges everywhere, as formalized by the following corollary to Theorem \ref{conserved_Hamiltonian}:

\begin{corollary}\label{corollary}
     Let $F\in C^{\infty}\left(\P\right)$ and $G\in C^{\infty}\left(\Q\right)$. {\bf Assume} that the modified Hamiltonian~\eqref{dynkin} converges pointwise on some subset $D$ of $\R^{d}\times\R^{d}$ when $0 < \eta < \eta_{\max}$ for $\eta_{\max}$ possibly a function of $D$, $F$, and $G$. If $\left(p_{k},q_{k}\right)\in D$ for all $k\in\N_{0}$, then $\widetilde{H}_{\eta}\left(p_{k},q_{k}\right)=\widetilde{H}_{\eta}\left(p_{0},q_{0}\right)$ for all $k\in\N_{0}$. 
\end{corollary}

\begin{proof}
    If there exists an $\eta_{\max}>0$ such that $\widetilde{H}_{\eta}$ converges for all $\eta\in \left(0,\eta_{\max}\right]$ on a large enough domain that contains all iterations, then we can apply Theorem \ref{conserved_Hamiltonian} with the same $\eta_{\max}$ to every possible iterate of symplectic Euler. 
\end{proof}

In particular, if we can establish that $\widetilde{H}_{\eta}$ converges pointwise on all of $D=\R^{d}\times\R^{d}$ for some $\eta>0$, then we can ensure that $\widetilde{H}_{\eta}\left(p_{k},q_{k}\right)=\widetilde{H}_{\eta}\left(p_{0},q_{0}\right)$ for all $k\in\N_{0}$.

The convergence and existence of $\widetilde{H}_{\eta}$ is closely related to Kolmogorov–Arnold–Moser (KAM) theory, which in superficial terms describes the persistence of quasiperiodic motion under small perturbations and when or how such motion persists---for further reading, see, e.g., Chapter 6 of \cite{Arnold2006}, \cite{Lazutkin1993}, or Chapters IX and X from \cite{Hairer2006}. To preserve the symplectic structure of a Hamiltonian system, and in particular, the level sets of some Hamiltonian, these perturbations must satisfy certain conditions related to their size and functional properties (see, e.g., Chapter~3 in Part~I of \cite{Lazutkin1993}, 
Chapter~6 of \cite{Jose1998}, and Chapter~IX of \cite{Hairer2006}). In particular, Corollary \ref{corollary} implies that, since the Hamiltonian flow generated by the modified Hamiltonian $\widetilde{H}_{\eta}$ interpolates the iterates of symplectic Euler, we can think of this modified Hamiltonian flow (and therefore the iterates of symplectic Euler) as being a perturbation of the original Hamiltonian flow generated by the unperturbed $H$. 

\section{Truncations of the modified Hamiltonian}\label{Section_TMC_Regret}

When formally rearranged as a power series in $\eta$, the first few terms of the series $\tilde H_\eta$~\eqref{dynkin} are as follows: 
\begin{align*} 
    \widetilde{H}_{\eta}\left(p,q\right) &= F\left(p\right)+G\left(q\right) +\frac{\eta}{2}\left\{F,G\right\} \nonumber\\ 
    &\quad+\frac{\eta^{2}}{12}\left(\left\{\left\{F,G\right\},G\right\}+\left\{\left\{G,F\right\},F\right\}\right)
    -\frac{\eta^{3}}{24}\left\{\left\{\left\{F,G\right\},G\right\},F\right\}+\mathcal{O}\left(\eta^{4}\right),\nonumber
\end{align*}
as also stated in, e.g.,~\citep{Jacobson1979,Hall2015,Casas,Yoshida1993}.

Hence, perturbation theory motivates the following definition of each order-by-order correction to $\widetilde{H}$ and the sum of all corrections up to a given order in $\eta$.

\begin{definition}\label{def1}
    For $n\in\N_{0}$, the $n^{\text{th}}$-order correction $H_{n}$ to the modified Hamiltonian is defined to be
    \begin{gather}
        H_{n}\coloneqq \frac{1}{n+1}\sum_{m=1}^{n+1}\frac{(-1)^{m-1}}{m}\sum_{\substack{ r_{1}+s_{1}>0  \dots r_{m}+s_{m}>0 \\ \sum_{i=1}^m (r_i+s_i)=n+1}}{\frac{\left\{G^{r_{1}}F^{s_{1}}G^{r_{2}}F^{s_{2}}\cdots G^{r_{m}}F^{s_{m}}\right\} (p,q)}{\prod_{i=1}^{m}{r_{i}!s_{i}!}}}.\label{correction}
    \end{gather}
    For $N \in \N_{0}$, the $N^{\text{th}}$-order truncated modified Hamiltonian is defined as follows:
    $$\widetilde{H}_{\eta}^{(N)}(p, q)\coloneqq \sum_{j=0}^{N}{\eta^{j}H_{j}}$$
    where $H_0 = H$ is the original Hamiltonian.
\end{definition}

We remark that $\widetilde{H}_{\eta}(p, q)=\lim_{N \to \infty} \widetilde{H}_{\eta}^{(N)}(p, q)$ only when $\widetilde{H}_{\eta}(p, q)$ converges absolutely. For a review on modes of convergence (e.g., absolute convergence), we refer the reader to Section \ref{odes}. 

Heuristically, one would expect that $\widetilde{H}_{\eta}^{(N)}$ is approximately conserved by the iterations of symplectic Euler \eqref{init}, and hence also by DAMD \eqref{DAMD} when transforming back to $(x,y)$-coordinates from $(p,q)$-coordinates. In particular, one would expect that the conservation in accuracy for $\widetilde{H}_{\eta}^{(N)}$ increases with $N$ (as the degree of approximation increases) and decreases with $\eta$ (as the discretization in time becomes less coarse). 
This is still a conjecture, which we state below on how the accuracy of conservation of $\widetilde{H}_{\eta}^{(N)}$ depends on problem parameters.

\begin{conjecture}\label{Conjecture_MH_General}
    If $\P,\Q\subseteq\R^{d}$ are both closed and convex, and $F:\P\rightarrow\R$ and $G:\Q\rightarrow\R$ are both $L$-smooth of orders $1,\dots,N+2$, then for symplectic Euler \eqref{init} with any stepsize $\eta>0$, 
    \begin{gather}
        |\widetilde{H}_{\eta}^{(N)}(z_{k})-\widetilde{H}_{\eta}^{(N)}(z_0)|\leq k\Phi\left(N\right)L^{N+3}\eta^{N+2}\label{conj}
    \end{gather}
    holds for each $N\in\N_{0}$ and some monotonic increasing, bounded function $\Phi:\N_{0}\rightarrow\mathbb{Q}$. 
\end{conjecture}

\noindent \textit{Proof sketch.} The main idea of the framework used to prove Conjecture \ref{Conjecture_MH_General} is as follows. As shown in detail in Appendix \ref{Appendix_Thm5}, we notice that $z_{k+1}-z_k = (-\eta \nabla G(q_k), \eta \nabla F(p_{k+1}))$. Thus, we can express the value of $\widetilde{H}_{\eta}^{(N)}$ at $z_{k+1}$ by expanding each term as a Taylor series about $z_k$. By choosing the order of expansion judiciously for each term in $\widetilde{H}_{\eta}^{(N)}$, we show that the sum of all lower-order terms vanish, leaving only $\mathcal{O}\left(\eta^{N+2}\right)$ terms.  \hfill$\blacksquare$

\medskip

We have proven Conjecture \ref{Conjecture_MH_General} in the following cases: 
\begin{enumerate}
    \item[(a)] For $N \in \{0,1,2,3\}$ in any $d\in\N$ dimensions; and
    \item[(b)] For $N \in \{0,1,\dots,10\}$ in $d=1$ dimension.
\end{enumerate}
In Appendix \ref{Appendix_Thm5}, we use the framework described in the proof sketch to systematically show that all lower-order terms up to $\mathcal{O}\left(\eta^{N+2}\right)$ (exclusive) vanish for $N \in \{0,1,2,3\}$ to prove (a). We also describe our implementation of this framework in SymPy \citep{SymPy} to show (b). This computer implementation uses the recursive form of the BCH formula---see Section 2.15 in \cite{Varadarajan1984}, \cite{Casas}, and also \eqref{recursive1}---to compute each $n$th-order correction to $\widetilde{H}_{\eta}$ and applies symbolic differentiation to show that the necessary higher-order terms vanish. Finally, in Appendix \ref{Phi}, we describe what we know thus far about the coefficient function $\Phi$.

Conjecture \ref{Conjecture_MH_General} implies a weaker form of~\cite[Theorem~4.4]{Wibisono2022} for $N=1$, and it generalizes their theorem to at least $N=3$. We note that~\cite[Theorem~4.4]{Wibisono2022} has $1/12$ in place of $\Phi\left(1\right)=3$, which they achieve by relating the Bregman divergence of $H$ to $\widetilde{H}_{\eta}^{(1)}$; we could not replicate as tight of a constant for $N>1$. 
But unlike how \cite{Wibisono2022} distinguishes the smoothness of the first- and third-derivatives of $H=F+G$, we could but choose not to express the dependence of our conjecture \eqref{conj} on the smoothness of separate higher-order derivatives of $F$ and $G$. This is done for brevity---the interested reader may refer to Appendix \ref{Appendix_Thm5} to see the explicit relationships for each separate derivative. 
\cite{Wibisono2022} also has a factor of $\sigma_{\max}\left(A\right)^{N+2}$, where $\sigma_{\max}\left(A\right)$ is the largest singular value of $A$. We do not have this in Conjecture \ref{Conjecture_MH_General} because we are working with $F$ and $G$ as functions of $(p,q)$ rather than $f$ and $g$ as functions of $(x,y)$. By generalizing how a factor of $\sigma_{\max}\left(A\right)^{3}$ shows up in~\cite[Theorem~4.4]{Wibisono2022}, we see how Conjecture \ref{Conjecture_MH_General} would change for the trajectories of DAMD instead of symplectic Euler: we would incur a factor of $\sigma_{\max}\left(A\right)^{N+2}$ on the righthand side of \eqref{conj}, $F$ and $G$ would be replaced by $f$ and $g$, respectively, and the domains $\P$ and $\Q$ would be replaced by $\frA$ and $\frB$, respectively.

 \begin{figure}[H]
\includegraphics[width=\textwidth]{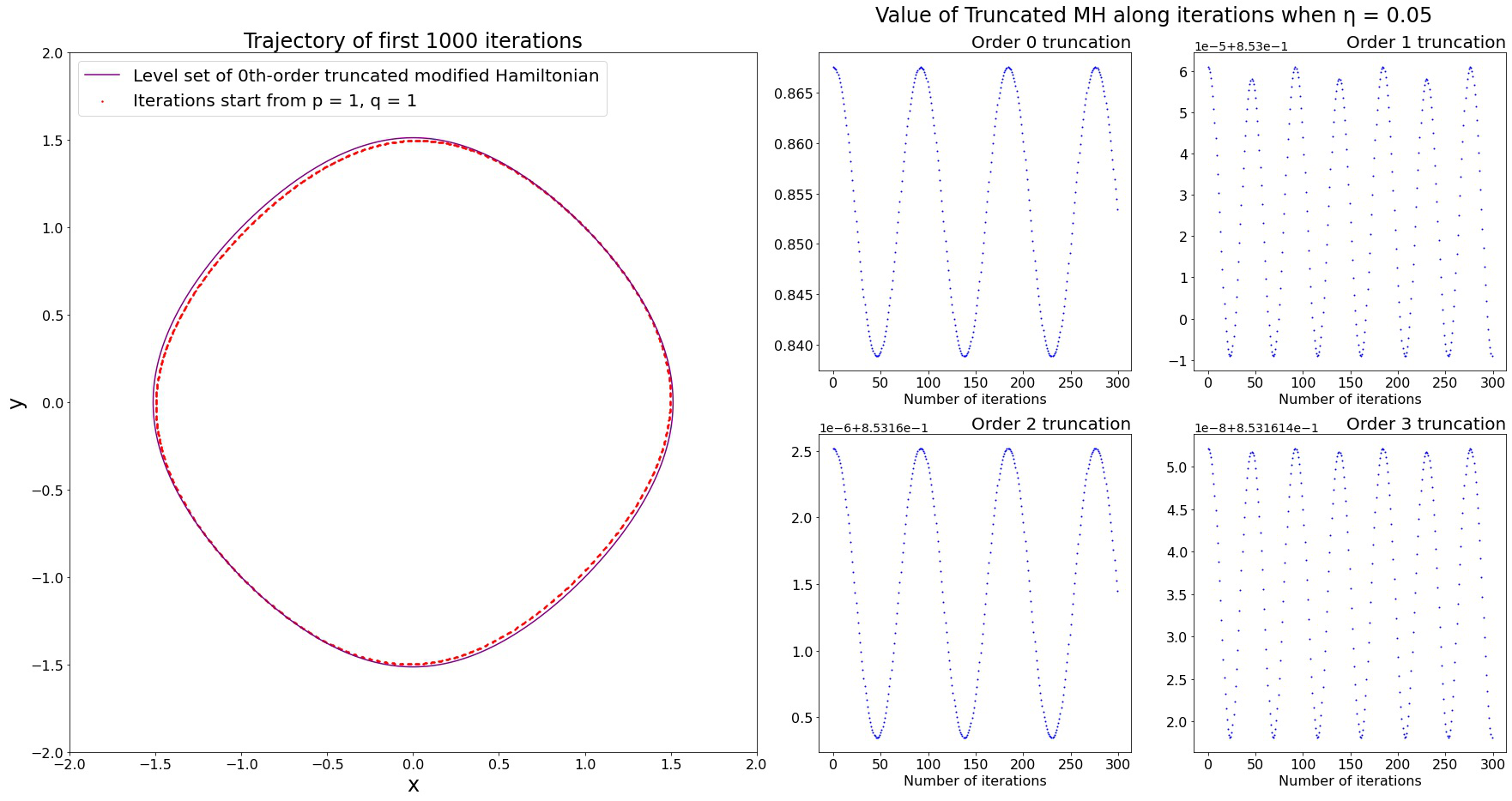}
\caption{\footnotesize A simulation of the first 1000 iterations of the algorithm in the case where $F=\log(\cosh(p)), G=\log(\cosh(q))$ and $\eta=0.05$ starting from $p_0=1, q_0=1$. The larger picture on the left contains the trajectory of the iterations, and the smaller four pictures on the right contain the value of the modified Hamiltonian along the iterations. Note that the small label 1e-6+8.5316e-1 (or analogous notation for the other subplots) is Matplotlib’s offset notation: every $y$-value on that axis actually represents $10^{-6}y+0.85316$. In other words, the modified Hamiltonian oscillates a few micro-units around $0.85316$, and that level of oscillation decreases as the truncation order increases.}
\label{fig:1}
\end{figure}

The example in Figure \ref{fig:1} does not exhibit a linear growth rate in the error \eqref{conj} with the number of iterations, $k$---rather, it demonstrates an error which remains bounded for all $k$. This suggests that further work may be necessary to refine the dependence on $k$ in Conjecture \ref{Conjecture_MH_General}, perhaps upon additional assumptions for $F$ and $G$, e.g., convexity or closed level sets for $H=F+G$ as in Figure \ref{fig:1}. Prior results \cite[Theorem~8.1, Chapter~IX; Theorem~3.1, Chapter~X]{Hairer2006} feature bounds like \eqref{conj} with no linear growth rate in $k$ but require either knowing that the trajectories of iterations of symplectic Euler \eqref{init} are bounded or that certain resonance conditions are met, both of which we could not guarantee a priori. Our interest is in showing long-term energy conservation without assuming boundedness of trajectories, since if we already knew that the trajectories of the algorithm were bounded, then the following lemma \cite[Theorem~4.5]{Wibisono2022} implies that the total regret for AMD is bounded, as formalized later in Corollary \ref{bounded}:

\begin{lemma}\label{Brregman}
Suppose that $\left(a_{k},b_{k}\right)$ evolve according to the AMD algorithm \eqref{losses} with stepsize $\eta>0$. Let $\zeta_{k}\coloneqq \left(x_{k},y_{k}\right)=\left(\nabla\alpha\left(a_{k}\right),\nabla\beta\left(b_{k}\right)\right)$ be the dual variables of $\left(a_{k},b_{k}\right)$ for $k=0,1,\dots,K$, as defined in Section \ref{primaldual}. Then, we have the following formula for the cumulative regret $R_{1,K}(a)+R_{2,K}(b)$ after $K$ iterations of the AMD algorithm for any $(a,b)\in\A\times\B$:
    \begin{gather}
    R_{1,K}(a)+R_{2,K}(b)=\frac{1}{\eta}\left(D_{H}\left(\zeta_{0},\zeta\right)-D_{H}\left(\zeta_{K},\zeta\right)+\widetilde{H}_{\eta}^{(1)}\left(\zeta_{K}\right)-\widetilde{H}_{\eta}^{(1)}\left(\zeta_{0}\right)\right),\label{regret}
\end{gather}
where and $D_{H}$ is the Bregman divergence of $H=\alpha^{*}+\beta^{*}$ (see Section \ref{Section_Settings} for the definition). Moreover, since $H$ is convex, we can further bound $D_{H}\left(\zeta_{K},\zeta\right)\geq 0$, thereby yielding the following upper bound on \eqref{regret}:
\begin{gather}
    R_{1,K}(a)+R_{2,K}(b)\leq\frac{1}{\eta}\left(D_{H}\left(\zeta_{0},\zeta\right)+\widetilde{H}_{\eta}^{(1)}\left(\zeta_{K}\right)-\widetilde{H}_{\eta}^{(1)}\left(\zeta_{0}\right)\right).\label{regretbound}
\end{gather}
\end{lemma}

\begin{proof}
    See \cite[][Appendix B.3.4]{Wibisono2022} for the proof.
\end{proof}
As shown in Section \ref{Regret_Analysis}, Lemma \ref{Brregman} can also be used to relate $H$ to regret bounds for AMD.

\section{Example of a closed-form modified Hamiltonian: quadratic case}\label{Quadratic}
We now compute the modified Hamiltonian $\widetilde{H}_{\eta}$ in closed form for the quadratic case, i.e., when
$$F(p) = p^{\top}Bp, \qquad G(q) = q^{\top}Cq$$
for some $B, C \in \R^{d \times d}$.
To the best of the authors' knowledge, this is among the first attempts (\cite{Field2003} is another) to compute  $\widetilde{H}_{\eta}$ by summing the Dynkin series \eqref{dynkin} explicitly.

The quadratic case is closely related to when AMD or AGD is applied to unconstrained bilinear zero-sum games, as detailed in \cite{pmlr-v125-bailey20a}. In particular, when $\A=\B=\R^{d}$ and $\alpha$ and $\beta$ are positive-definite quadratic forms, their convex conjugates $\alpha^{*}$ and $\beta^{*}$ are also positive-definite quadratic forms and would thereby fall under this case.

For \cite{pmlr-v125-bailey20a}, the iterations of DAMD follow the updates \eqref{init} with $f=\frac{1}{2}x^{\top}x$ and $g=\frac{1}{2}y^{\top}y$. The authors show that these updates conserve the second-order quantity $x^{\top}x+y^{\top}y-2\eta x^{\top}Ay$ in the original dual variables $(x,y)$. Equivalently, in the $(p,q)$-coordinates, we have $F=\frac{1}{2}p^{\top}U^{2}p$ and $G=\frac{1}{2}q^{\top}V^{2}q$, and therefore, the conserved quantity is
$$\mathsf{S}\left(p,q\right)\coloneqq p^{\top}U^{2}p+q^{\top}V^{2}q-2\eta p^{\top}UAVq,$$ where $A = UV$ is the symmetric decomposition of $A$. Hence,
this is an instance of the quadratic case above with $B=U^{2}$ and $C=V^{2}$ (cf. \cite[][Theorem~4]{pmlr-v125-bailey20a} for details and a derivation). We demonstrate that $\widetilde{H}_{\eta}$ and $\mathsf{S}\left(p,q\right)$ recover the same conserved quantity (up to a scalar factor which is a function of $\eta$) when the dimension $d$ of the strategy space is equal to $1$, but recover a different one when $d>1$.

We have the following theorem regarding the closed form of $\widetilde{H}_{\eta}$ in the quadratic case.

\begin{theorem}\label{Thm_Quadratic_MH_multivar}
    Let $d\in\N$, $p, q \in \R^{d}, B, C \in \R^{d \times d}$. Let $F(p)=p^\top Bp, G(q)=q^\top Cq$, where we assume without loss of generality that $B$ and $C$ are symmetric (if not, then replace $B$ and $C$ with their symmetrizations $\frac{B+B^{\top}}{2}$ and $\frac{C+C^{\top}}{2}$, respectively).
    Furthermore, assume that $BC$ is diagonalizable and denote its diagonalization as $BC=Q^{-1}\Lambda Q$. 

    Define $T(\eta, \lambda)$ as the following:
    \begin{gather}
        T(\eta, \lambda) = \begin{cases}
            \frac{\arcsin (\sqrt{\lambda \eta^2})}{\sqrt{\lambda\eta^2(1-\lambda\eta^2)}} & \lambda > 0,\\
    \frac{\arcsinh (\sqrt{-\lambda \eta^2})}{\sqrt{-\lambda\eta^2(1-\lambda\eta^2)}}& \lambda < 0.
        \end{cases}\label{Tfunc}
    \end{gather}
When $\sigma_{\max}(BC)\eta^2<1$, the series for the modified Hamiltonian $\widetilde{H}_{\eta}$ conserved by the iterations of symplectic Euler \eqref{init} converges absolutely to
   \begin{align}
        \widetilde{H}_{\eta}(p, q)&=pQ^{-1}T(\eta, \Lambda) QBp+qCQ^{-1}T(\eta, \Lambda) Qq-2pQ^{-1} |\Lambda|T(\eta, \Lambda) Qq,\label{multi}
    \end{align}
where $T(\eta, \Lambda)$ signifies applying the function $T(\eta, \cdot)$ to the entries of $\Lambda$ element-wise, and $|\Lambda|$ denotes the element-wise absolute value of the matrix $\Lambda$.

In particular, when $d = 1$, the closed-form expression \eqref{multi} simplifies to
\begin{align}
    \widetilde{H}_{\eta}(p, q) = T(\eta, BC)(Bp^2+Cq^2-2BCpq)=T(\eta, BC)\mathsf{S}\left(p,q\right),\label{1D}
\end{align}
noting that $B$ and $C$ are scalars when $d=1$.
\end{theorem}

We provide the proof of Theorem~\ref{Thm_Quadratic_MH_multivar} in Appendix \ref{bigcomputation} for the univariate case, and in Appendix~\ref{E.4} for the multivariate case. The proofs are computational and use the integral form of $\widetilde{H}_{\eta}$ (Theorem \ref{BCH_integral}).

When $B$ and $C$ are positive semidefinite matrices, their product $BC$ is diagonalizable and has nonnegative eigenvalues \citep{HONG1991}. Hence, in particular, the positive-definite quadratic form case mentioned above satisfies the conditions for Theorem~\ref{Thm_Quadratic_MH_multivar}.

For a plot of $T(\eta, BC)$ vs. $\eta$ for varying values of $B$ and $C$, see Figure \ref{fig:2} in Appendix \ref{bigcomputation}. 
We can compute the Taylor series expansion for $T(\eta, \lambda)$ centered at $\eta=0$ as: 
\begin{align*}
    T(\eta, \lambda)=1+\frac{2\lambda\eta^2}{3}+\frac{8\lambda^{2}\eta^4}{15}+ O(\eta^6).
\end{align*}

As shown by \eqref{1D}, when $d=1$, we also note that the modified Hamiltonian $\tilde H_\eta$ from Theorem \ref{Thm_Quadratic_MH_multivar} is equivalent with (as a scalar multiple of) the known conserved quantity, $\mathsf{S}\left(p,q\right)=p^{\top}U^{2}p+q^{\top}V^{2}q-2\eta p^{\top}UAVq$, introduced in \cite{pmlr-v125-bailey20a}. However, in general, $\widetilde{H}_{\eta}$ from Theorem \ref{Thm_Quadratic_MH_multivar} is functionally independent from $Q$, which is possible since higher-dimensional dynamical systems can have multiple independent conserved quantities. In particular, an autonomous dynamical system in $d$ dimensions can have at most $d-1$ functionally independent constants of motion (since one dimension corresponds to the flow itself), while a Hamiltonian system of dimension $d=2n$ can have at most $n$ independent integrals in involution---see, e.g., Sections~3.2 and 6.2 in 
\cite{Jose1998}, or Section~3.2 and Chapter~5 in \cite{Arnold2006}.

We further observe that, for the quadratic case, although $\widetilde{H}_{\eta}$ is different from $Q$, the range of $\eta$ for which the modified Hamiltonian converges lies within the same range as is necessary for the algorithm to have bounded trajectories---i.e., $\sigma_{\max}\left(B\right)\sigma_{\max}\left(C\right)\eta^2<1$---as can be shown by studying the level sets of $Q$ \citep[Theorem~5]{pmlr-v125-bailey20a}. This is because $\sigma_{\max}\left(BC\right)\leq\sigma_{\max}\left(B\right)\sigma_{\max}\left(C\right)$, in which case $\sigma_{\max}\left(B\right)\sigma_{\max}\left(C\right)\eta^2<1$ implies $\sigma_{\max}\left(BC\right)\eta^2<1$. Thus, the sufficient condition for $\widetilde{H}_{\eta}$ to converge under Theorem \ref{Thm_Quadratic_MH_multivar} is weaker than that for symplectic Euler to have bounded trajectories, and so $\widetilde{H}_{\eta}$ could still converge when the trajectories are unbounded.

In Appendix \ref{E_flow}, we also explicitly check the implications from Section \ref{interpolant}---and in particular, Theorem \ref{conserved_Hamiltonian}---in the quadratic case for $d=1$. Namely, we verify that if $(\tilde{p}(t), \tilde{q}(t))$ follows the Hamiltonian flow generated by the modified Hamiltonian \eqref{1D} with initial conditions $\left( \tilde{p}(0), \tilde{q}(0) \right)=\left(p_{k}, q_{k}\right)$, then the value of the flow at $t=\eta$, $\left( \tilde{p}(\eta), \tilde{q}(\eta) \right)$, coincides with $\left(p_{k+1}, q_{k+1}\right)$. 

\section{Improved regret analysis of Alternating Mirror Descent}
\label{Regret_Analysis}

An application of Conjecture \ref{Conjecture_MH_General} lies in studying the algorithmic performance of AMD in the context of zero-sum games introduced in Section \ref{Section_Settings}. Theorem 3.2 from \cite{Wibisono2022} derives a bound on the total regret after $K$ iterations, $R_{K}$, in terms of the smoothness of $\alpha$ and $\beta$ and the sizes of $\A$ and $\B$, and from there concludes that choosing an $\eta=\Theta\left(K^{-1/3}\right)$ stepsize leads to $R_{K}=\mathcal{O}\left(K^{1/3}\right)$ and a duality gap of the average iterates $\overline{\dg}_{K}=\mathcal{O}\left(K^{-2/3}\right)$. Using the truncated modified Hamiltonians $\widetilde{H}_{\eta}^{(N)}$ and Conjecture \ref{Conjecture_MH_General}, we generalize Theorem 3.2 from \cite{Wibisono2022} upon assuming higher-order smoothness of $\alpha$ and $\beta$. 

Firstly, Lemma \ref{Brregman} allows us to prove Theorem \ref{regretting}:

\begin{theorem}\label{regretting}
    Let $\mathcal{A},\mathcal{B}$ be closed, convex subsets of $\R^{d}$. Suppose we start from $(a_{0},b_{0})\in\A\times\B$ such that $D_{\alpha}\left(a_{k},a_{0}\right)$ and $D_{\beta}\left(b_{k},b_{0}\right)$ are uniformly bounded across $k$---i.e., there exists $M>0$ with $D_{\alpha}\left(a_{k},a_{0}\right),D_{\beta}\left(b_{k},b_{0}\right)\leq M/2$ for all $k\in\N$. Furthermore, suppose that $\alpha^{*}$ and $\beta^{*}$ are $L$-smooth of orders $1,\dots,N+2$ for some $N\in\N$ over the closed convex hulls of the dual spaces $\frA$ and $\frB$, respectively.
    
    Then, for all $N\in\N$ such that Conjecture \ref{Conjecture_MH_General} is true, there exists some universal constant $C>0$ such that if both players follow AMD with stepsize $\eta=\mathcal{O}\left(1\right)$, then the total regret $R_{K}$ at iteration $K$ is bounded by
    \begin{gather}
        R_{K}\leq \frac{M}{\eta}+2\eta\sup_{z\in\Z}\sum_{j=0}^{N-2}{\eta^{j}\left|H_{j+2}(z)\right|}+CKL^{N+3}\sigma_{\max}\left(A\right)^{N+2}\eta^{N+1},\label{algo2}
    \end{gather}
    where the middle summation in \eqref{algo2} vanishes when $N=1$, $\P,\Q\subseteq\R^{d}$ are chosen such that $U\P=\mathfrak{A}$ and $V\Q=\mathfrak{B}$, and the $H_{j}$'s are the $j$th-order corrections to the modified Hamiltonian $\widetilde{H}_{\eta}$ generated by $F\left(\cdot\right)=\alpha^{*}\left(U\left(\cdot\right)\right)$ and $G\left(\cdot\right)=\beta^{*}\left(V\left(\cdot\right)\right)$ (see Section \ref{MH_more}). 
    
    In particular, given a horizon $K$, if we set the stepsize $\eta=\Theta\left(K^{-1/(N+2)}\right)$, then $R_{K}$ is $\mathcal{O}\left(K^{1/(N+2)}\right)$. Finally, provided that $\lnorm a_k\rnorm$ and $\lnorm b_k\rnorm$ are uniformly bounded for all $k\in\N$, the duality gap of the average iterates $\overline{\dg}_{K}$ decays like $\mathcal{O}\left(K^{-(N + 1)/(N + 2)}\right)$. 
\end{theorem}

\begin{proof}
    By the considerations in Section \ref{primaldual}, the dynamics of symplectic Euler imply those for DAMD for any $A\in\R^{d\times d}$. Furthermore, any bounds on the $k$th derivatives for $F$ and $G$ are the same as those for $f$ and $g$, respectively, except at worst scaled by $\sigma_{\max}\left(A\right)^{k}$. Hence, without loss of generality, we can assume that $A=I$---in which case $p=x$, $q=y$, $f=F$, and $g=G$---and rescale $L$ by $\sigma_{\max}\left(A\right)$ wherever it appears at the end of the proof and then divide by $\sigma_{\max}\left(A\right)$ once.
    
    Take $\mathcal{A}$ and $\mathcal{B}$ to be closed, convex sets. Then, for any $\zeta,\zeta_{0}\in\mathfrak{A}\times\mathfrak{B}$ and all $k\in\N$,
\begin{gather}
    D_{H}\left(\zeta_{0},\zeta_k\right)=D_{f}\left(x_{0},x_k\right)+D_{g}\left(y_{0},y_k\right)=D_{\alpha}\left(a_k,a_{0}\right)+D_{\beta}\left(b_k,b_{0}\right)\leq \frac{M}{2}+\frac{M}{2}=M.\label{boundd}
\end{gather}
Thus, by applying \eqref{regretbound} from Lemma \ref{Brregman}, \eqref{boundd}, and the triangle inequality, we have
\begin{align}
    &\eta R_{K}\leq M+\left|\widetilde{H}_{\eta}^{(1)}\left(z_{K}\right)-\widetilde{H}_{\eta}^{(1)}\left(z_{0}\right)\right|\nonumber\\
    &\leq M+\left|\widetilde{H}_{\eta}^{(N)}\left(z_{K}\right)-\widetilde{H}_{\eta}^{(1)}\left(z_{K}\right)\right|+\left|\widetilde{H}_{\eta}^{(N)}\left(z_{K}\right)-\widetilde{H}_{\eta}^{(N)}\left(z_{0}\right)\right|+\left|\widetilde{H}_{\eta}^{(N)}\left(z_{0}\right)-\widetilde{H}_{\eta}^{(1)}\left(z_{0}\right)\right|.\label{huge}
\end{align}
Then, taking $N\geq 1$, we can use the successive terms $H_{2},\dots,H_{N}$ in the Dynkin series, order-by-order in $\eta$, to express the following:
\begin{gather}
    \widetilde{H}_{\eta}^{(N)}-\widetilde{H}_{\eta}^{(1)}=
    \begin{cases}
        0, & N=1,\\
        \eta^{2}\left(H_{2}+\eta H_{3}+\dots+\eta^{N-2}H_{N}\right), & N>1.
    \end{cases} \label{BCHterms}
\end{gather}
By applying the triangle inequality to \eqref{BCHterms}, we deduce that there exists some $B_{N,\alpha,\beta}\geq 0$ such that $\left|\widetilde{H}_{\eta}^{(N)}-\widetilde{H}_{\eta}^{(1)}\right|\leq\eta^{2}B_{N,\alpha,\beta}/2$ on all of $\Z$ for any fixed $N\geq 1$, where 
\begin{gather}
    B_{N,\alpha,\beta}=
    \begin{cases}
        0, & N=1,\\
        2\sup_{z\in\Z}\sum_{j=0}^{N-2}{\eta^{j}\left|H_{j+2}(z)\right|}, & N>1.
    \end{cases} \nonumber
\end{gather}
The bound in \eqref{huge} implies
\begin{gather}
    \eta R_{K}\leq M+\eta^{2}B_{N,\alpha,\beta}+\left|\widetilde{H}_{\eta}^{(N)}\left(z_{K}\right)-\widetilde{H}_{\eta}^{(N)}\left(z_{0}\right)\right|.\label{onemore}
\end{gather}
Hence, applying Conjecture \ref{Conjecture_MH_General} directly to \eqref{onemore}:
\begin{gather}
    \left|\widetilde{H}_{\eta}^{(N)}\left(z_{K}\right)-\widetilde{H}_{\eta}^{(N)}\left(z_{0}\right)\right|\leq CKL^{N+3}\eta^{N+2},\label{conj'}
\end{gather}
where $C>0$ is an upper bound on $\Phi$: $\Phi\left(n\right)\leq C$ for all $n\in\N_{0}$. Substituting \eqref{conj'} into \eqref{onemore}, we have derived \eqref{algo2}, i.e.,
\begin{gather}
R_{K}\leq \frac{M}{\eta}+2\eta\sup_{z\in\Z}\sum_{j=0}^{N-2}{\eta^{j}\left|H_{j+2}(z)\right|}+CKL^{N+3}\sigma_{\max}\left(A\right)^{N+2}\eta^{N+1}\label{combined}
\end{gather}
if we take the convention that ``$\sum_{j=0}^{-1}=0$.'' In particular, given a horizon $K$, if we further set $\boxed{\eta=\Theta\left(K^{-1/(N+2)}\right)}$, then \eqref{combined} gives us a total regret $\boxed{R_{K}=\mathcal{O}\left(K^{1/(N+2)}\right)}$. 

It remains to show the iteration complexity of the duality gap of the average iterates, $\overline{\dg}_{K}$. By assumption, there exists some $B>0$ such that $\lnorm a_k\rnorm,\lnorm b_k\rnorm\leq B$ for all $k\in\N$. Thus, after using Lemma \ref{dgK_and_RK}, the duality gap \eqref{dG} can be bounded as follows:
\begin{align}
    \left|\overline{\dg}_{K}\right|
    &=\left|\frac{1}{K}R_{K}-\frac{1}{2K}\left(a_{0}^{\top}Ab_{0}-a_{K}^{\top}Ab_{K}\right)\right| \nonumber\\
    &\le \frac{1}{K}\left|R_{K}\right|+\frac{1}{2K}\left|a_{0}^{\top}Ab_{0}\right|+\frac{1}{2K}\left|a_{K}^{\top}Ab_{K}\right|\nonumber\\
    &\leq \frac{1}{K}\left|R_{K}\right|+\frac{1}{2K}\sigma_{\max}\left(A\right)\lnorm a_{0}\rnorm\lnorm b_{0}\rnorm+\frac{1}{2K}\sigma_{\max}\left(A\right)\lnorm a_{K}\rnorm\lnorm b_{K}\rnorm \nonumber \\ 
    &\leq\frac{1}{K}\left|R_{K}\right|+\frac{1}{K}\sigma_{\max}\left(A\right)B^{2}.\label{dgbound}
\end{align}
In the first inequality above, we use the triangle inequality.
Furthermore, since $R_{K}=\mathcal{O}\left(K^{1/(N+2)}\right)$, it follows that $R_{K}/K=\mathcal{O}\left(K^{-(N+1)/(N+2)}\right)$. Hence, provided that $\lnorm a_k\rnorm$ and $\lnorm b_k\rnorm$ are uniformly bounded for all $k\in\N$, \eqref{dgbound} implies the following: 
$$\overline{\dg}_{K}=\mathcal{O}\left(K^{-(N+1)/(N+2)}\right)+\mathcal{O}\left(K^{-1}\right)=\mathcal{O}\left(K^{-(N+1)/(N+2)}\right).$$

Finally, we show why the series remainder $\sum_{j=0}^{N-2}{\eta^{j}\left|H_{j+2}(z)\right|}$ in Theorem \ref{regretting} is bounded; establishing this is essential to ensure that the algorithmic guarantees claimed therein are actually valid. 
Note that since $f$ and $g$ are $L$-smooth of orders $1,\dots,N+2$, this implies that the operator norms of the derivatives up to order $N+2$ of $f$ and $g$ are bounded by $L$.
We also note from Definition \ref{def1} that $H_{j}$ consists of a linear combination of IPBs whose order (i.e., the number of iterated Poisson brackets) is $j$. 
Each of these IPBs of order $j$ consists of an inner product of higher-order derivatives of $f$ and $g$ of at most order $j$ each, and so these are each bounded by $L^{j+1}$. 
Thus, since the $\sum_{j=0}^{N-2}{\eta^{j}\left|H_{j+2}(z)\right|}$ only goes up to order $j+2=N$ (i.e., with respect to $H_{j+2}$), $\sum_{j=0}^{N-2}{\eta^{j}\left|H_{j+2}(z)\right|}$ involves a sum of products of higher-order derivatives of at most order $N$, and so it is bounded.
\end{proof}

Given what we have proven so far for Conjecture \ref{Conjecture_MH_General}, as of now, Theorem \ref{regretting} confirms that for any dimension $d$, choosing $\eta=\Theta\left(K^{-1/5}\right)$ gives us $R_{K}=\mathcal{O}\left(K^{1/5}\right)$ and $\overline{\dg}_{K}=\mathcal{O}\left(K^{-4/5}\right)$ when all derivatives up to order $N=5$ for $F$ and $G$ are bounded on $\P$ and $\Q$, respectively. 

Furthermore, for $d=1$, choosing $\eta=\Theta\left(K^{-1/12}\right)$ gives us $R_{K}=\mathcal{O}\left(K^{1/12}\right)$ and $\overline{\dg}_{K}=\mathcal{O}\left(K^{-11/12}\right)$ when all derivatives up to order $N=12$ for $F$ and $G$ are bounded. 

But finally, \textit{provided that Conjecture \ref{Conjecture_MH_General} is true for all $N\in\N_{0}$}, Theorem \ref{regretting} implies that setting $\eta=\Theta\left(K^{-\varepsilon}\right)$ gives us a total regret $R_{K}=\mathcal{O}\left(K^{\varepsilon}\right)$ and a duality gap of the average iterates which decays like $\overline{\dg}_{K}=\mathcal{O}\left(K^{-1+\varepsilon}\right)$ for any fixed $\varepsilon>0$ when all higher-order derivatives of $F$ and $G$ are bounded. 

\subsection{Analytic and geometric paths to bounded regret for AMD}

Even if Conjecture \ref{Conjecture_MH_General} is true for all $N\in\N_{0}$, we still \textit{cannot} just take $\varepsilon\rightarrow 0$ (or equivalently, $N\rightarrow\infty$) in $R_{K}$ and $\overline{\dg}_{K}$ unless the $\widetilde{H}_{\eta}$ series remainder $\sup_{z\in\Z}\sum_{j=0}^{N-2}{\eta^{j}\left|H_{j+2}(z)\right|}$ remains bounded as $N\rightarrow\infty$ and $\left|\eta L\sigma_{\max}\left(A\right)\right|\leq 1$. The latter condition necessitates a uniform bound $L$ on \textit{all} higher-order derivatives of $F$ and $G$, which is highly restrictive. 

That being said, we can avoid using Theorem \ref{regretting} and effectively take $\varepsilon\rightarrow 0$ when $\widetilde{H}_{\eta}$ converges under certain conditions, as summarized below:

\begin{theorem}\label{absolute}
    Let $\mathcal{A},\mathcal{B}\subseteq \R^{d}$ be closed and convex. Suppose we start from $(a_{0},b_{0})\in\A\times\B$ such that there exists $M>0$ with $D_{\alpha}\left(a_k,a_{0}\right),D_{\beta}\left(b_k,b_{0}\right)\leq M/2$ for all $k\in\N$, and that $\alpha^{*}\in C^{\infty}\left(\mathfrak{A}\right)$, $\beta^{*}\in C^{\infty}\left(\mathfrak{B}\right)$. Furthermore, suppose that the following hold when $0<\eta<\eta_{\max}$ for some $\eta_{\max}>0$ possibly a function of $A$, $\A$, $\B$, $\alpha$, and $\beta$: 
    \begin{itemize}
        \item The modified Hamiltonian $\widetilde{H}_{\eta}$ generated by $F\left(\cdot\right)=\alpha^{*}\left(U\left(\cdot\right)\right)$ and $G\left(\cdot\right)=\beta^{*}\left(V\left(\cdot\right)\right)$ exists (i.e., converges pointwise) at each point in $\Z$, where $\P,\Q\subseteq\R^{d}$ are chosen such that $U\P=\mathfrak{A}$ and $V\Q=\mathfrak{B}$; and
        \item The second-order tail $\left|\widetilde{H}_{\eta}(z_k)-\widetilde{H}_{\eta}^{(1)}(z_k)\right|$ is uniformly bounded for all $k\in\N$ by some $G>0$, where the $H_{j}$'s are the $j$th-order corrections to $\widetilde{H}_{\eta}$ (see Section \ref{MH_more}).
    \end{itemize}
    Then, for all $\eta\in\left(0,\eta_{\max}\right)$, the total regret $R_{K}$ for any horizon $K$ is bounded by $\\ \frac{1}{\eta}\left(2G+M\right)<\infty$. Thus, $R_{K}=\mathcal{O}\left(1\right)$, and provided that $\lnorm a_k\rnorm$ and $\lnorm b_k\rnorm$ are uniformly bounded for all $k\in\N$, the duality gap of the average iterates $\overline{\dg}_{K}$ decays like $\mathcal{O}\left(K^{-1}\right)$. 
\end{theorem}

\begin{proof}
    As with the proof for Theorem \ref{regretting}, we start by taking $\mathcal{A}$ and $\mathcal{B}$ to be closed, convex sets, in which case \eqref{boundd} holds all the same. Hence, after applying Lemma \ref{Brregman}, \eqref{boundd}, and the triangle inequality, we have
\begin{align}
    \eta R_{K}
    % &\leq M+\left|\widetilde{H}_{\eta}^{(1)}\left(z_{K}\right)-\widetilde{H}_{\eta}^{(1)}\left(z_{0}\right)\right|\nonumber\\
    &\leq M+\left|\widetilde{H}_{\eta}\left(z_{K}\right)-\widetilde{H}_{\eta}^{(1)}\left(z_{K}\right)\right|+\left|\widetilde{H}_{\eta}\left(z_{K}\right)-\widetilde{H}_{\eta}\left(z_{0}\right)\right|+\left|\widetilde{H}_{\eta}\left(z_{0}\right)-\widetilde{H}_{\eta}^{(1)}\left(z_{0}\right)\right|.\label{conservedfull}
\end{align}
We assume that the modified Hamiltonian $\widetilde{H}_{\eta}$ converges everywhere on $\Z$. Thus, the third term in \eqref{conservedfull} vanishes since $\widetilde{H}_{\eta}$ is conserved under the iterations of \eqref{init}, and hence likewise by DAMD \eqref{DAMD} and then AMD \eqref{losses}. Furthermore, by assuming that $\widetilde{H}_{\eta}(z_k)-\widetilde{H}_{\eta}^{(1)}(z_k)$ is bounded by $G>0$ for all $k\in\N$, we bound \eqref{conservedfull} from above as follows:
\begin{gather}
    \eta R_{K}\leq M+G+0+G=M+2G\qquad\iff\qquad R_{K}\leq\frac{1}{\eta}\left(M+2G\right).\nonumber
\end{gather}
The rest of the proof follows analogously to the proof for Theorem \ref{regretting} as it does after \eqref{combined}.
\end{proof}

Hence, when the modified Hamiltonian $\widetilde{H}_{\eta}$ converges globally and its second-order correction $\widetilde{H}_{\eta}-\widetilde{H}_{\eta}^{(1)}$ remains uniformly bounded along the iterates, Theorem~\ref{absolute} implies that AMD attains the same algorithmic complexity as proximal (implicit) mirror descent~\citep{nemirovski2004mirrorprox,Wibisono2022}. 

As an application of Theorem~\ref{absolute}, Theorem~\ref{Thm_Quadratic_MH_multivar} identifies the quadratic setting as a case where absolute convergence of $\widetilde{H}_{\eta}$ can be established rigorously. This setting is equivalent to Alternating Gradient Descent (AGD) as studied by~\cite{pmlr-v125-bailey20a}, whose Theorem~1 proves bounded regret by a different argument. 

In this case, we have $\Z=\mathbb{R}^{d}\times\mathbb{R}^{d}$, $F(p)=p^{\top}Bp$ and $G(q)=q^{\top}Cq$ with $B$ and $C$ positive-definite, and regularizers with conjugates
\begin{gather*}
    \alpha(a)=\frac{1}{4}a^{\top}B^{-1}a \quad \Longleftrightarrow \quad \alpha^{*}(p)=p^{\top}Bp,
\qquad
\beta(b)=\frac{1}{4}b^{\top}C^{-1}b \quad \Longleftrightarrow \quad \beta^{*}(q)=q^{\top}Cq,
\end{gather*}
such that $F=\alpha^{*}$ and $G=\beta^{*}$. The modified Hamiltonian $\widetilde{H}_{\eta}$ converges absolutely whenever $\eta^{2}\sigma_{\max}(BC) < 1$. Furthermore, when $\eta^{2}\sigma_{\max}(B)\sigma_{\max}(C)<1$, Theorem~5 in \cite{pmlr-v125-bailey20a} implies that the symplectic Euler (or equivalently, AGD) iterates lie on compact quadratic level sets (ellipsoids). The corresponding Bregman divergences are the weighted Euclidean (Mahalanobis) distances
\begin{gather*}
    D_{\alpha}(a_{k},a_{0})=\frac{1}{4}\|a_{k}-a_{0}\|_{B^{-1}}^{2}=\frac{1}{4}(a_{k}-a_{0})^{\top}B^{-1}(a_{k}-a_{0}),\\
    D_{\beta}(b_{k},b_{0})=\frac{1}{4}\|b_{k}-b_{0}\|_{C^{-1}}^{2}=\frac{1}{4}(b_{k}-b_{0})^{\top}C^{-1}(b_{k}-b_{0}),
\end{gather*}
and hence, these remain uniformly bounded for all iterates $k\in\N$. Moreover, since $\widetilde{H}_{\eta}$ converges absolutely on these compact ellipsoids by Theorem~\ref{Thm_Quadratic_MH_multivar}, the remainder $\sum_{j=2}^{\infty}\eta^{j}H_{j}(z_{k})$ is uniformly bounded for all $k\in\mathbb{N}$. 
Hence, by applying Theorem~\ref{absolute}, we recover the $\mathcal{O}(K^{-1})$ time-average regret obtained in \cite{pmlr-v125-bailey20a}.  

However, to show that the conditions from Theorem~\ref{absolute} hold in the quadratic case, we used the fact that the symplectic Euler iterates lie on compact ellipsoids. As we noted at the end of Section \ref{Section_TMC_Regret}, bounded iterates already imply a bounded regret by Lemma~\ref{Brregman}:

\begin{corollary}\label{bounded}
Let $\mathcal{A},\mathcal{B}\subseteq\mathbb{R}^d$ be closed and convex, and let 
$z_k\in\mathcal{Z}$ denote the symplectic Euler iterates with stepsize $\eta>0$.
Assume the following:
\begin{enumerate}
\item For some $C>0$, the dual iterates $z_k=\left(p_k,q_k\right)$ satisfy $\lnorm p_k-p_0\rnorm,\lnorm q_k-q_0\rnorm\leq C$ for all $k\in\N$, and
\item $\nabla F=\nabla \alpha^{*}\left(U(\cdot)\right)$ and $\nabla G=\nabla \beta^{*}\left(V(\cdot)\right)$ are continuous on $\P$ and $\Q$, respectively.
\end{enumerate}
Define the closed balls $\mathbb{B}_p:=\{p\in\mathcal{P}:\lnorm p-p_0\rnorm\leq C\}$ and $\mathbb{B}_q:=\{q\in\mathcal{Q}:\lnorm q-q_0\rnorm\leq C\}$. By continuity, the following suprema are finite:
$$
L_F:=\sup_{u\in\mathbb B_p}{\lnorm\nabla F(u)\rnorm},\quad L_G:=\sup_{v\in\mathbb B_q}{\lnorm\nabla G(v)\rnorm},\quad M_F:=\sup_{u\in\mathbb B_p}{\left|F(u)\right|},\quad M_G:=\sup_{v\in\mathbb B_q}{\left|G(v)\right|}.
$$
Let $L:=\max\{L_F,L_G\}$ and $M:=\max\{M_F,M_G\}$. Then, for every horizon $K\in\mathbb{N}$,
$$
R_K\leq\frac{4}{\eta}\left(M+CL\right)+L^{2}.
$$
In particular, $R_K=\mathcal{O}(1)$ and $\overline{\mathrm{dg}}_K=\mathcal{O}(K^{-1})$.
\end{corollary}

\begin{proof}
We start from \eqref{regretbound} in Lemma \ref{Brregman} with $H=F+G$ and $z_k=\left(p_k,q_k\right)$:
\begin{gather}
    R_{K}\leq\frac{1}{\eta}\left(D_{H}\left(z_{0},z_K\right)+\widetilde{H}_{\eta}^{(1)}\left(z_{K}\right)-\widetilde{H}_{\eta}^{(1)}\left(z_{0}\right)\right).\label{key}
\end{gather}
Note that $D_{H}\left(z_0,z_K\right)=D_{F}\left(p_0,p_K\right)+D_{G}\left(q_0,q_K\right)$. Moreover, by convexity,
\begin{align}
    0\leq D_F(p_0,p_K)&=F(p_0)-F(p_K)-\left<\nabla F(p_K),p_0-p_K\right>\nonumber\\ &\leq |F(p_0)|+|F(p_K)|+\lnorm\nabla F(p_K)\rnorm\lnorm p_K-p_0\rnorm.\label{Fangle}
\end{align}
By assumption, $\left|F(p_0)\right|,\left|F(p_K)\right|\leq M$, $\lnorm\nabla F(p_K)\rnorm\leq L$, and $\lnorm p_K-p_0\rnorm\leq C$. Hence, we can bound \eqref{Fangle} as follows:
\begin{gather*}
    D_F(p_0,p_K)\leq M+M+LC=2M+LC.
\end{gather*}
The same bound on $D_{G}\left(q_0,q_K\right)$ follows analogously, whence we have
\begin{gather}
    D_{H}\left(z_0,z_K\right)\leq 2\left(2M+CL\right).\label{DHbound}
\end{gather}
Next, recall that
\begin{gather*}
    \widetilde H^{(1)}_{\eta}(p,q)=F(p)+G(q)-\frac{\eta}{2}\left<\nabla F(p),\nabla G(q)\right>.
\end{gather*}
Thus,
\begin{subequations}\label{H1diff}
    \begin{align}
    \left|\widetilde H^{(1)}_{\eta}(z_K)-\widetilde H^{(1)}_{\eta}(z_0)\right|&\leq|F(p_K)-F(p_0)|+|G(q_K)-G(q_0)|
    \\ &\quad+\frac{\eta}{2}\left|\left<\nabla F(p_K),\nabla G(q_K)\right>-\left<\nabla F(p_0),\nabla G(q_0)\right>\right|.
\end{align}
\end{subequations}
By the mean value inequality,
\begin{gather}
    |F(p_K)-F(p_0)|\leq L\lnorm p_K-p_0\rnorm\leq LC\quad\textrm{and}\quad |G(q_K)-G(q_0)|\leq LC,\label{MVI}
\end{gather}
and by applying the triangle and Cauchy-Schwartz inequalities to \eqref{H1diff}, 
\begin{align}
    \left|\left<\nabla F(p_K),\nabla G(q_K)\right>-\left<\nabla F(p_0),\nabla G(q_0)\right>\right|&\leq\left|\left<\nabla F(p_K),\nabla G(q_K)\right>\right|+\left|\left<\nabla F(p_0),\nabla G(q_0)\right>\right|\nonumber\\
    &\leq \lnorm\nabla F(p_K)\rnorm\lnorm\nabla G(q_K)\rnorm+\lnorm\nabla F(p_0)\rnorm\lnorm\nabla G(q_0)\rnorm\nonumber\\
    &\leq (L)(L)+(L)(L)\nonumber\\
    &=2L^2.\label{triCS}
\end{align}
Hence, using \eqref{MVI} and \eqref{triCS} to bound \eqref{H1diff} from above, we have
\begin{gather}
    \left|\widetilde H^{(1)}_{\eta}(z_K)-\widetilde H^{(1)}_{\eta}(z_0)\right|\leq LC+LC+\frac{\eta}{2}\left(2L^{2}\right)=2LC+\eta L^{2}.\label{difffinal}
\end{gather}
Finally, using \eqref{DHbound} and \eqref{difffinal} together to bound \eqref{key} from above, 
\begin{gather*}
    \boxed{R_{K}\leq\frac{4}{\eta}\left(M+CL\right)+L^{2}.}
\end{gather*}
We then proceed as we did in the proof for Theorem \ref{regretting} after \eqref{combined}.
\end{proof}
Corollary \ref{bounded} shows that when the symplectic Euler iterates remain bounded, the total regret is also bounded under mild assumptions, thereby providing a purely geometric route to proving bounded regret that does not rely on the existence or convergence of the modified Hamiltonian.

Consequently, in the quadratic case, Corollary \ref{bounded} renders Theorem \ref{absolute} unnecessary: The presence of bounded iterates already implies bounded regret, even though this case is among the few for which $\widetilde{H}_{\eta}$ is known to converge absolutely. As mentioned already in Section \ref{MH_more}, for other cases, rigorous proofs of convergence for $\widetilde{H}_{\eta}$ remain scarce, and counterexamples are well documented~\citep{Suris1989,Field2003,Alsallami2018}. Therefore, identifying general, practically verifiable conditions under which a Hamiltonian $H$ yields a convergent modified Hamiltonian $\widetilde{H}_{\eta}$ is an important open problem---one whose resolution would make Theorem \ref{absolute} an analytic foundation for proving bounded regret in AMD.

Finally, we note that a bounded regret might seem superficially obvious from the start, since if the original Hamiltonian is closed, proper, and coercive convex (e.g., for an ellipsoid), 
then its level sets are convex and bounded 
\citep[Section 8]{Rockafellar1996}. However, that is for the Hamiltonian flow generated by the \textit{unperturbed} Hamiltonian $H$. The implication here is that even the Hamiltonian flow generated by the \textit{modified} Hamiltonian $\widetilde{H}_{\eta}$ is associated with a bounded regret in discrete time, albeit under assumptions on $\widetilde{H}_{\eta}$. 

\section{Conclusion}\label{Section_Conclude}

With regards to energy conservation, there are symplectic integrators which tend to conserve the Hamiltonian better than symplectic Euler for the same stepsize---the leapfrog and implicit midpoint integrators are some examples; see, e.g., \cite{bou2018geometric} and Chapter I in \cite{Hairer2006} for further details and numerical illustrations. There are also algorithms with better performance than AMD for the setting introduced in Section \ref{Section_Settings}, such as simultaneous proximal mirror descent \citep{nemirovski2004mirrorprox,Wibisono2022}. However, we are mainly interested in studying symplectic Euler because, when the original Hamiltonian is convex, symplectic Euler is the dual space representation of AMD. Moreover, we are interested in studying AMD because it serves as a relatively simple proof of concept for breaching the barrier between the fields of symplectic integration and theoretical computer science. We encourage future studies which use such connections, e.g., by using known symplectic integrators which tend to perform well and seeing if these can be translated to an analogous algorithm in a game-theoretic setting.

Since the field of symplectic integration (let alone Hamiltonian dynamics as a whole) has existed for over three decades, there is no doubt that solid results and rigorous error bounds already exist regarding the trajectories and energy conservation of said methods. And for countless applications in computational physics, chemistry, and biology, these have proved to be sufficient \citep[e.g.,][]{Engle2005,Skeel,Casas}. However, in the setting of Hamiltonian dynamics as it appears in sampling, optimization, and games, we are more interested in the algorithmic implications and guarantees which such methods might provide. We hope that this paper not only provides a foundation of properties and concepts---a framework---that demonstrates how new and old results in symplectic numerical analysis can be used in such contexts, but also reveals how many more questions remain unanswered.

\appendix

\section{Lie algebra and Poisson brackets}\label{Liealgebra}

\paragraph{Lie algebra.}
Recall a {\bf Lie algebra} $L$ is a vector space (over $\R$ in our case) with a {\em Lie bracket} operation $[\,\cdot\, , \,\cdot\,] \colon L \times L \to L$ which is bilinear, anticommutative ($[\ell_1,\ell_2] = -[\ell_2, \ell_1]$ for all $\ell_1,\ell_2 \in L$), and satisfies the Jacobi identity ($[\ell_1,[\ell_2,\ell_3]] + [\ell_2,[\ell_3,\ell_1]] + [\ell_3,[\ell_1,\ell_2]] = 0$ for all $\ell_1,\ell_2,\ell_3 \in L$).
For further references on Lie algebra, see e.g.~\cite{Hall2015,Jacobson1979,Varadarajan1984}.

\paragraph{Adjoint representation.}
Recall that for any Lie algebra $L$ equipped with Lie bracket $\left[\cdot,\cdot\right]$, for any element $\ell \in L$ we can define the {\bf adjoint representation} of $\ell$:
$$\ad_{\ell}\left(\cdot\right)\coloneqq \left[\ell,\cdot\right]$$ 
as a linear function on $L$. 

We define the {\em adjoint action} of an element $\ell \in L$ by:
\begin{align}
\ad_{\ell }\left( \cdot \right)\coloneqq \left[\ell, \cdot\right].
\end{align}
This means for any  $\widetilde{\ell} \in L$,
\begin{align}
\ad_{\ell}\left(\widetilde{\ell}\right)=\left[\ell,\widetilde{\ell}\right].
\end{align}

\noindent We also define the $k$-fold application of adjoint actions $\ad_\ell^{k}$ by:
\begin{align}
    \ad_\ell^1 &= \ad_\ell \\
    \ad_\ell^{k+1} &= \ad_\ell^k(\ad_\ell(\cdot)).
\end{align}
With the Poisson bracket as the Lie bracket, this is equivalent to iterated Poisson bracket:
\begin{align}
    \ad_\varphi^k = \{\{\{...\{\cdot,\underbrace{\varphi\},\varphi\},\varphi}_{k~\text{times}}\} = \{\cdot, \varphi^k\},
\end{align}
where $\varphi\in C^{\infty}\left(\Z\right)$.

\paragraph{Poisson brackets.}
For $F,G \in C^n(\Z)$ for some $n \in \N$, we recall their {\em Poisson bracket} $\{F,G\} \in C^{n-1}(\Z)$ is the function given by:
\begin{align}\label{Eq:PBDef}
    \{F,G\} = -(\nabla_{p}F)^\top (\nabla_{q} G) + (\nabla_{q}F)^\top (\nabla_{p}G).
\end{align}
More precisely, for any $(p,q) \in \Z$,
\begin{align*}
    \{F,G\}(p,q) = -(\nabla_{p}F (p,q))^\top (\nabla_{q} G(p,q)) + (\nabla_{q} F(p,q))^\top (\nabla_{p} G(p,q)).
\end{align*}
Note that we can also write the Poisson bracket as an inner product of the gradient vectors of $F$ and $G$ with respect to the symplectic matrix $\Omega = \begin{pmatrix}
    0 & -I \\ 
    I & 0
\end{pmatrix}$:
\begin{align*}
    \{F,G\}(p,q) = (\nabla F(p,q))^\top \, \Omega \, \nabla G(p,q).
\end{align*}

\paragraph{Iterated Poisson brackets.}
For $N \in \N$ and $r_{1},\dots,r_{N},s_{1},\dots,s_{N} \in \N_0$, we define the {\bf iterated Poisson brackets (IPBs)} of $F$ and $G$ as follows:
\begin{align*}
    \left\{F^{s_{1}}G^{r_{1}}\cdots F^{s_{N}}G^{r_{N}}\right\} =\{\cdots\underbrace{\left\{\left\{F,\dots,F\right\},F\right\}}_{s_{1}},\underbrace{\left.\left.\left.G\right\},\cdots G\right\},G\right\}}_{r_{1}},\cdots\underbrace{\left.\left.\left.F\right\},\cdots F\right\},F\right\}}_{s_{N}},\underbrace{\left.\left.\left.\left.G\right\},\cdots\right\}G\right\},G\right\}}_{r_{N}}.
\end{align*}
Here we assume $F$ and $G$ are differentiable as many times as necessary for the definition above to make sense.
The definition of Poisson brackets ~\eqref{Eq:PBDef} corresponds to the case $N=1$, $r_1 = s_1 = 1$: $\{F^1 G^1\} = \{F,G\}$, and we can define the general IPB inductively.
Using this notation, we can write for example $\{\{F, G\}, G\} = \{F^1 G^2\}$, and 
$\{FG^4F^2\} = \{\{\{\{\{\{F,G\}, G\}, G\}, G\}, F\}, F\}.$

One can check that the Poisson bracket satisfies the axioms for a Lie bracket. Thus, results for Lie algebra also apply to those for a space of $C^{\infty}$ functions equipped with a Poisson bracket.

\paragraph{Free Lie algebra.}
Let $F,G\in C^{\infty}(\Z)$. 
Consider the subspace $\mathcal{L}\left(F,G\right)$ of $C^{\infty}\left(\R^{d}\right)$ of functions generated by the linear combinations of $F$, $G$, and their IPBs: 
\begin{align}\label{Eq:FreeLieAlgebra}
    \mathcal{L}\left(F,G\right) \coloneqq  \text{Span}\left\{\left\{F^{s_{1}}G^{r_{1}}\cdots F^{s_{N}}G^{r_{N}}\right\} \colon r_{1},s_{1},\dots,r_{N},s_{N}\in\N_0, N\in\N\right\}.
\end{align}
The vector space $\mathcal{L}\left(F,G\right)$ equipped with the Poisson bracket $\left\{\,\cdot\,,\,\cdot\,\right\}$ defines a Lie algebra; this is called the {\bf free Lie algebra} generated by $F$ and $G$. For a more general definition of free Lie algebra and some properties, see Chapter 3 from \cite{Varadarajan1984}.

\paragraph{Exponential map.}
For some $\eta>0$, we can formally expand the exponential $\exp\left(\eta\ad_{\varphi}\right)$ of the adjoint representation $\ad_{\varphi}$ of some $\varphi\in C^{\infty}\left(\Z\right)$ via the repeated application of $\ad_{\varphi}$, which can be expressed as IPBs:
\begin{gather}
    \exp\left(\eta\ad_{\varphi}\right)=\sum_{k=0}^{\infty}{\frac{\eta^{k}\ad_{\varphi}^{k}}{k!}}=\sum_{k=0}^{\infty}{\frac{\eta^{k}\{\{\{...\{\cdot,\overbrace{\varphi\},\varphi\},\varphi\},..., \varphi}^{k}\}}{k!}} = \sum_{k=0}^{\infty}{\frac{\eta^{k}\{\cdot, \varphi^k\}}{k!}}\label{you}
\end{gather}
See \cite[][Chapter 2]{Varadarajan1984} or \cite[][Chapters 3-5]{Hall2015} for how the exponential of an adjoint representation \eqref{you} generalizes over other Lie algebras.

\paragraph{Hamiltonian flow from Poisson bracket.}
The time evolution of the phase space variables $z=\left(p,q\right)\in\Z$ under a Hamiltonian flow \eqref{Eq:HF} is generated by the Poisson bracket operation with $H$ (see, e.g., Section 5.2 in \cite{Jose1998} or Section 1.3 in \cite{Arnold2006}):
\begin{gather}
    \dot z(t)=\left\{z,H(z)\right\}=\Omega\nabla H\left(z(t)\right).\label{z-bracket}
\end{gather}
Then we can integrate \eqref{z-bracket} from $t=0$ to $t=\eta$ to get
\begin{gather}
z(\eta)=\exp\left(\eta\ad_{H}\right)z(0).\label{expexp}
\end{gather}

\section{The modified Hamiltonian and the BCH series}\label{Section_BCH}

We now discuss how to construct the modified Hamiltonian for the symplectic Euler method via the Baker-Campbell-Hausdorff series. 
To explain the results, we use the tools from Lie algebra and representation theory introduced in Appendix \ref{Liealgebra}.

\subsection{Baker-Campbell-Hausdorff series}

For any $x,y\in L$, the value of $z$ which solves $\exp\left(x\right)\exp\left(y\right)=\exp\left(z\right)$ can be given via the Dynkin form of the \textit{Baker-Campbell-Hausdorff formula} \citep{Jacobson1979} in {\em Dynkin form}, which we call the {\it BCHD series}:
\begin{align}
    z&=\log\left(\exp\left(x\right)\exp\left(y\right)\right)\\
    &=\sum_{n=1}^{\infty}{\frac{\left(-1\right)^{n-1}}{n}{\sum_{\substack{ r_{1}+s_{1}>0 \\ \vdots \\ r_{n}+s_{n}>0}}{\frac{\eta^{r_{1}+\dots+r_{n}+s_{1}+\dots+s_{n}-1}\left[x^{r_{1}}y^{s_{1}}x^{r_{2}}y^{s_{2}}\cdots x^{r_{n}}y^{s_{n}}\right]}{\left(r_{1}+\dots+r_{n}+s_{1}+\dots+s_{n}\right)\prod_{i=1}^{n}{r_{i}!s_{i}!}}}}}.\label{predynkin}
\end{align}
Here for $r_1,s_1,\dots,r_n,s_n \in \N_0$ and $n \in \N$, we denote
\begin{gather}
    \left[x^{r_{1}}y^{s_{1}}x^{r_{2}}y^{s_{2}}\cdots x^{r_{n}}y^{s_{n}}\right]=[\underbrace {x,[x,\dotsm [x} _{r_{1}},[\underbrace {y,[y,\dotsm [y} _{s_{1}},\,\dotsm \,[\underbrace {x,[x,\dotsm [x} _{r_{n}},[\underbrace {y,[y,\dotsm y} _{s_{n}}]]\dotsm ]]].\nonumber
\end{gather}
See also \cite{Achilles2012} and references therein for an overview of different proofs of the BCH formula for the Dynkin form and others.

In particular, when $L=\mathcal{L}\left(F,G\right)$ and the Poisson bracket is the Lie bracket, we can define the adjoint representation \citep{Hall2015} $\ad_{\varphi}\left(\cdot\right)\coloneqq \left\{\cdot,\varphi\right\}$ of \textit{any} function $\varphi\in\mathcal{L}\left(F,G\right)$---i.e., the linear map defined by the Poisson bracket of a function in $\mathcal{L}\left(F,G\right)$ with $\varphi$. Furthermore, instantiating \eqref{predynkin} with $\varphi,\psi\in\mathcal{L}\left(F,G\right)$ and the Poisson bracket gives the same expression as \eqref{predynkin}, except with the iterated Lie brackets as IPBs \textit{in reverse order}.\footnote{This is because when defining the adjoint representation over a general Lie algebra $L$ vs. $\mathcal{L}\left(F,G\right)$, we swap the ordering of the arguments for a general Lie bracket vs. the Poisson bracket, respectively. The derivation of the BCH formula involves interpreting $\exp\left(x\right)$ and $\exp\left(y\right)$ via their adjoint representations, thereby leading to an expansion in terms of iterated Lie brackets.}

\subsection{Deriving the modified Hamiltonian}\label{cons_section}

In summary, the Dynkin form \eqref{dynkin} of the modified Hamiltonian $\widetilde{H}_{\eta}$ is derived as a formal expansion by constructively working from the initial assumptions that there is a Hamiltonian flow \eqref{Eq:HFComp} generated by some modified Hamiltonian $\widetilde{H}_{\eta}$ which interpolates the iterates of, and is thereby conserved by, the iterates of the symplectic Euler method \eqref{init}. Hence, we first show that the iterations of the symplectic Euler method \eqref{init} conserve $\widetilde{H}_{\eta}$ under the convergence conditions outlined in Theorem \ref{conserved_Hamiltonian}. But at first, we provide a proof sketch of Theorem \ref{conserved_Hamiltonian}, following the expositions in \cite{Field2003,Alsallami2018}. 

\paragraph{Proof sketch.} Recall from \eqref{expexp} that Hamiltonian flow \eqref{Eq:HFComp} is equivalent to exponentiating the adjoint operation of the Hamiltonian function:
\begin{align*}
z(\eta)=\exp\left(\eta\ad_{H}\right)z(0).
\end{align*}
Furthermore, we can expand
\begin{gather}
    \exp\left(\eta\ad_{H}\right)=\sum_{k=0}^{\infty}{\frac{\eta^{k}\ad_{H}^{k}}{k!}} = \sum_{k=0}^{\infty}{\frac{\eta^{k}\{\cdot, H^k\}}{k!}}.
\end{gather}
Analogously (see Appendix \ref{connection}), the flow generated by the symplectic Euler method \eqref{init} can be expressed as iterations of $\exp\left(\eta \ad_{G}\right)\exp\left(\eta \ad_{F}\right)$ acting on $z(0)$. 
Then we want to find a modified Hamiltonian $\widetilde{H}_{\eta}$ for the flow generated by $\exp\left(\eta \ad_{F}\right)\exp\left(\eta \ad_{G}\right)$.
That is, we want to find $\widetilde{H}_{\eta}$ such that
\begin{align}
    \exp\left(\eta \ad_{F}\right)\exp\left(\eta \ad_{G}\right)=\exp\left(\eta \ad_{\widetilde{H}_{\eta}}\right),\label{Bieber}
\end{align}
or equivalently, such that
\begin{align}
    \left(p_{k+1},q_{k+1}\right)=\exp\left(\eta \ad_{\widetilde{H}_{\eta}}\right)\left(p_{k},q_{k}\right),\nonumber
\end{align}
since $z(t=0)=\left(p_{k},q_{k}\right)$ and $z\left(t=\eta\right)=\left(p_{k+1},q_{k+1}\right)$. Taking the formal inverse of the exponential operation in \eqref{Bieber} gives
\begin{align}
\ad_{\widetilde{H}_{\eta}}=\frac{1}{\eta}\log\left(\exp\left(\eta \ad_{F}\right)\exp\left(\eta \ad_{G}\right)\right).\label{expand}
\end{align}
We expand \eqref{expand} formally via the BCHD series \ref{dynkin} in terms of Lie brackets on the set of automorphisms $\textrm{Aut}\left(\mathcal{L}\left(F,G\right)\right)$ of $\mathcal{L}\left(F,G\right)$:
\begin{gather}
    \ad_{\widetilde{H}_{\eta}}=\ad_{F}+\ad_{G}+\frac{\eta}{2}\left[\ad_{F},\ad_{G}\right]+\frac{\eta^{2}}{12}\left(\left[\ad_{F},\left[\ad_{F},\ad_{G}\right]\right]+\left[\ad_{G},\left[\ad_{G},\ad_{F}\right]\right]\right)+\mathcal{O}\left(\eta^{3}\right),\label{fin}
\end{gather}
where $\left[\Phi,\Psi\right]=\Phi\Psi-\Psi\Phi$ for any $\Phi,\Psi\in\textrm{Aut}\left(\mathcal{L}\left(F,G\right)\right)$. 
The Jacobi identity and the anti-commutativity of the Poisson bracket (e.g., Section 1.3 in \cite{Arnold2006} or Section 5.2 in \cite{Jose1998}) imply that
\begin{gather}
    \ad_{\left\{\psi,\varphi\right\}}=\left[\ad_{\varphi},\ad_{\psi}\right]\label{Lie}
\end{gather}
for all $\varphi,\psi\in\mathcal{L}\left(F,G\right)$.
Combining \eqref{fin} and \eqref{Lie}, we derive the adjoint representation of \eqref{dynkin}, i.e.,
\begin{gather}
    \sum_{n=1}^{\infty}{\frac{\left(-1\right)^{n-1}}{n}{\sum_{\substack{ r_{1}+s_{1}>0 \\ \vdots \\ r_{n}+s_{n}>0}}{\frac{\eta^{r_{1}+\dots+r_{n}+s_{1}+\dots+s_{n}-1}\ad_{\left\{G^{r_{1}}F^{s_{1}}G^{r_{2}}F^{s_{2}}\cdots G^{r_{n}}F^{s_{n}}\right\}}}{\left(r_{1}+\dots+r_{n}+s_{1}+\dots+s_{n}\right)\prod_{i=1}^{n}{r_{i}!s_{i}!}}}}}.\nonumber
\end{gather}
Hence, tracing our steps backwards from here after using the series \eqref{dynkin} as an ansatz for $\widetilde{H}_{\eta}$, we find the formula for $\widetilde{H}_{\eta}$.

We note there are simpler forms of the BCH series, such as the celebrated Goldberg formula which expresses the order-by-order form of the BCH formula as a sum of (generally, non-commutative) products of $F$ and $G$ \textit{in the Lie algebra $\mathcal{L}\left(F,G\right)$}\footnote{The product $\oslash$ of any two $\varphi,\psi\in\mathcal{L}\left(F,G\right)$ is defined as $\varphi\oslash\psi\coloneqq \nabla_{q}\varphi\cdot\nabla_{p}\psi$.} scaled by the so-called Goldberg coefficients \citep{Newman1987} or a similar form which writes the terms via Lie brackets (or in particular, Poisson brackets) of $F$ and $G$ \citep{cyclic}. These Goldberg formulae are useful for computations in practice and for simplifying the number of terms in the sum, but we will not use them in this paper. 

Furthermore, there is also a recursive formula for the BCH series \citep{Casas,Varadarajan1984} which expresses each successive order-by-order correction $H_{j}$ to the BCH series in terms of a sum of iterated Lie brackets (or IPBs when working with Poisson brackets in particular), the latter of which are taken over prior terms in the BCH series. Lemma 2.15.3 in \cite{Varadarajan1984} gives the most general form of this recursive formula over a general Lie algebra, but for this study, \eqref{recursive1} and \eqref{recursive2} give the particular formula when applied to computing the order-by-order corrections $H_{j}$ of $\widetilde{H}_{\eta}$. We implement this formula for $H_{j}$ in a computer program to aid us in proving Lemma \ref{Coefficients_Cancellation} and Conjecture \ref{Conjecture_MH_General} up to $N=10$ for $d=1$---see Appendix \ref{CompTaylor} for more details.

Since determining the convergence of the $\widetilde{H}_{\eta}$ series \eqref{dynkin} is needed to use Theorem \ref{conserved_Hamiltonian}, this will be mentioned in Section \ref{MH_more}. However, even when such convergence cannot be assumed, \eqref{dynkin} still proves to be useful as a formal expansion, as we will show in Section \ref{Section_TMC_Regret}.

\paragraph{Connecting symplectic Euler with the Poisson bracket.}\label{connection}

There is a formal connection between translation by the gradient of a function, as seen in each of the two steps of symplectic Euler method \eqref{init}, and multiplication by the exponential of the adjoint representation (in the sense of Lie algebra representation) of that function. This connection can be summarized via the following lemma:

\begin{lemma}\label{Lem:expbracket}
    Let $F:\P\rightarrow\R$ and $G:\Q\rightarrow\R$ be differentiable over $\P\subseteq\R^{d}$ and $\Q\subseteq\R^{d}$, respectively. For all $(p,q)\in\Z$,\footnote{We remind the reader that $\ad_{F}\left(\cdot\right)\coloneqq \left\{\cdot,F\right\}$ and $\ad_{G}\left(\cdot\right)\coloneqq \left\{\cdot,G\right\}$, as defined in Section \ref{Liealgebra}.}
    \begin{gather}
        \left.\exp\left(\eta \ad_{G}\right)\right|_{q=q}p=p-\eta\nabla_{q}G\left(q\right),\label{p1}
    \end{gather}
    and for all $(p,q)\in\Z$,
    \begin{gather}
        \left.\exp\left(\eta \ad_{F}\right)\right|_{p=p}q=q+\eta\nabla_{p}F\left(p\right).\label{q1}
    \end{gather}
\end{lemma}

\begin{proof}
    We set out to prove \eqref{p1}; the proof for \eqref{q1} follows the same line of reasoning analogously. Let $p\in \P$. We introducing the time-like variable $0\leq s\leq\eta$ and define the trivial ODE
    \begin{gather}
        \frac{dx\left(s\right)}{ds}=\left.\ad_{G}\left(\widetilde{p}\right)\right|_{\widetilde{p}=x(s)}=-\nabla_{q}G\left(q\right),\label{sODE}
    \end{gather}
    and let $x\left(0\right)=p$.\footnote{For notational ease, we assume that both sides of \eqref{sODE} are evaluated at $q$ and the Poisson bracket derivatives are taken with respect to $\widetilde{p}$ and $q$ rather than $p$ and $q$, respectively.} We can solve for $x(s)$ evaluated at $s=\eta$ by directly integrating \eqref{sODE} as follows:
    \begin{align}
        x\left(\eta\right)&=x\left(0\right)-\int_{0}^{\eta}{\nabla_{q}G\left(q\right)ds}\nonumber\\
        &=p-\eta\nabla_{q}G\left(q\right).\label{miles}
    \end{align}
    Hence, $x\left(\eta\right)$ is equal to the right-hand side of \eqref{p1}. Furthermore, we can take the Taylor series of $x\left(\eta\right)$ about $\eta=0$ as
    \begin{gather}
        x\left(\eta\right)=\sum_{k=0}^{\infty}{\frac{\eta^{k}}{k!}\left.\frac{d^{k}x\left(s\right)}{ds^{k}}\right|_{s=0}}\label{Taylors}
    \end{gather}
    for all $\left|\eta\right|$ sufficiently small. (Bear for a moment as to how small.) To rewrite $\frac{d^{k}x\left(s\right)}{ds^{k}}$ in terms of $\ad_{G}$ and $p$, note that by \eqref{sODE},
    \begin{gather}
         \frac{dx}{ds}=\left.\ad_{G}\left(\widetilde{p}\right)\right|_{\widetilde{p}=x(s)}.\label{back}
    \end{gather}
    From here, we claim that
    \begin{gather}
        \frac{d^{k}x}{ds^{k}}=\left.\ad_{G}^{k}\left(\widetilde{p}\right)\right|_{\widetilde{p}=x(s)}\label{hypothesis}
    \end{gather}
    for all $k\geq 0$. The case for $k=0$ is trivially true and \eqref{back} establishes $k=1$. However, for any $k\geq 2$, since $\left\{p,G\right\}=-\nabla_{q}G$ and $\left\{q,F\right\}=\nabla_{p}F$, we have $\ad_{G}^{k}p=\ad_{G}^{k-2}\left\{\left\{p,G\right\},G\right\}=\ad_{G}^{k-2}\left\{-\nabla_{q}G,G\right\}=\ad_{G}^{k-2}0=0$ and $\ad_{F}^{k}q=\ad_{F}^{k-2}\left\{\left\{q,F\right\},F\right\}=\ad_{F}^{k-2}\left\{\nabla_{p}F,F\right\}=\ad_{F}^{k-2}0=0$. But this should be expected by \eqref{sODE}, since for $k\geq 2$, $$\frac{d^{k}x}{ds^{k}}=\frac{d^{k-1}}{ds^{k-1}}\frac{dx}{ds}=\frac{d^{k-1}}{ds^{k-1}}\left[-\nabla_{q}G\left(q\right)\right]=0$$ likewise. Therefore, using \eqref{hypothesis}, \eqref{Taylors} can be rewritten as
    \begin{gather}
        x\left(\eta\right)=\sum_{k=0}^{\infty}{\frac{\eta^{k}}{k!}\left.\frac{d^{k}x\left(s\right)}{ds^{k}}\right|_{s=0}}=\sum_{k=0}^{\infty}{\frac{\eta^{k}}{k!}\left.\ad_{G}^{k}\left(\widetilde{p}\right)\right|_{\widetilde{p}=x\left(0\right)}}=\left.\exp\left(\eta \ad_{G}\right)\right|_{q=q}x(0)=\left.\exp\left(\eta \ad_{G}\right)\right|_{q=q}p.\label{inter}
    \end{gather}
    By \eqref{miles} and \eqref{inter}, we conclude.
\end{proof}

 We note that Lemma \ref{Lem:expbracket} is, in a sense, trivial, since the terms in the series expansions for the exponentials in \eqref{p1} and \eqref{q1} vanish after two terms. After all, the two component flows under $G$ (applied to $p$) and then $F$ (applied to $q$) each involve a mere translation when considered separately. However, our goal here is to find an equivalent \textit{Hamiltonian} flow for the composition of these two flows at once---i.e., the flow induced by the modified Hamiltonian, $\widetilde{H}_{\eta}$, in the sense of \eqref{Eq:HFComp}---which turns out to be highly nontrivial. We introduce the exponential map notation because the algebraic properties of the exponential map will be used to derive $\widetilde{H}_{\eta}$ and show that it is conserved under the symplectic Euler method. 

\paragraph{Conservation of the modified Hamiltonian under symplectic Euler.} Here we state our proof of Theorem \ref{conserved_Hamiltonian}.

\begin{proof}
Making necessary assumptions on $F$ and $G$ to use Lemma \ref{Lem:expbracket}, we rewrite the next iteration after applying symplectic Euler to some initial point $\left(p,q\right)$ as 
\begin{gather}
    \exp\left(\eta \ad_{F}\right)\exp\left(\eta \ad_{G}\right)\left(p,q\right)\label{rewrite}
\end{gather}
using the following notation:
\begin{enumerate}[label={(\arabic*)}]
    \item $\exp\left(\eta \ad_{G}\right)$ is applied to the first argument with $\ad_{G}$ evaluated at the second argument;
    \item $\exp\left(\eta \ad_{G}\right)$ returns the results after application (as described in (1)) in the first argument and leaves the second argument unchanged; and
    \item $\exp\left(\eta \ad_{F}\right)$ does the same as in (1) and (2) except with the two arguments switched.
\end{enumerate}
Hence, by \eqref{rewrite}, to find a new modified Hamiltonian $\widetilde{H}_{\eta}$ whose associated continuous-time Hamiltonian system interpolates the discrete-time trajectories of the symplectic Euler method \eqref{init}, we want $\widetilde{H}_{\eta}$ to satisfy \begin{gather}
    \exp\left(\eta \ad_{F}\right)\exp\left(\eta \ad_{G}\right)\left(p,q\right)=\exp\left(\eta \ad_{\widetilde{H}_{\eta}}\right)\left(p,q\right)\label{adjoint}
\end{gather}
for all $(p,q)\in \Z$, where $\exp\left(\eta \ad_{\widetilde{H}_{\eta}}\right)$ is evaluated component-wise to $\left(p,q\right)$.

As covered in \cite[][Chapter 3]{Hall2015}, \eqref{adjoint} can be rewritten using the adjoint map $\ad:\mathbb{L}\left(F,G\right)\rightarrow\textrm{Aut}\left(\mathcal{L}\left(F,G\right)\right)$:\footnote{In the context of abstract algebra, an automorphism is a bijective homomorphism of an algebraic structure (e.g., a group, ring, or Lie algebra) with itself \citep{Jacobson1979,Hall2015}. For any algebraic structure $\mathcal{A}$, $\textrm{Aut}\left(\mathcal{A}\right)$ is the set of automorphisms on $\mathcal{A}$; $\textrm{Aut}\left(\mathcal{A}\right)$ is itself a group.}
\begin{gather}
\begin{aligned}
    &               &  \exp\left(\eta \ad_{F}\right)\exp\left(\eta \ad_{G}\right)\left(p,q\right) &= \exp\left(\eta \ad_{\widetilde{H}_{\eta}}\right)\left(p,q\right)\nonumber\\
    &\iff&  \exp\left(\ad_{\eta F}\right)\exp\left(\ad_{\eta G}\right)\left(p,q\right) &= \exp\left(\ad_{\eta \widetilde{H}_{\eta}}\right)\left(p,q\right)\nonumber\\
    &\iff&  \ad_{\exp\left(\eta F\right)}\ad_{\exp\left(\eta G\right)}\left(p,q\right) &= \ad_{\exp\left(\eta \widetilde{H}_{\eta}\right)}\left(p,q\right),
\end{aligned}
\end{gather}
where $\mathcal{L}\left(F,G\right)$ is the free Lie algebra generated by $F$ and $G$ and $\mathbb{L}\left(F,G\right)$ is the Lie group generated by $F$ and $G$ under the group operation $f\oslash g=\nabla_{p}f\cdot\nabla_{q}g$, as defined in Section \ref{Liealgebra}. This follows due to the commuting diagram between $\mathbb{L}\left(F,G\right)$, $\mathcal{L}\left(F,G\right)$, $\textrm{Aut}\left(\mathbb{L}\left(F,G\right)\right)$, and the Lie algebra of $\textrm{Aut}\left(\mathbb{L}\left(F,G\right)\right)$ (Theorem 3.28 and Proposition 3.35 from \cite{Hall2015}), and how these apply to the adjoint map $\ad$. Furthermore, since $\ad$ is a group homomorphism, $\ad_{\exp\left(\eta F\right)}\ad_{\exp\left(\eta G\right)}=\ad_{\exp\left(\eta F\right)\exp\left(\eta G\right)}$, and hence the equations above is equivalent with
\begin{gather}
    \ad_{\exp\left(\eta F\right)\exp\left(\eta G\right)}\left(p,q\right)=\ad_{\exp\left(\eta \widetilde{H}_{\eta}\right)}\left(p,q\right).\label{Ad}
\end{gather}
If we take $\eta\widetilde{H}_{\eta}$ to be the Dynkin form of the Baker-Campbell-Hausdorff series applied to $\eta F$ and $\eta G$ (see Section \ref{Liealgebra}), then it follows that
\begin{gather}
    \eta\widetilde{H}_{\eta}=\log\left(\exp\left(\eta F\right)\exp\left(\eta G\right)\right),\nonumber
\end{gather}
from which we immediately deduce \eqref{Ad}. Therefore, we have shown \eqref{adjoint}, provided that $\widetilde{H}_{\eta}$ as defined using the BCH series converges.

Finally, by \eqref{adjoint}, when the BCH series for $F$ and $G$ converges, applying one iteration of symplectic Euler to $\left(p,q\right)$ is equivalent with solving the system of ODEs 
\begin{gather}
    \frac{dx}{ds}=\left.\ad_{\widetilde{H}_{\eta}}\left(\widetilde{p}\right)\right|_{\widetilde{p}=x}=-\nabla_{y}\widetilde{H}_{\eta}\left(x,y\right),\quad \frac{dy}{ds}=\left.\ad_{\widetilde{H}_{\eta}}\left(\widetilde{q}\right)\right|_{\widetilde{q}=y}=\nabla_{x}\widetilde{H}_{\eta}\left(x,y\right)\label{bracket}
\end{gather}
from $s=0$ to $s=\eta$ with initial condition $\left(x(0),y(0)\right)=\left(p,q\right)$ (see, e.g., Sections~III.4, IX.3, and~IX.9 of~\cite{Hairer2006}, and Section~6.3.2 of~\cite{Jose1998}). Thus, using \eqref{bracket}, we confirm that
\begin{gather}
    \frac{d}{ds}\widetilde{H}_{\eta}\left(x(s),y(s)\right)=\nabla_{x}\widetilde{H}_{\eta}\left(x(s),y(s)\right)\cdot\frac{dx}{ds}+\nabla_{y}\widetilde{H}_{\eta}\left(x(s),y(s)\right)\cdot\frac{dy}{ds}\nonumber\\
    =-\nabla_{x}\widetilde{H}_{\eta}\left(x(s),y(s)\right)\cdot\nabla_{y}\widetilde{H}_{\eta}\left(x(s),y(s)\right)+\nabla_{y}\widetilde{H}_{\eta}\left(x(s),y(s)\right)\cdot\nabla_{x}\widetilde{H}_{\eta}\left(x(s),y(s)\right)=0\nonumber\\
    \Rightarrow\widetilde{H}_{\eta}\left(x(0),y(0)\right)=\widetilde{H}_{\eta}\left(x(\eta),y(\eta)\right)\Rightarrow\boxed{\widetilde{H}_{\eta}\left(p,q\right)=\widetilde{H}_{\eta}\left(\exp\left(\eta \ad_{\widetilde{H}_{\eta}}\right)p,\exp\left(\eta \ad_{\widetilde{H}_{\eta}}\right)q\right).}\label{cons}
\end{gather}
Hence, Theorem \ref{conserved_Hamiltonian} follows.
\end{proof}

\subsection{The integral form of the BCH series}\label{IntegralThm1Proof}

For the Baker-Campbell-Hausdorff formula on \textit{any} normed Lie algebra $\left(L,\lnorm\cdot\rnorm\right)$, we propose the following equivalent formula to the Dynkin representation \eqref{predynkin} which involves the adjoint representation in a way similar to Theorem 5.5 in \cite{Miller1973} and Theorem 5.3 in \cite{Hall2015}. These aforementioned theorems involve derivations of similar integral formulas as Lemma \ref{Theorem_BCH_integral_form} below, which are related but not exactly the same. 

\begin{lemma}\label{Theorem_BCH_integral_form}
    For any $A,B\in L$ sufficiently small in norm---i.e., such that $\lnorm e^{t\cdot\ad_{A}}e^{t\cdot\ad_{B}}-I\rnorm_{\textrm{op}}<1$ or $\lnorm e^{t\cdot\ad_{A}}e^{t\cdot\ad_{B}}\rnorm_{\textrm{op}}<1$ for all $t\in\left[0,1\right]$,
\begin{gather}
    \log\left(e^{A}e^{B}\right)=\int_{0}^{1}{\log\left(e^{t\cdot\ad_{A}}e^{t\cdot\ad_{B}}\right)\left(e^{t\cdot\ad_{A}}e^{t\cdot\ad_{B}}-I\right)^{-1}\left(A+e^{t\cdot\ad_{A}}B\right)dt}.\label{th1}
\end{gather}
Moreover, the series expansion of \eqref{th1} converges absolutely.
\end{lemma}

 For the sake of completeness, we provide two unrelated proofs of Lemma \ref{Theorem_BCH_integral_form} separately in the end of this section, one of which is original and follows directly from the BCHD series \eqref{predynkin} and the other of which uses similar techniques as those used in the proofs of Theorem 5.5 in \cite{Miller1973} and Theorem 5.3 in \cite{Hall2015}. The former assumes $\lnorm e^{t\cdot\ad_{A}}e^{t\cdot\ad_{B}}-I\rnorm_{\textrm{op}}<1$ while the latter assumes $\lnorm e^{t\cdot\ad_{A}}e^{t\cdot\ad_{B}}\rnorm_{\textrm{op}}<1$. But before going into these two proofs, from Lemma \ref{Theorem_BCH_integral_form}, we have a few corollaries. The first one deals with how Lemma \ref{Theorem_BCH_integral_form} applies to the modified Hamiltonian (see \ref{MH_more} for details) of the symplectic Euler method. 

As the first proof in Appendix \ref{IntegralThm1Proof} demonstrates, the $\widetilde{H}_{\eta}$ series \ref{dynkin} directly implies Lemma \eqref{Theorem_BCH_integral_form} and vice-versa under absolute convergence, and either form could be more useful in computing $\widetilde{H}_{\eta}$ depending on the functions $F$ and $G$ in a particular case. Moreover, recall from Appendix \ref{Liealgebra} that the Poisson bracket (or the Poisson bracket rescaled by $\eta$) is itself a Lie bracket over the free Lie algebra $\mathcal{L}\left(F,G\right)$ generated by $F,G\in C^{\infty}\left(\R^{d}\right)$. Furthermore, equip $\mathcal{L}\left(F,G\right)$ with some function norm $\lnorm\cdot\rnorm$ which is continuous over the space. Then, for all $\varphi\in\mathcal{L}\left(F,G\right)$, we define the operator norm of $\varphi$, $\lnorm\varphi\rnorm_{\textrm{op}}$, as the smallest $L>0$ satisfying $\lnorm\left\{\varphi,\psi\right\}\rnorm\leq L\lnorm \psi\rnorm$ for all $\psi\in\mathcal{L}\left(F,G\right)$.

Thus, instantiating Lemma \ref{Theorem_BCH_integral_form} with the Poisson bracket rescaled by $\eta$, we have the following:

\begin{corollary}
\label{Corollary_SymplecticEuler_ConservedIntegral}
    The modified Hamiltonian $\widetilde{H}_{\eta}$ for symplectic Euler applied to the Hamiltonian $H=F\left(p\right)+G\left(q\right)$ has the form
\begin{gather}
    \widetilde{H}_{\eta}\left(p,q\right)=\frac{1}{\eta}\int_{0}^{\eta}{\log\left(e^{t\left\{\cdot,G\right\}}e^{t\left\{\cdot,F\right\}}\right)\left(e^{t\left\{\cdot,G\right\}}e^{t\left\{\cdot,F\right\}}-I\right)^{-1}\left(G+e^{t\left\{\cdot,G\right\}}F\right)dt}\label{co1}
\end{gather}
which converges absolutely whenever $\lnorm e^{t\left\{\cdot,G\right\}}e^{t\left\{\cdot,F\right\}}-I\rnorm_{\textrm{op}}<1$ or $\lnorm e^{t\left\{\cdot,G\right\}}e^{t\left\{\cdot,F\right\}}\rnorm_{\textrm{op}}<1$ for all $t\in\left[0,\eta\right]$.
\end{corollary}

 Even in situations where we can express the adjoint representations as matrices, computing the integrand for Lemma \ref{Theorem_BCH_integral_form} or Corollary \ref{Corollary_SymplecticEuler_ConservedIntegral} could be computationally prohibitive in practice. Hence, another option is to express the integrand as a single series, which Lemma \ref{Corollary_SymplecticEuler_ConservedIntegral_Compact} demonstrates below:

\begin{lemma}\label{Corollary_SymplecticEuler_ConservedIntegral_Compact}
Given the same setting as Lemma \ref{Theorem_BCH_integral_form} except assuming $\lnorm e^{t\cdot\ad_{A}}e^{t\cdot\ad_{B}}-I\rnorm_{\textrm{op}}<1$, we have
\begin{gather}
     \log\left(e^{A}e^{B}\right)=\int_{0}^{1}{\sum_{j=0}^{\infty}{\frac{\left(I-e^{t\cdot\ad_{A}}e^{t\cdot\ad_{B}}\right)^{j}}{j+1}\left(A+e^{t\cdot\ad_{A}}B\right)dt}}.\nonumber
\end{gather}
\end{lemma}

\begin{proof}
From p. 161 of \cite{Miller1973}, we have the following series expansion for $\\ \log\left(X\right)\left(X-I\right)^{-1}$ about $X=I$:
\begin{gather}
    \log\left(X\right)\left(X-I\right)^{-1}=\sum_{j=0}^{\infty}{\frac{\left(I-X\right)^{j}}{j+1}},\label{MatrixLog_Series}
\end{gather}
which converges whenever $\lnorm X-I\rnorm_{\textrm{op}}<1$. Then, use \eqref{MatrixLog_Series} to expand the integrand of \eqref{th1} in Lemma \ref{Theorem_BCH_integral_form}.

We can similarly use \eqref{MatrixLog_Series} to expand the integrand of \eqref{co1}.
\end{proof}

 Finally, by instantiating the Lie bracket in Lemma \ref{Corollary_SymplecticEuler_ConservedIntegral_Compact} with the Poisson bracket multiplied by $\eta$, we derive Theorem \ref{BCH_integral}:

\begin{theorem}
    \label{BCH_integral}
    For any $F\in C^{\infty}\left(\P\right)$ and $G\in C^{\infty}\left(\Q\right)$ such that $\lnorm e^{t\left\{\cdot,G\right\}}e^{t\left\{\cdot,F\right\}}-I\rnorm_{\textrm{op}}<1$ for all $t\in\left[0,1\right]$, 
\begin{gather}
\widetilde{H}_{\eta}\left(p,q\right)=\frac{1}{\eta}\int_{0}^{\eta}{\sum_{j=0}^{\infty}{\frac{\left(I-e^{t\left\{\cdot,G\right\}}e^{t\left\{\cdot,F\right\}}\right)^{j}}{j+1}\left(G+e^{t\left\{\cdot,G\right\}}F\right)dt}}.\label{BCHint}
\end{gather}
Moreover, the series expansion of \eqref{BCHint} converges absolutely.
\end{theorem}

 The integral form of the modified Hamiltonian (Theorem \ref{BCH_integral}) allows for direct calculation of a conserved quantity along the iterations of the algorithm in the quadratic case, as we will see in Appendix \ref{MHQC_BrutalCalc}.

\paragraph{An original proof of Lemma \ref{Theorem_BCH_integral_form}.}

\begin{proof}
We do the computations in reverse order, noting that the steps are reversible after establishing absolute convergence since we did not reorder the series and only rewrote the terms separately using integrals. We start with the BCHD series \ref{predynkin}, carefully separate out the sum when $n=1$ vs. $n\geq 2$, and rewrite the iterated Lie brackets in terms of their adjoint representations:
\begin{align}
    \log\left(e^{A}e^{B}\right) 
    =&\sum_{\substack{r\geq 0}}{\frac{\ad_{A}^{r}B}{\left(r+1\right)!}}+A\nonumber\\
    &+\sum_{n=2}^{\infty}{\frac{\left(-1\right)^{n-1}}{n}}\left(\sum_{\substack{r_{1}+s_{1}>0 \\ \cdots \\ r_{n-1}+s_{n-1}>0 \\ r_{n}\geq 0}}{\frac{\ad_{A}^{r_{1}}\ad_{B}^{s_{1}}\cdots\ad_{A}^{r_{n-1}}\ad_{B}^{s_{n-1}}\ad_{A}^{r_{n}}B}{\left(r_{1}+\cdots+r_{n-1}+s_{1}+\cdots+s_{n-1}+r_{n}+1\right)\cdot r_{n}!\prod_{i=1}^{n-1}{r_{i}!s_{i}!}}}\right.\nonumber\\ \nonumber \\
    &+\left.\sum_{\substack{r_{1}+s_{1}>0 \\ \cdots \\ r_{n-1}+s_{n-1}>0}}{\frac{\ad_{A}^{r_{1}}\ad_{B}^{s_{1}}\cdots \ad_{A}^{r_{n-1}}\ad_{B}^{s_{n-1}}A}{\left(r_{1}+\cdots+r_{n-1}+s_{1}+\cdots+s_{n-1}+1\right)\cdot\prod_{i=1}^{n-1}{r_{i}!s_{i}!}}}\right)
\nonumber
\end{align}
This further equals to:
\begin{align}
\log(e^A e^B)&=\sum_{\substack{r\geq 0}}{\int_{0}^{1}{\frac{\left(t\cdot\ad_{A}\right)^{r}B}{r!}dt}}+A+\sum_{n=2}^{\infty}{\frac{\left(-1\right)^{n-1}}{n}}\nonumber\\ 
&\left(\sum_{\substack{r_{1}+s_{1}>0 \\ \cdots \\ r_{n-1}+s_{n-1}>0 \\ r_{n}\geq 0}}{\int_{0}^{1}{\frac{\left(t\cdot\ad_{A}\right)^{r_{1}}}{r_{1}!}\frac{\left(t\cdot\ad_{B}\right)^{s_{1}}}{s_{1}!}\cdots\frac{\left(t\cdot\ad_{A}\right)^{r_{n-1}}}{r_{n-1}!}\frac{\left(t\cdot\ad_{B}\right)^{s_{n-1}}}{s_{n-1}!}\frac{\left(t\cdot\ad_{A}\right)^{r_{n}}}{r_{n}!}}dt}B\right.\nonumber\\ 
&\left.+\sum_{\substack{r_{1}+s_{1}>0 \\ \cdots \\ r_{n-1}+s_{n-1}>0}}{\int_{0}^{1}{\frac{\left(t\cdot\ad_{A}\right)^{r_{1}}}{r_{1}!}\frac{\left(t\cdot\ad_{B}\right)^{s_{1}}}{s_{1}!}\cdots\frac{\left(t\cdot\ad_{A}\right)^{r_{n-1}}}{r_{n-1}!}\frac{\left(t\cdot\ad_{B}\right)^{s_{n-1}}}{s_{n-1}!}}dtA}\right).\label{A}
\end{align}
 Under absolute convergence, we can then swap the integrals and summations in \eqref{A} and then rewrite the result using the exponential Taylor series $e^{X}=\sum_{r\geq 0}{X^{r}/r!}$ as follows:
\begin{align}
    \log\left(e^{A}e^{B}\right)=&\int_{0}^{1}{e^{t\cdot\ad_{A}}Bdt}+A\nonumber\\ 
    &\quad +\sum_{n=2}^{\infty}{\frac{\left(-1\right)^{n-1}}{n}}\left(\int_{0}^{1}{\left(\sum_{\substack{r+s>0}}{\frac{\left(t\cdot\ad_{A}\right)^{r}}{r!}\frac{\left(t\cdot\ad_{B}\right)^{s}}{s!}}\right)^{n-1}\sum_{r_{n}\geq 0}{\frac{\left(t\cdot\ad_{A}\right)^{r_{n}}}{r_{n}!}}dt}B\right.\nonumber\\ 
    &\quad \left.+\int_{0}^{1}{\left(\sum_{\substack{r+s>0}}{\frac{\left(t\cdot\ad_{A}\right)^{r}}{r!}\frac{\left(t\cdot\ad_{B}\right)^{s}}{s!}}\right)^{n-1}dt}A\right)\nonumber\\
    =&\int_{0}^{1}{\left(e^{t\cdot\ad_{A}}B+A\right)dt}\nonumber+\sum_{n=2}^{\infty}{\frac{\left(-1\right)^{n-1}}{n}}\left(\int_{0}^{1}{\left(e^{t\cdot\ad_{A}}e^{t\cdot\ad_{B}}-I\right)^{n-1}\left(e^{t\cdot\ad_{A}}B+A\right)dt}\right)\nonumber\\
    =&\sum_{n=1}^{\infty}{\frac{\left(-1\right)^{n-1}}{n}}\int_{0}^{1}{\left(e^{t\cdot\ad_{A}}e^{t\cdot\ad_{B}}-I\right)^{n-1}\left(e^{t\cdot\ad_{A}}B+A\right)dt}.\label{B}
\end{align}
Finally, under absolute convergence again and the Taylor series $\log\left(X+I\right)=\sum_{r\geq 1}{\left(-1\right)^{r+1}X^{r}/r}$ (which converges absolutely whenever $\lnorm X\rnorm_{\textrm{op}}<1$), \eqref{B} becomes
\begin{gather}
    \int_{0}^{1}{\sum_{n=1}^{\infty}{\frac{\left(-1\right)^{n-1}}{n}\left(e^{t\cdot\ad_{A}}e^{t\cdot\ad_{B}}-I\right)^{n-1}}\left(e^{t\cdot\ad_{A}}B+A\right)dt}=\nonumber\\ 
    \boxed{\int_{0}^{1}{\log\left(e^{t\cdot\ad_{A}}e^{t\cdot\ad_{B}}\right)\left(e^{t\cdot\ad_{A}}e^{t\cdot\ad_{B}}-I\right)^{-1}\left(e^{t\cdot\ad_{A}}B+A\right)dt},}\label{intform}
\end{gather}
as we wanted to show. Note finally that absolute convergence of the series expansion of \eqref{intform} is ensured by the bound on the exponentials of the adjoints stated in Lemma \ref{Theorem_BCH_integral_form}, thereby ensuring absolute convergence for the original BCH series.
\end{proof}

\paragraph{A classical proof of Lemma \ref{Theorem_BCH_integral_form}.}
\begin{proof}
Define the auxiliary function $C\left(t\right)=\log\left(e^{tA}e^{tB}\right)$ for $0\leq t\leq 1$. We have
\begin{gather}
    \frac{d}{dt}e^{C\left(t\right)}=\frac{d}{dt}\left[e^{tA}e^{tB}\right]=Ae^{tA}e^{tB}+e^{tA}e^{tB}B=Ae^{C\left(t\right)}+e^{C\left(t\right)}B.\label{C}
\end{gather}
Furthermore, by the work of \cite{Tuynman_1995} and \cite[][Theorem~5.4]{Hall2015}, since $C\left(t\right)$ is smooth, the leftmost expression in \eqref{C} can be rewritten as
\begin{gather}
    \frac{d}{dt}e^{C\left(t\right)}=e^{C\left(t\right)}\left(I-e^{-\ad_{C\left(t\right)}}\right)\ad_{C\left(t\right)}^{-1}\frac{dC(t)}{dt}.\label{D}
\end{gather}
Hence, by both \eqref{C} and \eqref{D}, and noting that the inverse of $e^{X}$ is $e^{-X}$, 
\begin{gather}
    \left(I-e^{-\ad_{C\left(t\right)}}\right)\ad_{C\left(t\right)}^{-1}\frac{dC}{dt}=e^{-C\left(t\right)}\frac{d}{dt}e^{C\left(t\right)}=e^{-C\left(t\right)}\left[Ae^{C\left(t\right)}+e^{C\left(t\right)}B\right]\nonumber\\
    =e^{-C\left(t\right)}Ae^{C\left(t\right)}+B.\label{E}
\end{gather}
For sufficiently small $A$ and $B$ (i.e., $A$ and $B$ such that $\lnorm e^{\ad_{C\left(t\right)}}\rnorm_{\textrm{op}}=\lnorm e^{t\ad_{A}}e^{t\ad_{B}}\rnorm_{\textrm{op}}<1$),\footnote{Indeed, because $\textrm{Ad}$ is a Lie group homomorphism \citep{Hall2015}, by Proposition 3.35 from \cite{Hall2015}, we have $e^{\ad_{C\left(t\right)}}=\textrm{Ad}_{e^{C\left(t\right)}}= \textrm{Ad}_{e^{tA}e^{tB}}=\textrm{Ad}_{e^{tA}}\textrm{Ad}_{e^{tB}}=e^{t\cdot\ad_{A}}e^{t\cdot\ad_{B}}$.} we can invert $\left(I-e^{-\ad_{C\left(t\right)}}\right)\ad_{C\left(t\right)}^{-1}$, and hence, \eqref{E} implies that
\begin{gather}
    \frac{dC}{dt}=\ad_{C\left(t\right)}\left(I-e^{-\ad_{C\left(t\right)}}\right)^{-1}\left[e^{-C\left(t\right)}Ae^{C\left(t\right)}+B\right].\label{F}
\end{gather}
Then, following the same steps used to go from (5.16) to (5.18) in \cite{Hall2015}, \eqref{F} simplifies to
\begin{gather}
    \frac{dC}{dt}=\log\left(e^{t\cdot\ad_{A}}e^{t\cdot\ad_{B}}\right)\left(e^{t\cdot\ad_{A}}e^{t\cdot\ad_{B}}-I\right)^{-1}e^{t\cdot\ad_{A}}e^{t\cdot\ad_{B}}\left[e^{-C\left(t\right)}Ae^{C\left(t\right)}+B\right].\label{almost}
\end{gather}
Using the fact that $\ad_{B}B=0$, $e^{t\cdot\ad_{A}}e^{t\cdot\ad_{B}}B=e^{t\cdot\ad_{A}}B$. Furthermore, using the identities from \cite[][p. 117]{Hall2015} and \cite[][Lemma 5.3]{Miller1973}, 
\begin{gather}
    e^{t\cdot\ad_{A}}e^{t\cdot\ad_{B}}e^{-C\left(t\right)}Ae^{C\left(t\right)}=e^{\ad_{C\left(t\right)}}e^{-C\left(t\right)}Ae^{C\left(t\right)}\nonumber\\ =e^{\ad_{C\left(t\right)}}e^{\ad_{-C\left(t\right)}}A=e^{\ad_{C\left(t\right)}}e^{-\ad_{C\left(t\right)}}A=A,\nonumber
\end{gather}
such that \eqref{almost} implies
\begin{gather}
    \frac{dC}{dt}=\log\left(e^{t\cdot\ad_{A}}e^{t\cdot\ad_{B}}\right)\left(e^{t\cdot\ad_{A}}e^{t\cdot\ad_{B}}-I\right)^{-1}\left(A+e^{t\cdot\ad_{A}}B\right).\label{ODE}
\end{gather}
Integrating \eqref{ODE} from $0$ to $1$ gives us
\begin{gather}
    C\left(1\right)-C\left(0\right)=\int_{0}^{1}{\log\left(e^{t\cdot\ad_{A}}e^{t\cdot\ad_{B}}\right)\left(e^{t\cdot\ad_{A}}e^{t\cdot\ad_{B}}-I\right)^{-1}\left(A+e^{t\cdot\ad_{A}}B\right)dt}.\nonumber
\end{gather}
Since $C\left(1\right)=\log\left(e^{A}e^{B}\right)$ and $C\left(0\right)=\log\left(e^{0}e^{0}\right)=\log\left(I\right)=0$, we are done. 
\end{proof}

\section{The modified Hamiltonian in the quadratic case}\label{MHQC_BrutalCalc}

As a test case to demonstrate the potency of the integral forms of the MH (Corollary \ref{Corollary_SymplecticEuler_ConservedIntegral} and Theorem \ref{BCH_integral}), we now study the particular, fundamental case when $F=ap^2, G=bq^2$, and $d=1$---the quadratic case. In this case, the function 
$$\mathsf{S}\left(p,q\right)=ap^{2}+bq^{2}-2\eta abpq$$ 
is a closed-form conserved quantity throughout the iterations of the symplectic Euler method \eqref{init}, i.e.,
\begin{gather}
    \mathsf{S}\left(p_{n},q_{n}\right)=\mathsf{S}\left(p_{0},q_{0}\right) \quad \forall n\nonumber
\end{gather} 
for any stepsize $\eta>0$. We encourage readers to verify this fact for themselves or see \cite{pmlr-v125-bailey20a,Wibisono2022} for further details, which demonstrates how the above results also generalize for $d>1$. In theory, the function $\mathsf{S}\left(p, q\right)$ should also appear in the modified Hamiltonian. And this is indeed the case: By applying symbolic calculation programs to calculate the first terms in the modified Hamiltonian, expressed as the series \eqref{dynkin} when $f=ap^2$, $g=bq^2$, we reveal the following pattern in the first few terms of the series:
\begin{gather}\label{quadratic_MH_series}
    \widetilde{H}_{\eta}(p, q) = \mathsf{S}\left(p, q\right)\left(1+\frac{2ab\eta^2}{3}+\frac{8(ab)^2\eta^{4}}{15}+\frac{16(ab)^{3}\eta^6}{35}+\frac{128(ab)^{4}\eta^8}{215}+\cdots\right). 
\end{gather} 
This pattern encourages us to calculate the limit of the series using the integral form of the BCH formula (Theorem \ref{BCH_integral}), i.e., 
\begin{gather}
    \widetilde{H}_{\eta}\left(p,q\right)=\frac{1}{\eta}\int_{0}^{\eta}{\sum_{j=0}^{\infty}{\frac{\left(I-e^{t\left\{\cdot,G\right\}}e^{t\left\{\cdot,F\right\}}\right)^{j}}{j+1}\left(G+e^{t\left\{\cdot,G\right\}}F\right)dt}},\nonumber
\end{gather}
directly. To the best of our knowledge, this has not been done in prior literature.

It is highly inefficient to compute the terms of the original series formula \eqref{dynkin} for $\widetilde{H}_{\eta}$ directly. The most efficient algorithm for symbolic calculation that the authors could implement takes several hours to obtain the first 13 terms. Thus, computing $\widetilde{H}_{\eta}$ via the integral formula could demonstrate its efficiency when compared to the traditional Dynkin series formula. 

Furthermore, as \eqref{quadratic_MH_series} implies, the modified Hamiltonian in this case is the product of a  conserved quantity as a function of $p, q$ and a power series of $\eta$. For $\eta$ large enough, the power series diverges, yet the conserved quantity still remains conserved along the iterations. Thus, this is a direct example demonstrating that a conserved quantity could persist even when the Dynkin series \eqref{dynkin} diverges, which could perhaps be explained by analytic continuation. This could motivate further exploring the Dynkin series from the perspective of renormalization theory for more general functions (especially when no closed-form is known) to take care of cases when the Dynkin series diverges.

\subsection{Functional space of the Iterated Poisson Brackets}\label{E.1_space_IPBs}

We now apply Theorem \ref{BCH_integral} to compute the modified Hamiltonian in this case. We first have the following lemma regarding the structure of the Iterated Poisson Brackets (IPBs):
\begin{lemma}\label{Lemma_Structure_Quadratic}
    $F$, $G$, and any IPBs generated by $F$ and $G$ has the following quadratic form:
    $$
    \{\{\{...\{F,G\},F\},G\},..., \} = Ap^2+Bpq+Cq^2=
    \begin{bmatrix}
        p^2 & pq & q^2
    \end{bmatrix}
    \begin{bmatrix} A \\ B \\ C
    \end{bmatrix}
    $$
    for some constants $A,B,C\in\R$ dependent on the number of Poisson brackets and the permutation of the functions $F,G$ in the IPB.
\end{lemma}

\begin{proof}
    This can be shown easily via induction. Since $\{G, F\}=-\{F, G\}=4adpq$, the base case holds. Suppose this is true for an IPB $I$. Then, $\{F, I\} = -\nabla_{p}F\frac{\partial I}{\partial q}= -(2ap)(2Cq+Bp)=-2aBp^2+4aCpq$. Similarly, $\{G, I\} = \nabla_{q}G\frac{\partial I}{\partial p}=(2dq)(2Ap+Bq)=2Bdq^2+4dApq$. 
\end{proof}

 We remark that it is not the usual case that IPBs lie in a finite-dimensional function space. In most cases, the function space will be infinite, as we will see in the next section, when we talk about the cases in which the convergence criteria fail to hold.

Thus, the functional space endowed by the Poisson bracket in this case is simply the space of pure second-order polynomials of $p$ and $q$. In this case, there exists a matrix expression of the adjoint representations $\ad_{F}\coloneqq \left\{\cdot,F\right\}$ and $\ad_{G}\coloneqq \left\{\cdot,G\right\}$ on the free Lie algebra $\mathcal{L}\left(F,G\right)$ generated by $F$ and $G$ and endowed with the Poisson bracket, as suggested by the following lemma:

\begin{lemma}
    Given the previous representation, we have the matrix representations
    $$
    \ad_F = \begin{bmatrix}
0 & 0 & 2a \\
0 & 0 & 0  \\
0 & 4a & 0 \\
    \end{bmatrix}, \quad \ad_G =
    -\begin{bmatrix}
0 & 0 & 0 \\
0 & 0 & 2b\\
4b & 0 & 0 \\
\end{bmatrix}
    $$
    when we choose the basis $\left\{p^{2},q^{2},pq\right\}$ for $\mathcal{L}\left(F,G\right)$.
\end{lemma}
\begin{proof} From Lemma \ref{Lemma_Structure_Quadratic}, $\varphi\in\textrm{Span}\left\{p^{2},q^{2},pq\right\}\Rightarrow\left\{\varphi,F\right\},\left\{\varphi,G\right\}\in\textrm{Span}\left\{p^{2},q^{2},pq\right\}$, but we desire to express this implication more explicitly. Let $\varphi\in\textrm{Span}\left\{p^{2},q^{2},pq\right\}$, in which case $\varphi\left(p,q\right)=c_{2}p^{2}+c_{1}q^{2}+c_{0}pq$ for some constants $c_{2},c_{1},c_{0}\in\R$, which can be represented in $\mathcal{L}\left(F,G\right)$ as $\left(c_{2},c_{1},c_{0}\right)$ when using the basis $\left\{p^{2},q^{2},pq\right\}$. Routine computations show that
\begin{gather}
    \left\{\varphi,p^{2}\right\}=2 c_{0} p^{2} + 4 c_{1} p q,\quad \left\{\varphi,q^{2}\right\}=- 2 c_{0} q^{2} - 4 c_{2} p q,\quad \left\{\varphi,pq\right\}=2 c_{1} q^{2} - 2 c_{2} p^{2}.\label{ids}
\end{gather}
Since the Poisson bracket is linear in each component, \eqref{ids} implies that
\begin{align}
    \ad_{F}\left(\varphi\right)&=\left\{\varphi,F\right\}=\left\{\varphi,ap^{2}\right\}= \left(2a\right)c_{0}p^{2} + \left(4a\right)c_{1}p q,\nonumber\\
    \ad_{G}\left(\varphi\right)&=\left\{\varphi,G\right\}=\left\{\varphi,dq^{2}\right\}=\left(- 2d\right)c_{0} q^{2}+\left(- 4 d\right)c_{2} p q,\nonumber
\end{align}
i.e., 
\begin{gather}
    \ad_{F}\left(c_{2},c_{1},c_{0}\right)=\left(2ac_{0},0,4ac_{1}\right),\qquad \ad_{G}\left(c_{2},c_{1},c_{0}\right)=-\left(0,2dc_{0},4dc_{2}\right)\nonumber
\end{gather}
when expressed in the basis $\left\{p^{2},q^{2},pq\right\}$.\end{proof}
 Note that the existence of matrix expressions for $\ad_{F}$ and $\ad_{G}$ should already be expected from Ado’s theorem \citep{Hall2015}, which states that every finite-dimensional real Lie algebra (e.g., $\mathcal{L}\left(F,G\right)$ for $F$ and $G$ quadratic) is isomorphic to an algebra of matrices.

Since the adjoint matrix in this case is of finite dimension and nilpotent, we are motivated to compute the MH in this quadratic case by evaluating the formula proposed in Theorem \ref{BCH_integral} directly. Doing so gives us Theorem \ref{Thm_Quadratic_MH_multivar} and the observations that follow thereafter.

\subsection{Proof of Theorem \ref{Thm_Quadratic_MH_multivar}}\label{bigcomputation}

\subsubsection{Preliminary work}
We start by manipulating Theorem \ref{BCH_integral} to make calculations a bit easier. Recall that in last section we derived the matrix representation for $\text{ad}_G$ and $\text{ad}_F$, if we have access to the diagonalization of $e^{t\{\cdot, G\}}e^{t\{\cdot, F\}}=e^{t\text{ad}_G}e^{t\text{ad}_F}=PDP^{-1}$, we could dramatically simplify the integral form of the BCH series for the modified Hamiltonian \eqref{BCHint} into the following form:
\begin{align}
\widetilde{H}_{\eta}\left(p,q\right)=\frac{1}{\eta}\int_{0}^{\eta}{P\log\left(D\right)\left(D-I\right)^{-1}P^{-1}\left(G+e^{t\left\{\cdot,G\right\}}F\right)dt}.\label{modify}
\end{align}
Motivated by the aforementioned simplification, we start by calculating the matrix exponentials of the matrix expression of the adjoint representations and calculate its 
diagonalization. With the help of SymPy \citep{SymPy}, we have the following:
\begin{gather}
    e^{t \ad_G} = \begin{bmatrix}1 & 0 & 0\\4 b^{2} t^{2} & 1 & - 2 b t\\- 4 bt & 0 & 1\end{bmatrix},\qquad
e^{t \ad_F} = \begin{bmatrix}1 & 4 a^{2}t^{2} & 2 at\\0 & 1 & 0\\0 & 4 at & 1\end{bmatrix}.
\end{gather} 
We now calculate their product:
\begin{gather}
    e^{t\cdot \ad_G}e^{t\cdot \ad_F}=\begin{bmatrix}
    1 & 8a^2t^2 & 2at\\
    8b^2t^2 & 1+8abt^2+64a^2b^2t^4 & 2bt+16a^2bt^3\\
    -4bt &4at-32a^2bt^3 &1-8abt^2
\end{bmatrix}.
\end{gather} 
With some help from SymPy~\citep{SymPy}, we now calculate $P$ and $D$ in the diagonalization $e^{t\cdot \ad_G}e^{t\cdot \ad_F}=PDP^{-1}$:
\begin{gather}
    P=\left[\begin{matrix}- \frac{1}{2 b t} & \frac{- a b t - i\sqrt{a b \left(1-a b t^{2}\right)}}{2 b} & \frac{- a b t + i\sqrt{a b \left(1-a b t^{2}\right)}}{2 b}\\- \frac{1}{2 a t} & \frac{- a b t + i\sqrt{a b \left(1-a b t^{2}\right)}}{2 a} & \frac{- a b t - i\sqrt{a b \left(1-a b t^{2}\right)}}{2 a}\\1 & 1 & 1\end{matrix}\right].\label{P}
\end{gather} 
Here $i = \sqrt{-1}$ is the imaginary unit.

We can thereby compute $P^{-1}(G+e^{t ad_G}F)$ as follows:
\begin{gather}
    P^{-1}(G+e^{tad_{G}}F)=\left[\begin{matrix}\frac{2 a b t}{a b t^{2} - 1}\\\frac{a b t \left(3 a b t - 4 a^{2} b^{2} t^{3} + i\left(4 a b t^{2}-1\right)\sqrt{a d \left(1-a b t^{2} \right)}\right)}{a^{2} b^{2} t^{3} - a b t+i\left(1-a b t^{2}\right)\sqrt{a b \left(1-a b t^{2} \right)}}\\ i\frac{a b t \left(2 a b t - 2 a^{2} b^{2} t^{3} - i\left(2 a b t^{2}- 1\right)\sqrt{a b \left(1-a b t^{2}\right)}\right)}{\sqrt{a b \left(1-a b t^{2}\right)} \left(1-a b t^{2}\right)}\end{matrix}\right].\label{now}
\end{gather} 
On the other hand,
\begin{gather}
    D=\textrm{diag}\left(\begin{bmatrix}1  \\ 8 a^{2} b^{2} t^{4} - 8 a b t^{2} - 4 t \sqrt{a b \left(a b t^{2} - 1\right)} \left(2 a b t^{2} - 1\right) + 1 \\ 8 a^{2} b^{2} t^{4} - 8 a b t^{2} + 4 t \sqrt{a b \left(a b t^{2} - 1\right)} \left(2 a b t^{2} - 1\right) + 1\end{bmatrix}\right).\label{diagonal}
\end{gather}
Thus, we use \eqref{diagonal} to calculate $\log(D)(D-I)^{-1}$ by applying the series of matrix logarithm \eqref{MatrixLog_Series}:
\begin{multline}
    \log(D)(D-I)^{-1} = \sum_{j=0}^{\infty}\frac{(I-D)^j}{j+1} \\
    =\textrm{diag}
    \begin{bmatrix}
        1 \\ 
        \sum_{j=0}^{\infty}{\frac{4^j t^j}{j+1}(- 2a^2b^2t^3 + 2abt + \sqrt{ab(ab t^2-1)} (2abt^2-1))^{j}}\\ 
        \sum_{j=0}^{\infty}{\frac{4^jt^j}{j+1}(- 2a^2b^2t^3+2adt-\sqrt{ab(abt^2-1)}(2abt^2-1))^j}
    \end{bmatrix}.\label{diag}
\end{multline}
Note that by assumption, $t\leq \eta<1/\sqrt{ab}$, which implies $t^{2}<1/ab$, or $abt^{2}-1<0$. Thus, \eqref{diag} can be rewritten in terms of complex numbers: 
\begin{gather}
    \log\left(D\right)\left(D-I\right)^{-1}=\textrm{diag}\left(
    \begin{bmatrix}
        1\\
        \sum_{j=0}^{\infty}{\frac{4^{j} t^{j}}{j+1}\left(2 a b t- 2 a^{2} b^{2} t^{3} + i\left(2 a b t^{2} - 1\right)\sqrt{a b \left(1-a b t^{2}\right)} \right)^{j}}\\
        \sum_{j=0}^{\infty}{\frac{4^{j} t^{j}}{j+1}\left(2 a b t- 2 a^{2} b^{2} t^{3} - i\left(2 a b t^{2} - 1\right)\sqrt{a b \left(1-a b t^{2}\right)} \right)^{j}}
    \end{bmatrix}\right)
    =\nonumber\\
    -\textrm{diag}\left(
    \begin{bmatrix}
        -1\\
        \frac{1}{4t\left(2 a b t- 2 a^{2} b^{2} t^{3} + i\left(2 a b t^{2} - 1\right)\sqrt{a b \left(1-a b t^{2}\right)} \right)}\log\left[1-4t\left(2 a b t- 2 a^{2} b^{2} t^{3} + i\left(2 a b t^{2} - 1\right)\sqrt{a b \left(1-a b t^{2}\right)} \right)\right] \\
        \frac{1}{4t\left(2 a b t- 2 a^{2} b^{2} t^{3} - i\left(2 a b t^{2} - 1\right)\sqrt{a b \left(1-a b t^{2}\right)} \right)}\log\left[1-4t\left(2 a b t- 2 a^{2} b^{2} t^{3} - i\left(2 a b t^{2} - 1\right)\sqrt{a b \left(1-a b t^{2}\right)} \right)\right]
    \end{bmatrix}
    \right),\label{diagdiag}
\end{gather}
using the Taylor series for $-\log\left(1-x\right)=\sum_{k=1}^{\infty} \frac{x^{k}}{k} =\sum_{j=0}^{\infty} \frac{x^{j+1}}{j+1}
= x \sum_{j=0}^{\infty}\frac{x^{j}}{j+1}$. Note in particular that this means we are using the principal branch of the logarithm.

\subsubsection{Simplification of the integrand}

We now start substituting the variables we calculated into 
 Theorem \ref{Corollary_SymplecticEuler_ConservedIntegral}. We begin by multiplying \eqref{diagdiag} and \eqref{now} to get:
\begin{align}
    &\log\left(D\right)\left(D-I\right)^{-1}P^{-1}(G+e^{t \ad_{G}}F)\nonumber \\
    &=-\begin{bmatrix} \frac{2 a b t}{1 - a b t^{2}} \\ 
    \frac{a b \log\left[1-4t\left(2 a b t- 2 a^{2} b^{2} t^{3} + i\left(2 a b t^{2} - 1\right)\sqrt{a b \left(1-a b t^{2}\right)} \right)\right]}{4\left(2 a b t- 2 a^{2} b^{2} t^{3} + i\left(2 a b t^{2} - 1\right)\sqrt{a b \left(1-a b t^{2}\right)} \right)}\frac{3 a b t - 4 a^{2} b^{2} t^{3} + i\left(4 a b t^{2}-1\right)\sqrt{a b \left(1-a b t^{2} \right)}}{a^{2} b^{2} t^{3} - a b t+i\left(1-a b t^{2}\right)\sqrt{a b \left(1-a b t^{2} \right)}}\\
    i\frac{a b\log\left[1-4t\left(2 a b t- 2 a^{2} b^{2} t^{3} - i\left(2 a b t^{2} - 1\right)\sqrt{a b \left(1-a b t^{2}\right)} \right)\right]}{4\sqrt{a b \left(1-a b t^{2}\right)} \left(1-a b t^{2}\right)}\end{bmatrix}.\label{longone}
\end{align}
The second, rightmost factor in the second row of \eqref{longone} simplifies to
\begin{align}
    \frac{-i}{4\left(1-abt^2\right)\sqrt{ab(1-abt^2)}}.\label{secondsecond}
\end{align}
Thus, we plug \eqref{secondsecond} back into \eqref{longone} and simplify \eqref{longone} to
\begin{align}
    \frac{ab}{4i\left(1-a b t^2\right)\sqrt{a b (1 - a b t^2)}}\left[\begin{matrix}-8it\sqrt{ab(1-abt^2)} \\ -\log\left[1-4t\left(2 a b t\left(1- a b t^{2}\right) + i\left(2 a b t^{2} - 1\right)\sqrt{a b \left(1-a b t^{2}\right)} \right)\right]\\ \log\left[1-4t\left(2 a b t\left(1- a b t^{2}\right) - i\left(2 a b t^{2} - 1\right)\sqrt{a b \left(1-a b t^{2}\right)} \right)\right]\end{matrix}\right].\label{simple}
\end{align}
We now let $g(t)=\sqrt{ab(1-abt^2)}$. This allows us to rewrite \eqref{P} and \eqref{simple} in a compact way:
\begin{align}
    P
    &=\left[\begin{matrix}- \frac{1}{2 b t} & \frac{- a b t - ig\left(t\right)}{2 b} & \frac{- a b t + ig\left(t\right)}{2 b}\\- \frac{1}{2 a t} & \frac{- a b t + ig\left(t\right)}{2 a} & \frac{- a b t - ig\left(t\right)}{2 a}\\1 & 1 & 1\end{matrix}\right]\label{P_simple}\\
\log\left(D\right)\left(D-I\right)^{-1}P^{-1}(G+e^{tab_{G}}F)&=\frac{a^2b^2}{4ig^3(t)}\left[\begin{matrix} -8itg(t) \\ -\log\left[1-4t\left(2 t g^{2}\left(t\right) + i\left(2 a b t^{2} - 1\right)g\left(t\right) \right)\right]\\ \log\left[1-4t\left(2 t g^{2}\left(t\right) - i\left(2 a b t^{2} - 1\right)g\left(t\right) \right)\right]\end{matrix}\right]. \label{Integrand_vector_simple}
\end{align}
 
Using \eqref{P_simple} and \eqref{Integrand_vector_simple} to assemble the entire integrand of \eqref{modify}, we have
\begin{align}
    &P\log\left(D\right)\left(D-I\right)^{-1}P^{-1}(G+e^{tad_{G}}F) \nonumber \\
    &=\frac{a^2b^2}{4ig^3(t)}
    \left[\begin{matrix}- \frac{1}{2 b t} & \frac{- a b t - ig\left(t\right)}{2 b} & \frac{- a b t + ig\left(t\right)}{2 b}\\- \frac{1}{2 a t} & \frac{- a b t + ig\left(t\right)}{2 a} & \frac{- a b t - ig\left(t\right)}{2 a}\\1 & 1 & 1\end{matrix}\right]\left[\begin{matrix} -8itg(t) \\ -\log\left[1-4t\left(2 t g^{2}\left(t\right) + i\left(2 a b t^{2} - 1\right)g\left(t\right) \right)\right]\\ \log\left[1-4t\left(2 t g^{2}\left(t\right) - i\left(2 a b t^{2} - 1\right)g\left(t\right) \right)\right]\end{matrix}\right].\label{hm}
\end{align}
Let $z\left(t\right)\coloneqq a b t + ig\left(t\right)$ and $\nu\left(t\right)\coloneqq \log\left[1-4t\left(2 t g^{2}\left(t\right)-i\left(2 a b t^{2} - 1\right)g\left(t\right) \right)\right]$. Then, \eqref{hm} implies that
\begin{align}
    &P\log\left(D\right)\left(D-I\right)^{-1}P^{-1}(G+e^{tad_{G}}F)
    =\frac{a^{2}b^{2}}{2g^{3}\left(t\right)}\left[\begin{matrix}\frac{1}{2b}\Im\left\{z\left(t\right)\overline{\nu\left(t\right)}\right\} +\frac{2g(t)}{b}\\ \frac{1}{2a}\Im\left\{\overline{z\left(t\right)\nu\left(t\right)}\right\} +\frac{2g(t)}{a}\\  \Im\left\{\nu\left(t\right)\right\}-4tg(t)\end{matrix}\right],\label{evenlonger}
\end{align}
where $\Im\left(c\right)$ is the imaginary component of any $c\in\mathbb{C}$.

Noticing the recurrence of the expression in \eqref{evenlonger}, we are motivated to expand out $\nu(t)$ to
\begin{align}
    \nu\left(t\right) 
    &=\log\left[1-8abt^2(1-abt^2)+4it(2abt^2-1)\sqrt{ab(1-abt^2)}\right].\label{mod}
\end{align} 
By routine algebraic manipulations, we can show that the complex number within the logarithm of \eqref{mod} has modulus one:
\begin{align}
    &(1-8abt^2(1-abt^2))^2+(4t(2abt^2-1)\sqrt{ab(1-abt^2)})^2 
    =1.\label{modone}
\end{align} Thus, in light of \eqref{mod} and \eqref{modone}, we can use Euler's formula, double angle formulae, and then trigonometry to rewrite $\nu(t)$ as the following:
\begin{align}
    \nu\left(t\right)=4i\arcsin(\sqrt{abt^2}).\label{iota(t)}
\end{align}
Hence, by \eqref{evenlonger} and \eqref{iota(t)}, we can simplify the integrand of \eqref{modify} as follows:
\begin{align}
\widetilde{H}_{\eta}\left(p,q\right)&=\frac{a^{2}b^{2}}{\eta}\left[\begin{matrix} p^{2} & q^{2} & pq \end{matrix}\right]\int_{0}^{\eta}\frac{1}{2g^{3}\left(t\right)}\left[\begin{matrix} 2at\arcsin(\sqrt{abt^2})+\frac{1}{b}2g(t)\\
    2bt\arcsin(\sqrt{abt^2})+\frac{1}{a}2g(t)\\
    -4\arcsin(\sqrt{abt^2})-4tg(t)\\
    \end{matrix}\right]dt.
    \label{FullySimplifiedIntegrand}
\end{align}
Finally, we split the integrand in \eqref{FullySimplifiedIntegrand} into the two terms shown in the vector and integrate them separately with the help of Wolfram Mathematica \citep{Mathematica}. Doing so and adding together the results appropriately gives us the following:
\begin{align}
    \widetilde{H}_{\eta}\left(p,q\right) 
&=\frac{a^{2}b^{2}}{\eta}\left[\begin{matrix} p^{2} & q^{2} & pq \end{matrix}\right]
\left[\begin{matrix}
    \frac{\arcsin(\sqrt{ab\eta^2})}{4ab^2\sqrt{ab(1-ab\eta^2)}}\\
    \frac{\arcsin(\sqrt{ab\eta^2})}{4a^2b\sqrt{ab(1-ab\eta^2)}}\\
     -\frac{2ab\eta}{4}\frac{\arcsin(\sqrt{ab\eta^2})}{a^2b^2\sqrt{ab(1-ab\eta^2)}}\\
\end{matrix}\right]\nonumber \\
&=\frac{1}{\sqrt{ab\eta^2(1-ab\eta^2)}}\arcsin(\sqrt{ab\eta^2})(ap^2+bq^2-2abpq\eta). \label{E.2_MH_monoquad_final} 
\end{align}

When $ab<0$, we first show the following proper definition of inverse sine functions. Suppose we are interested in finding $y=\arcsin{ix}$, which we can write equivalently as follows:
\begin{align}
    y = \arcsin{ix} = i\arcsinh (-x).\nonumber
\end{align}

\noindent Then the modified Hamiltonian takes the following form:
\begin{align}
    \frac{\arcsin(\sqrt{ab\eta^2})}{\sqrt{ab\eta^2(1-ab\eta^2)}}(ap^2+bq^2-2abpq\eta)=\frac{\arcsinh(\sqrt{-ab\eta^2})}{\sqrt{-ab\eta^2(1-ab\eta^2)}}(ap^2+bq^2-2abpq\eta).\nonumber
\end{align}

\subsubsection{Convergence radius}

The argument above is only valid if the conditions of Theorem \ref{BCH_integral} are actually satisfied. Thus, we need to ensure when $\lnorm e^{t\cdot \ad_G}e^{t\cdot \ad_F}-I\rnorm_{\textrm{op}}=\sigma_{\max}\left(D-I\right)<1$, or equivalently, when 
\begin{subequations}\label{cool}
\begin{gather}
    \left|-8abt^2(1-abt^2)-4it(2abt^2-1)\sqrt{ab(1-abt^2)}\right|<1,  \\ \left|-8abt^2(1-abt^2)+4it(2abt^2-1)\sqrt{ab(1-abt^2)}\right|<1
\end{gather}
\end{subequations}
for all $t\in\left[0,\eta\right]$, as taken directly from from \eqref{diagonal}. Noting the equivalence of forms \eqref{mod} and \eqref{iota(t)} for $\nu(t)$, \eqref{cool} can be rewritten as
\begin{gather}
    \left|e^{\overline{\nu(t)}}-1\right|=\left|e^{-4i\arcsin z}-1\right|<1, \nonumber \\ \left|e^{\nu(t)}-1\right|=\left|e^{4i\arcsin z}-1\right|<1,\nonumber
\end{gather}
respectively, where we have let $z\coloneqq \sqrt{abt^2}$. From here, we split the work into two cases based on the sign of $ab$.

\paragraph{Case 1: $ab\geq 0$.}

In this case, $\arcsin z$ is real, and since $\arcsin$ only takes arguments from $-1$ to $1$, we get a contradiction when $z>1$. Since $e^{-4i\arcsin z} = \overline{e^{4i\arcsin z}}$, $|e^{-4i\arcsin z}-1|<1$ if and only if $|e^{4i\arcsin z}|<1$. When $0<4\arcsin{z}<\frac{\pi}{3}$, we have that $|e^{4i\arcsin z}-1|<1$. Similarly, when $\frac{5\pi}{3}<4\arcsin{z}<2\pi$, $|e^{4i\arcsin z}-1|<1$ again. Thus, the absolute convergence criteria in Theorem \ref{BCH_integral} is satisfied whenever $z \in \left(0, \sin(\frac{\pi}{12}) \right) \cup \left(\sin(\frac{5\pi}{12}), 1\right)$. We now show that this actually implies that the modified Hamiltonian \eqref{dynkin} is absolutely convergent whenever $|z| \in (0, 1)$. 

For any $\eta \in (0, \sqrt{\frac{1}{ab}})$, we pick $\eta' = \max(\sqrt{\sin{\frac{11\pi}{24}}}, \eta)$. By construction $\eta' \in \left(\sin(\frac{5\pi}{12}), 1\right)$ and thus the series \eqref{dynkin} converges absolutely for $\eta'$. We now compare the terms of the two series:
\begin{multline}
    \left| \frac{\eta^{r_{1}+\dots+r_{n}+s_{1}+\dots+s_{n}-1}\left\{G^{r_{1}}F^{s_{1}}\cdots G^{r_{n}}F^{s_{n}}\right\}(p,q)}{n\left(r_{1}+\dots+r_{n}+s_{1}+\dots+s_{n}\right)\prod_{i=1}^{n}{r_{i}!s_{i}!}} \right| \\
    \leq \left| \frac{(\eta')^{r_{1}+\dots+r_{n}+s_{1}+\dots+s_{n}-1}\left\{G^{r_{1}}F^{s_{1}}\cdots G^{r_{n}}F^{s_{n}}\right\}(p,q)}{n\left(r_{1}+\dots+r_{n}+s_{1}+\dots+s_{n}\right)\prod_{i=1}^{n}{r_{i}!s_{i}!}}\right|.\nonumber
\end{multline}
By the dominated convergence theorem, the modified Hamiltonian \eqref{dynkin} converges absolutely for $\eta$ chosen above. 

\paragraph{Case 2: $ab<0$.}
When $ab<0$, we let $F'=|a|p^2, G'=|b|q^2$. From Appendix \ref{E.1_space_IPBs}, we know that $\left\{G^{r_{1}}F^{s_{1}}\cdots G^{r_{n}}F^{s_{n}}\right\}(p,q)$ is always a monomial of $p, q$ and $a, b$, regardless of the signs of $p$ and $q$. We have the following equality:
\begin{align}
     &\left|\frac{\eta^{r_{1}+\dots+r_{n}+s_{1}+\dots+s_{n}-1}\left\{G^{r_{1}}F^{s_{1}}\cdots G^{r_{n}}F^{s_{n}}\right\}(p,q)}{n\left(r_{1}+\dots+r_{n}+s_{1}+\dots+s_{n}\right)\prod_{i=1}^{n}{r_{i}!s_{i}!}} \right| \nonumber \\
     &\qquad\qquad= \frac{\eta^{r_{1}+\dots+r_{n}+s_{1}+\dots+s_{n}-1}\left|\left\{G^{r_{1}}F^{s_{1}}\cdots G^{r_{n}}F^{s_{n}}\right\}(p,q)\right|}{n\left(r_{1}+\dots+r_{n}+s_{1}+\dots+s_{n}\right)\prod_{i=1}^{n}{r_{i}!s_{i}!}} \nonumber \\
     &\qquad\qquad= \frac{\eta^{r_{1}+\dots+r_{n}+s_{1}+\dots+s_{n}-1}\left|\left\{G'^{r_{1}}F'^{s_{1}}\cdots G'^{r_{n}}F'^{s_{n}}\right\}(p,q)\right|}{n\left(r_{1}+\dots+r_{n}+s_{1}+\dots+s_{n}\right)\prod_{i=1}^{n}{r_{i}!s_{i}!}} \nonumber \\
     &\qquad\qquad=\left|\frac{\eta^{r_{1}+\dots+r_{n}+s_{1}+\dots+s_{n}-1}\left\{G'^{r_{1}}F'^{s_{1}}\cdots G'^{r_{n}}F'^{s_{n}}\right\}(p,q)}{n\left(r_{1}+\dots+r_{n}+s_{1}+\dots+s_{n}\right)\prod_{i=1}^{n}{r_{i}!s_{i}!}} \right|. \nonumber
\end{align}
Since the modified Hamiltonian \eqref{dynkin} converges absolutely whenever $\sqrt{|a||b|t^2} < 1$, it converges absolutely for $F'$ and $G'$. Since the absolute value of the terms is the same, the convergence is also absolute when $ab<0$, but except when $\sqrt{|ab|t^2} < 1$. 

Figure \ref{fig:2} shows the graph of the function $\eta \mapsto \frac{\arcsin\left(\sqrt{ab\eta^2}\right)}{\sqrt{ab\eta^2(1-ab\eta^2)}}$ that appears in the final expression \eqref{E.2_MH_monoquad_final} of the modified Hamiltonian. Note that the Taylor expansion of $\eta \mapsto \frac{\arcsin\left(\sqrt{ab\eta^2}\right)}{\sqrt{ab\eta^2(1-ab\eta^2)}}$ is as follows \citep{Mathematica}:
\begin{gather}
    \frac{\arcsin\left(\sqrt{ab\eta^2}\right)}{\sqrt{ab\eta^2(1-ab\eta^2)}}= 1+\frac{2(\eta\sqrt{ab})^2}{3}+\frac{8(\eta\sqrt{ab})^4}{15}+\frac{16(\eta\sqrt{ab})^6}{35}+\frac{128(\eta\sqrt{ab})^8}{215}+\cdots ,\nonumber
\end{gather}
which is precisely what the simplified series formula \eqref{quadratic_MH_series} for $\widetilde{H}_{\eta}$ suggests. One may further check that the same expansion holds when $ab<0$.

\begin{figure}[H]
\centering
\includegraphics[width=.8\textwidth]{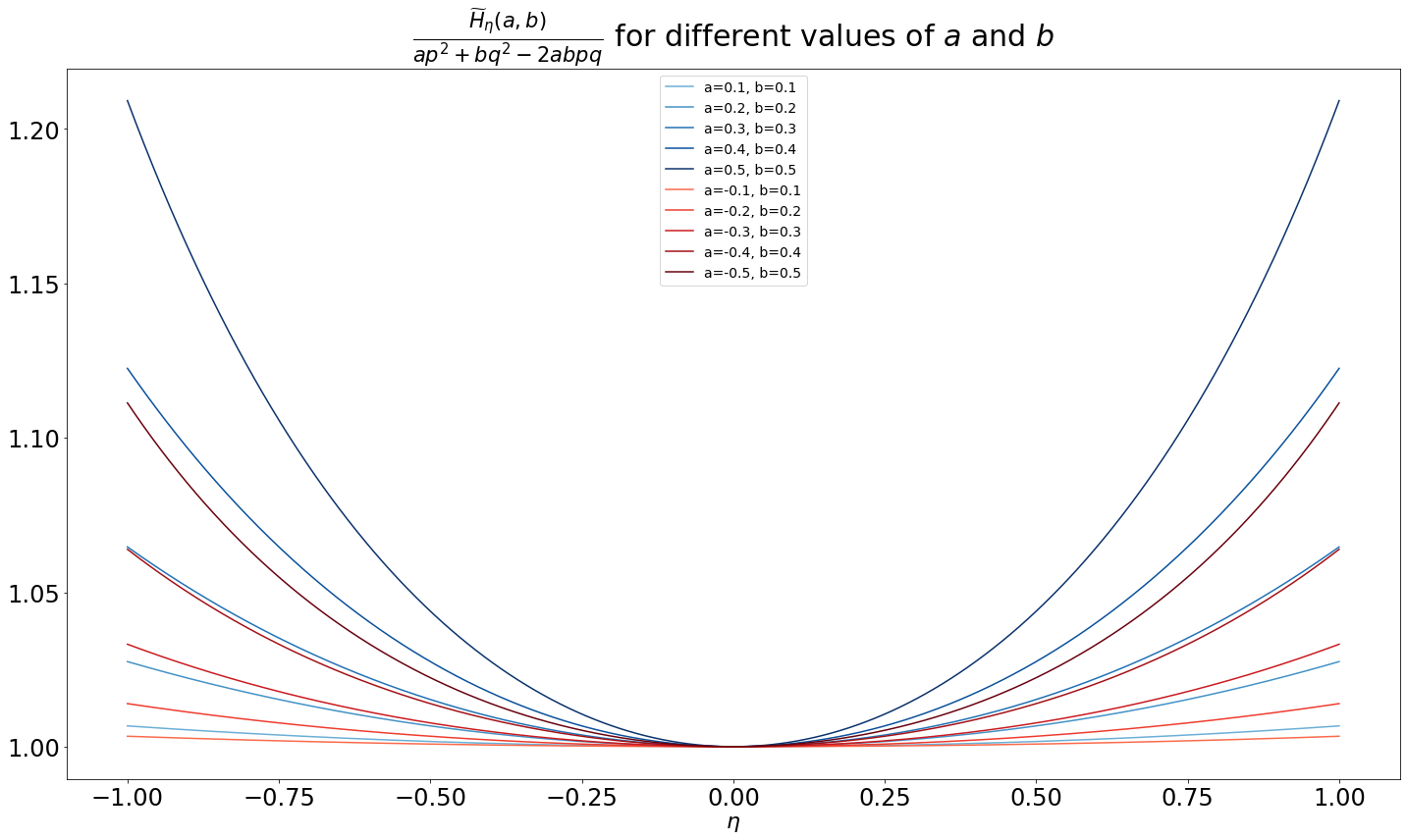}
\caption{\footnotesize The function $\eta \mapsto \frac{\arcsin\left(\sqrt{ab\eta^2}\right)}{\sqrt{ab\eta^2(1-ab\eta^2)}}$ with different choices of $a$ and $b$.}
\label{fig:2}
\end{figure}

\subsection{Verification: Hamiltonian flow of modified Hamiltonian}\label{E_flow}
In this subsection, we verify Theorem \ref{conserved_Hamiltonian} in the single-variable quadratic case. To this end, we have the following lemma:
\begin{lemma}
    Define the modified Hamiltonian as $$
    \tilde H_\eta(p,q) = \frac{\arcsin(\sqrt{ab\eta^2})}{\sqrt{ab\eta^2(1-ab\eta^2)}} (ap^2+bq^2-2abpq\eta).
    $$
    If $\left(p(t), q(t)\right)$ follows the Hamiltonian flow of the modified Hamiltonian:
    $$
    \frac{d}{dt}
    \begin{pmatrix}
        p(t)\\q(t)
    \end{pmatrix}=
    \begin{pmatrix}
        -\frac{\partial }{\partial q}\widetilde{H}_{\eta}(p(t), q(t))\\
        \frac{\partial }{\partial p}\widetilde{H}_{\eta}(p(t), q(t))
    \end{pmatrix},
    $$
    and starts at $$
    \begin{pmatrix}
        p(0) \\ q(0)
    \end{pmatrix}=\begin{pmatrix}
        p_0 \\ q_0
    \end{pmatrix},
    $$ then we have:
    $$
    \begin{pmatrix}
        p(\eta)\\q(\eta)
    \end{pmatrix}=\begin{pmatrix}
        1 & - 2b\eta \\
        2a \eta & 1-4ab \eta^2
    \end{pmatrix}
    \begin{pmatrix}
        p_0 \\ q_0
    \end{pmatrix},
    $$
    which coincides with the iterates of symplectic Euler \eqref{init} when $H(p, q)=ap^2+bq^2$.
\end{lemma}
\begin{proof}
    For $\eta > 0$, let
$$C_\eta = \frac{\arcsin(\sqrt{ab\eta^2})}{\sqrt{ab\eta^2(1-ab\eta^2)}}$$
so that the modified Hamiltonian above becomes:
$$\tilde H_\eta(p,q) = C_\eta (ap^2+bq^2-2abpq\eta).$$
The Hamiltonian flow is:
\begin{align}
    \frac{d}{dt} \begin{pmatrix}
        p(t) \\ q(t)
    \end{pmatrix}
    = \begin{pmatrix}
        2C_\eta b (-q(t) + a\eta p(t)) \\ 
        2C_\eta a (p(t) - b\eta q(t))
    \end{pmatrix}
    = 2C_\eta 
    \begin{pmatrix}
        ab\eta & -b \\
        a & -ab\eta
    \end{pmatrix}
    \begin{pmatrix}
        p(t) \\ q(t)
    \end{pmatrix}.\label{Hflow}
\end{align}
Using the matrix exponential, the solution to the Hamiltonian flow \eqref{Hflow} at time $t=\eta$ is:
\begin{align}
    \begin{pmatrix}
        p(\eta) \\ q(\eta)
    \end{pmatrix}
    = \exp\left(2\eta C_\eta \begin{pmatrix}
        ab\eta & -b \\
        a & -ab\eta
    \end{pmatrix}\right) 
    \begin{pmatrix}
        p_0 \\ q_0
    \end{pmatrix}.    \label{matexp}
\end{align}
We compute the matrix exponential in \eqref{matexp} as follows:
\begin{align}
    \exp\left(2\eta C_\eta \begin{pmatrix}
        ab\eta & -b \\
        a & -ab\eta
    \end{pmatrix}\right) &=
    \begin{bmatrix}
        E_{11} & E_{12} \\
        E_{21} & E_{22}
    \end{bmatrix} \nonumber,
\end{align}
where the coefficients are as follows:
\begin{subequations}
    \begin{align}
        E_{11} &= \frac{\left(ab\eta+\sqrt{ab(ab\eta^2-1)}\right) \exp\left(2\sqrt{ab(ab\eta^2-1)}\eta C_\eta\right)}{2\sqrt{ab(ab\eta^2-1)}}\nonumber\\ &\quad -
        \frac{\left(ab\eta-\sqrt{ab(ab\eta^2-1)}\right) \exp\left(-2\sqrt{ab(ab\eta^2-1)}\eta C_\eta\right)}{2\sqrt{ab(ab\eta^2-1)}} \nonumber \\
        E_{12} &= \frac{\sqrt{b} \exp\left(-2\sqrt{ab(ab\eta^2-1)}\eta C_\eta\right)}{2\sqrt{a(ab\eta^2-1)}}-
        \frac{\sqrt{b} \exp\left(2\sqrt{ab(ab\eta^2-1)}\eta C_\eta\right)}{2\sqrt{a(ab\eta^2-1)}}\nonumber \\
        E_{21} &= \frac{\sqrt{a} \exp\left(2\sqrt{ab(ab\eta^2-1)}\eta C_\eta\right)}{2\sqrt{b(ab\eta^2-1)}}-
        \frac{\sqrt{a} \exp\left(-2\sqrt{ab(ab\eta^2-1)}\eta C_\eta\right)}{2\sqrt{b(ab\eta^2-1)}}\nonumber \\
        E_{22} &= \frac{\left(-ab\eta+\sqrt{ab(ab\eta^2-1)}\right) \exp\left(2\sqrt{ab(ab\eta^2-1)}\eta C_\eta\right)}{2\sqrt{ab(ab\eta^2-1)}}\nonumber\\ &\quad -
        \frac{\left(-ab\eta-\sqrt{ab(ab\eta^2-1)}\right) \exp\left(-2\sqrt{ab(ab\eta^2-1)}\eta C_\eta\right)}{2\sqrt{ab(ab\eta^2-1)}}\nonumber.
    \end{align}\label{E.Trial_coefficients}
\end{subequations}
To simplify the coefficients above, we let $S=\sqrt{ab\eta^2(1-ab\eta^2)}$, $T=ab\eta^2$. We further calculate:
\begin{align}
    \exp\left(2\sqrt{ab(ab\eta^2-1)}\eta C_\eta\right) &= \exp\left(2\sqrt{ab(ab\eta^2-1)}\eta \frac{\arcsin(\sqrt{ab\eta^2})}{\sqrt{ab\eta^2(1-ab\eta^2)}}\right) \nonumber \\
    &= \exp\left(2i \arcsin\left(\sqrt{ab\eta^2}\right)\right) \nonumber \\
    %&= \cos\left(2\arcsin\left(\sqrt{ab\eta^2}\right)\right)+i\sin\left(2\arcsin\left(\sqrt{ab\eta^2}\right)\right) \nonumber \\
    %&=  1-2\sin\left(\arcsin\left(\sqrt{ab\eta^2}\right)\right)^2\nonumber\\ &\quad +2i\sin\left(\arcsin\left(\sqrt{ab\eta^2}\right)\right)\cos\left(\arcsin\left(\sqrt{ab\eta^2}\right)\right) \nonumber \\
    %&= 1-2ab\eta^2+2i\sqrt{ab\eta^2}\sqrt{1-ab \eta^2} \nonumber \\
    &= 1-2T+2iS \nonumber
\end{align}
Similarly, we have:
\begin{gather}
    \exp\left(-2\sqrt{ab\left(ab\eta^2-1\right)}\eta C_\eta\right) = 1-2T-2iS\nonumber.
\end{gather}
Hence, we could simplify the coefficients in the following way:
\begin{gather}
    \begin{pmatrix}
        E_{11} & E_{12} \\
        E_{21} & E_{22}
    \end{pmatrix}=
    \begin{pmatrix}
        1&-2b\eta \\
        2a\eta &1-4ab\eta^2
    \end{pmatrix}\nonumber.
\end{gather}
The symplectic Euler algorithm, \eqref{init}, in this case is:
\begin{align}
    \begin{pmatrix}
        p_{\text{new}} \\ q_{\text{new}}
    \end{pmatrix}
    &=
    \begin{pmatrix}
        p_0 - \eta (2bq_0) \\
        q_0 + \eta (2ap_{\text{new}})
    \end{pmatrix}=
    \begin{pmatrix}
        p_0 - 2b\eta q_0 \\
        q_0 + 2a \eta \left(p_0 - 2b\eta q_0 \right)
    \end{pmatrix}= 
    \begin{pmatrix}
        1 & - 2b\eta \\
        2a \eta & 1-4ab \eta^2
    \end{pmatrix}
    \begin{pmatrix}
        p_0 \\ q_0
    \end{pmatrix}\nonumber.
\end{align}
Since the coefficients match, we confirm that the iterations are indeed interpolated by the Hamiltonian flow generated by the modified Hamiltonian $\widetilde{H}_{\eta}$.
\end{proof}
\vspace{-0.5cm}
\subsection{The multivariate case}\label{E.4}
In this subsection we derive the closed-form of the modified Hamiltonian in the multivariate quadratic case by utilizing results from monovariate case. 

The general idea of the proof is the following: we first derive connections between the iterated Poisson brackets in multivariate and univariate cases. We then reorder the series form \eqref{dynkin} of the modified Hamiltonian in the multivariate case $\widetilde{H}^{F, G}_{\eta}(p, q)$ and relate the reordered series with the series form of the modified Hamiltonian in the univariate case to get closed-form expression of $\widetilde{H}^{F, G}_{\eta}(p, q)$. Furthermore, this allows us to determine the absolute convergence radius of the reordered form, which then implies equivalence between the original form and the reordered form. 

\subsubsection{Notations and settings}
We assume that $p, q \in \R^d$, where $d \geq 2$.  We assume that $B, C \in \R^{d \times d}$ are symmetric matrices such that $BC$ is positive semi-definite. Let $Q^{-1} \Lambda Q$ be the diagonalization of $BC$. We define $F(p)=p^{\top}Bp$, $G(q)=q^{\top}Cq$, and $H^{F, G}(p, q)=F(p)+G(q)$. In using results from last subsection, we assume that $u, v \in \R$, and define $f(u)=au^2, g(v)=bv^2$, where $ab>0$, and let $H^{f, g}(u, v)=f(u)+g(v)$. 

We want to find the closed-form expression of modified Hamiltonian for the following symplectic Euler updates:
\begin{gather}
    p_{k+1} = p_k - 2\eta C q_{k}, \quad q_{k+1} = q_k + 2\eta Bp_{k+1} \nonumber
\end{gather}
Throughout the subsection, we let $\star$ denote the tuple $(r_1, s_1, \dots, r_n, s_n)$, i.e.,
$$\star\coloneqq  (r_1, s_1, \dots, r_n, s_n).$$
We denote the rank (for its definition, check the review on iterated Poisson bracket in Section \ref{MH_more}) of the Poisson bracket corresponding to the tuple $\star$ with $\mathcal{N}(\star)$, denote $\M(\star)\coloneqq \sum_{i=1}^n (r_i+s_i),$ and denote the set of tuples of rank $N$ with $\mathcal{S}(N)$. Finally, we define the following function to express the constant coefficient in the $\widetilde{H}_{\eta}$ series:
\begin{align}
    \tau (\star) = \tau((r_1, s_1, \cdots, r_n, s_n)) ={\frac{\left(-1\right)^{n-1}}{n}{{\frac{1}{\left(r_{1}+\dots+r_{n}+s_{1}+\dots+s_{n}\right)\prod_{i=1}^{n}{r_{i}!s_{i}!}}}}} \nonumber
\end{align}

\noindent For simplicity of notations, we define
\begin{subequations}
    \begin{align}
    I_{p, q}^{F, G}&\coloneqq \left\{G^{r_{1}}F^{s_{1}}G^{r_{2}}F^{s_{2}}\cdots G^{r_{n}}F^{s_{n}}\right\}(p, q) \nonumber \\
    I_{u, v}^{f, g}&\coloneqq \left\{g^{r_{1}}f^{s_{1}}g^{r_{2}}f^{s_{2}}\cdots g^{r_{n}}f^{s_{n}}\right\}(u, v) \nonumber
    \end{align}
\end{subequations}

Results from Appendix \ref{E.1_space_IPBs} imply that $I_{u, v}^{f, g}(\star)$ must be a monomial containing $a, b, u, v$. We denote the constant factor within $I_{u, v}^{f, g}(\star)$ with $J(\star)$, denote the power of $a$ within $I_{u, v}^{f, g}(\star)$ with $L_a(\star)$ and the power of $b$ within $I_{u, v}^{f, g}(\star)$ with $L_b(\star)$. The following proposition is a direct consequence of our work above:
\begin{proposition}
    $I_{u, v}^{f, g}(\star)$ must be one of the following forms: 
    \begin{itemize}
        \item $J(\star)a^{L_a(\star)}b^{L_b(\star)}u^2$, 
        \item $J(\star)a^{L_a(\star)}b^{L_b(\star)}v^2$, or
        \item $J(\star)a^{L_a(\star)}b^{L_b(\star)}uv$.
    \end{itemize}
\end{proposition}

\noindent We define the following subsets of possible tuples:
\begin{subequations}
    \begin{align}
    \mathcal{S}_{pp}(n) \coloneqq  \mathcal{S}_{uu}(n) & \coloneqq  \{\star \in \mathcal{S}(n) \mid I_{u, v}^{f, g}(\star) \text{ contains } uu\}\nonumber \\
    \mathcal{S}_{qq}(n) \coloneqq  \mathcal{S}_{uu}(n) & \coloneqq  \{\star \in \mathcal{S}(n) \mid I_{u, v}^{f, g}(\star) \text{ contains } uu\}\nonumber \\
    \mathcal{S}_{pq}(n) \coloneqq  \mathcal{S}_{uv}(n) & \coloneqq  \{\star \in \mathcal{S}(n) \mid I_{u, v}^{f, g}(\star) \text{ contains } uv\}\nonumber
\end{align}
\end{subequations}

\noindent Notice that $\star \in \mathcal{S}_{uu}(n) \cup \mathcal{S}_{uu}(n) \cup \mathcal{S}_{uv}(n)$ implies that $I_{u, v}^{f, g}(\star) \neq 0$.

\noindent Finally, we define the following functions to connect multivariate and univariate quadratic cases.
\begin{gather}
T^{\text{pure}}_{\eta}(\lambda) \coloneqq  \begin{cases}
    \frac{\arcsin (\sqrt{\lambda \eta^2})}{\sqrt{\lambda\eta^2(1-\lambda\eta^2)}} & \lambda > 0,\\
    \frac{\arcsinh (\sqrt{-\lambda \eta^2})}{\sqrt{-\lambda\eta^2(1-\lambda\eta^2)}}& \lambda < 0.
\end{cases}, \quad
T^{\text{cross}}_{\eta}(\lambda) \coloneqq \begin{cases}
    \frac{-2\sqrt{\lambda}\arcsin (\sqrt{\lambda \eta^2}) }{\sqrt{(1-\lambda\eta^2)}} & \lambda > 0,\\
    \frac{2\sqrt{-\lambda}\arcsinh (\sqrt{-\lambda \eta^2}) }{\sqrt{(1-\lambda\eta^2)}}& \lambda < 0. 
\end{cases}\nonumber
\end{gather}

\subsubsection{Connections between the iterated Poisson brackets in two cases}
\begin{lemma}
    If $\star \in \mathcal{S}_{uu}(n)$, then $L_a(\star)=L_b(\star)+1=\frac{\M(\star)+1}{2}$. If $\star \in \mathcal{S}_{uu}(n)$, then $L_b(\star)=L_a(\star)+1=\frac{\M(\star)+1}{2}$.  If $\star \in \mathcal{S}_{uv}(n)$, then $L_a(\star)=L_b(\star)=\frac{\M(\star)}{2}$. 
\end{lemma}
\begin{proof}
    This can be verified via an inductive argument based on $\M(\star) = \{1, 2, 3, \cdots\}$. As the base cases, the proposed lemma holds for $M=1$, where $f=au^2$, $g=bv^2$ and $M=2$, where $\{f, g\}=-4abuv=-\{g, f\}$. Suppose that the lemma holds  for any tuples $\star$ such that $\M(\star) \leq n$, we verify that it also holds for any tuples $\star$ such that $\M(\star) \leq n+1$. For the sake of simplicity we only verify the case where $\{\cdot, f\}$ is applied to $I_{u, v}^{f, g}(\star)$. We denote the new tuple in this case as $\star'$ and remark that $\M(\star')=\M(\star)+1$.

    \begin{itemize}
        \item If $I_{u, v}^{f, g}(\star)=J(\star)a^{\frac{\M(\star)+1}{2}}b^{\frac{\M(\star)-1}{2}}u^2$, then \begin{align}
        \{I_{u, v}^{f, g}(\star), f\}&=\{J(\star)a^{\frac{\M(\star)+1}{2}}b^{\frac{\M(\star)-1}{2}}u^2, au^2\} \nonumber \\
        &=0 \nonumber.
    \end{align}
    \item If $I_{u, v}^{f, g}(\star)=J(\star)a^{\frac{\M(\star)-1}{2}}b^{\frac{\M(\star)+1}{2}}v^2$, then \begin{align}
        \{I_{u, v}^{f, g}(\star), f\} &= \{J(\star)a^{\frac{\M(\star)-1}{2}}b^{\frac{\M(\star)+1}{2}}v^2, au^2\} \nonumber \\
        &=J(\star')a^{\frac{\M(\star)+1}{2}}b^{\frac{\M(\star)+1}{2}}uv \nonumber \\
        &=J(\star')a^{\frac{\M(\star')}{2}}b^{\frac{\M(\star')}{2}}uv \nonumber.
    \end{align}
    \item If $I_{u, v}^{f, g}(\star)=J(\star)a^{\frac{\M(\star)}{2}}b^{\frac{\M(\star)}{2}}uv$, then \begin{align}
        \{I_{u, v}^{f, g}(\star), f\} &= \{J(\star)a^{\frac{\M(\star)}{2}}b^{\frac{\M(\star)}{2}}uv, au^2\} \nonumber \\
        &= J(\star')a^{\frac{\M(\star)+2}{2}}b^{\frac{\M(\star)}{2}}u^2\nonumber \\
        &= J(\star')a^{\frac{\M(\star')+1}{2}}b^{\frac{\M(\star')}{2}}u^2 \nonumber.
    \end{align}
    \end{itemize}
    Thus, the induction argument is verified in all cases.
\end{proof}

\begin{lemma}
    If $\star \in \mathcal{S}_{uu}(n)$, then $I_{p, q}^{F, G}(\star) = J(\star)p^\top (BC)^{\frac{\M(\star)-1}{2}}B p$. 
    If $\star \in \mathcal{S}_{uu}(n)$, then $I_{p, q}^{F, G}(\star) = J(\star)q^\top C(BC)^{\frac{\M(\star)-1}{2}} q$. 
    If $\star \in \mathcal{S}_{uv}(n)$, then $I_{p, q}^{F, G}(\star) = J(\star)p^\top (BC)^{\frac{\M(\star)}{2}} q$. \label{IPB_forms_lemma}
\end{lemma}

\begin{proof}
    This can be verified using an inductive argument based on $\M(\star) = \{1, 2, 3, \cdots\}$. As the base case, the proposed lemma holds for $M=1$, where $F=p^{\top}Bp$, $G=q^{\top}Cq$ and $M=2$, where $\{F, G\}=-4p^{\top}BCq=-\{G, F\}$. Suppose the lemma holds for any tuples $\star$ such that $\M(\star) \leq n$, we verify that it also holds for any tuples $\star$ such that $\M(\star) \leq n+1$. For the sake of simplicity we only verify the case where $\{\cdot, F\}$ is applied to $I_{p, q}^{F, G}(\star)$ and when $\{\cdot, f\}$ is applied to $I_{u, v}^{f, g}(\star)$. To show the inductive argument, we need to compare $\{I_{p, q}^{F, G}(\star), F\}$ and $\{I_{u, v}^{f, g}(\star), f\}$. 

    If $\star$ ends with $\{\dots, r_n, 0\}$ where $r_n>0$, the new tuple $\star'$ will be $\{\dots, r_n, 1\}$. On the other hand, if $\star$ ends with $\{\dots, r_n, s_n \}$ where $r_n, s_n>0$, the new tuple $\star'$ will be $\{\dots, r_n, s_n+1\}$. In both cases, $I_{p, q}^{F, G}(\star')={I_{p, q}^{F, G}(\star), F}$. Note that $\M(\star')=\M(\star)+1$.

     When $\star \in \mathcal{S}_{uu}(n)$,  $I_{p, q}^{F, G}=J(\star)p^\top (BC)^{\frac{\M(\star)-1}{2}}B p$ and $I_{u, v}^{f, g}=J(\star)a^{\frac{\M(\star)+1}{2}}b^{\frac{\M(\star)-1}{2}}u^2$. Thus,
     \begin{gather}
         \{I_{p, q}^{F, G}(\star), F\} = \{I_{u, v}^{f, g}(\star), f\} = 0 \nonumber.
     \end{gather}

     \noindent When $\star \in \mathcal{S}_{vv}(n)$,  $I_{p, q}^{F, G}=J(\star)q^\top C(BC)^{\frac{\M(\star)-1}{2}} q$ and $I_{u, v}^{f, g}=J(\star)a^{\frac{\M(\star)-1}{2}}b^{\frac{\M(\star)+1}{2}}v^2$. In this case, 
     \begin{align}
         \{I_{u, v}^{f, g}(\star), f\} &= \{J(\star)a^{\frac{\M(\star)-1}{2}}b^{\frac{\M(\star)+1}{2}}v^2, au^2 \} \nonumber \\
         &= 4J(\star)a^{\frac{\M(\star)+1}{2}}b^{\frac{\M(\star)+1}{2}}uv \nonumber \\
         &= 4J(\star)a^{\frac{\M(\star')}{2}}b^{\frac{\M(\star')}{2}}uv \nonumber.
     \end{align}
     On the other hand, 
     \begin{align}
         \{I_{p, q}^{F, G}(\star), F\} &= \{J(\star)q^\top C(BC)^{\frac{\M(\star)-1}{2}} q, p^{\top}Bp\} \nonumber \\
         &= 4J(\star)p^\top (BC)^{\frac{\M(\star)+1}{2}} q \nonumber \\
         &= 4J(\star)p^\top (BC)^{\frac{\M(\star')}{2}} q \nonumber.
     \end{align}
     Since the constant term match and the power of $a$, $b$ matches with that of $B$ and $C$, the assumption is verified.

     When $\star \in \mathcal{S}_{uv}(n)$,  $I_{p, q}^{F, G}=J(\star)p^\top (BC)^{\frac{\M(\star)}{2}} q$ and $I_{u, v}^{f, g}=J(\star)a^{\frac{\M(\star)}{2}}b^{\frac{\M(\star)}{2}}uv$. In this case, 
     \begin{align}
         \{I_{u, v}^{f, g}(\star), f\} &= \{J(\star)a^{\frac{\M(\star)}{2}}b^{\frac{\M(\star)}{2}}uv, au^2\} \nonumber \\
         &= 2J(\star)a^{\frac{\M(\star)+2}{2}}b^{\frac{\M(\star)}{2}}u^2 \nonumber \\
         &= 2J(\star)a^{\frac{\M(\star')+1}{2}}b^{\frac{\M(\star')-1}{2}}u^2 \nonumber.
     \end{align}
     On the other hand, \begin{align}
         \{I_{p, q}^{F, G}(\star), F\} &= \{J(\star)p^\top (BC)^{\frac{\M(\star)}{2}} q, p^{\top}Bp\} \nonumber \\
         &= 2J(\star)p^\top (BC)^{\frac{\M(\star)}{2}}B p \nonumber \\
         &= 2J(\star)p^\top (BC)^{\frac{\M(\star')-1}{2}}B p \nonumber
     \end{align}
     Since the constant term match and the power of $a$, $b$ matches with that of $B$ and $C$, the assumption is verified.
\end{proof}

\begin{corollary}
    If $\star \in \mathcal{S}_{uu}(n)$, then $I_{p, q}^{F, G}(\star) = J(\star)(Qp)^\top \Lambda ^{\frac{\M(\star)-1}{2}}QBp$.
    If $\star \in \mathcal{S}_{vv}(n)$, then $I_{p, q}^{F, G}(\star) = J(\star)(QC^\top q)^\top \Lambda ^{\frac{\M(\star)-1}{2}}Qq$.
    If $\star \in \mathcal{S}_{uv}(n)$, then $I_{p, q}^{F, G}(\star) = J(\star)(Qp)^\top \Lambda^{\frac{\M(\star)}{2}}Qq$ \nonumber. 
\end{corollary}

\begin{proof}
    We note that $(BC)^{\frac{\M(\star)-1}{2}}B = (Q^{-1}\Lambda Q)^{\frac{\M(\star)-1}{2}}B=Q^{-1} \Lambda^{\frac{\M(\star)-1}{2}} Q B$, $C(BC)^{\frac{\M(\star)-1}{2}}=\\ C(Q^{-1}\Lambda Q)^{\frac{\M(\star)-1}{2}}=CQ^{-1} \Lambda^{\frac{\M(\star)-1}{2}} Q$, and $(BC)^{\frac{\M(\star)}{2}}=(Q^{-1}\Lambda Q)^{\frac{M(\star)}{2}}=Q^{-1} \Lambda^{\frac{M(\star)}{2}}Q$.
\end{proof}

 This confirms that if $\star \in \mathcal{S}_{pp}$, then $I_{u, v}^{f, g}(\star)$ is a bilinear product of a matrix with $p, p$. Similar properties hold for $\mathcal{S}_{qq}$ and $\mathcal{S}_{pq}$.

\subsubsection{Reordering the series: finding the closed-form and absolute convergence radius}
\begin{lemma}
    The Dynkin form \eqref{dynkin} of the modified Hamiltonian in multivariate quadratic case $\widetilde{H}_{\eta}^{F, G}(p, q)$ can be rearranged into the following:
    \begin{align}
    \Hat{H}_{\eta}^{F, G}(p, q)&=
    (Qp)^\top \left( \sum_{N=1}^{\infty}\frac{(-1)^{N-1}}N 
    \sum_{\star \in \mathcal{S}_{pp}(N)} \eta^{\M(\star)-1} \tau (\star)  J(\star) \Lambda ^{\frac{\M(\star)-1}{2}}\right)QBp\nonumber \\
    &\quad +(QC^\top q)^\top \left( \sum_{N=1}^{\infty}\frac{(-1)^{N-1}}N 
    \sum_{\star \in \mathcal{S}_{qq}(N)} \eta^{\M(\star)-1} \tau (\star)  J(\star) \Lambda ^{\frac{\M(\star)-1}{2}}\right)Qq\nonumber \\
    &\quad +(Qp)^\top \left( \sum_{N=1}^{\infty}\frac{(-1)^{N-1}}N 
    \sum_{\star \in \mathcal{S}_{pq}(N)} \eta^{\M(\star)-1} \tau (\star)  J(\star) \Lambda ^{\frac{\M(\star)}{2}}\right)Qq \label{eq_E.3_matrixseries}
    \end{align}
    Note that the equivalence between $\widetilde{H}_{\eta}^{F, G}(p, q)$ and $\Hat{H}_{\eta}^{F, G}(p, q)$ is not guaranteed, because we have not proven the absolute convergence yet.\label{lemma_reordered_BCH}
\end{lemma}

\begin{proof}
    Let us start from the Dynkin form \eqref{dynkin} of the modified Hamiltonian:
    $$
    \widetilde{H}^{F, G}_{\eta}\left(p,q\right) =\sum_{n=1}^{\infty}{\frac{\left(-1\right)^{n-1}}{n}{\sum_{\substack{ r_{1}+s_{1}>0 \\ \vdots \\ r_{n}+s_{n}>0}}{\frac{\eta^{r_{1}+\dots+r_{n}+s_{1}+\dots+s_{n}-1}\left\{G^{r_{1}}F^{s_{1}}G^{r_{2}}F^{s_{2}}\cdots G^{r_{n}}F^{s_{n}}\right\}(p,q)}{\left(r_{1}+\dots+r_{n}+s_{1}+\dots+s_{n}\right)\prod_{i=1}^{n}{r_{i}!s_{i}!}}}}}.
    $$
    The summation can be reordered as:
    $$
    \Hat{H}^{F, G}_{\eta}\left(p,q\right) = \sum_{n=1}^{\infty} \sum_{\star \in \mathcal{S}_{pp}(n) \cup \mathcal{S}_{pq}(n) \cup \mathcal{S}_{pq}(n)} \tau(\star)\eta^{\M(\star)}I_{p, q}^{F, G}(\star).
    $$
    We compute the summation for $\mathcal{S}_{pp}(n), \mathcal{S}_{qq}(n)$ and $\mathcal{S}_{pq}(n)$ separately, and apply the results in Lemma \ref{IPB_forms_lemma} to conclude.
\end{proof}

\begin{theorem}
    The reordered $\widetilde{H}_{\eta}$ series in Lemma \ref{lemma_reordered_BCH} has the following closed form:
    \begin{align*}
        \Hat{H}_{\eta}^{F, G}(p, q)&=(Qp)^\top T^{\text{pure}}_{\eta}(\Lambda) QBp\\
&\quad +(QC^\top q)^\top T^{\text{pure}}_{\eta}(\Lambda) Qq\\
&\quad +(Qp)^\top T^{\text{cross}}_{\eta}(\Lambda) Qq.
    \end{align*}
    Whenever each eigenvalue $\lambda$ in $\Lambda$ satisfies that $|\lambda \eta^2| \leq 1$, the original BCH series $\widetilde{H}_{\eta}^{F, G}(p, q)$ converges absolutely and reordering of summation is allowed. Since $BC$ is assumed to be diagonalizable, if the maximal singular value of $BC$ is smaller than $\frac{1}{\eta^2}$, then $\widetilde{H}_{\eta}^{F, G}(p, q)=\Hat{H}_{\eta}^{F, G}(p, q)$.
\end{theorem}

\begin{proof}
    Theorem \ref{Thm_Quadratic_MH_multivar} says that in univariate case, BCH series $\widetilde{H}_{\eta}^{f, g}(u, v)$ converges to $\frac{1}{\sqrt{ab\eta^2(1-ab\eta^2)}}\arcsin(\sqrt{ab\eta^2})(au^2+bv^2-2abuv\eta)$. Furthermore, the convergence is absolute whenever $|ab\eta^2| < 1$.

    Assume without loss of generality that $ab>0$, for the other case is similar. Absolute convergence allows for the following reordering of the summation:
    \begin{align}
\widetilde{H}_{\eta}^{f, g}(u, v)&= \left( \sum_{N=1}^{\infty}\frac{(-1)^{N-1}}N 
\sum_{\star \in \mathcal{S}_{uu}(N)} \eta^{\M(\star)-1} \tau (\star)  J(\star)(ab)^{\frac{\M(\star)-1}{2}}\right)au^2\nonumber \\
&\quad + \left( \sum_{N=1}^{\infty}\frac{(-1)^{N-1}}N
\sum_{\star \in \mathcal{S}_{vv}(N)} \eta^{\M(\star)-1} \tau (\star)   J(\star)(ab)^{\frac{\M(\star)-1}{2}}\right)bv^2\nonumber \\
&\quad +\left( \sum_{N=1}^{\infty}\frac{(-1)^{N-1}}N
\sum_{\star \in \mathcal{S}_{uv}(N)} \eta^{\M(\star)-1} \tau (\star)  J(\star)(ab)^{\frac{\M(\star)}{2}}\right)uv \nonumber \\
&=\frac{1}{\sqrt{ab\eta^2(1-ab\eta^2)}}\arcsin(\sqrt{ab\eta^2})(au^2+bu^2-2abuv\eta). \label{eq_E.3_univarseries}
\end{align}
We now substitute the variable $ab$ with $\lambda$. Equation \eqref{eq_E.3_univarseries} above implies the following:
\begin{subequations}\label{eq_E.3_closedTs}
    \begin{align}
    \sum_{N=1}^{\infty}\frac{(-1)^{N-1}}N 
\sum_{\star \in \mathcal{S}_{uu}(N)} \eta^{\M(\star)-1} \tau (\star)  J(\star) \lambda^{\frac{\M(\star)-1}{2}} &= \frac{\arcsin (\sqrt{\lambda \eta^2})}{\sqrt{\lambda\eta^2(1-\lambda\eta^2)}} = T^{\text{pure}}_{\eta}(\lambda)\\
\sum_{N=1}^{\infty}\frac{(-1)^{N-1}}N
\sum_{\star \in \mathcal{S}_{vv}(N)} \eta^{\M(\star)-1} \tau (\star)  J(\star) \lambda^{\frac{\M(\star)-1}{2}} &= \frac{\arcsin (\sqrt{\lambda \eta^2})}{\sqrt{\lambda\eta^2(1-\lambda\eta^2)}} = T^{\text{pure}}_{\eta}(\lambda) \\
\sum_{N=1}^{\infty}\frac{(-1)^{N-1}}N
\sum_{\star \in \mathcal{S}_{uv}(N)} \frac{\eta^{\M(\star)-1}}{\M(\star)} \tau (\star)  J(\star) \lambda^{\frac{\M(\star)}{2}} &= \frac{-2\sqrt{\lambda}\arcsin (\sqrt{\lambda \eta^2}) }{\sqrt{(1-\lambda\eta^2)}}= T^{\text{cross}}_{\eta}(\lambda).
\end{align}
\end{subequations}
Furthermore, the absolute convergence of original Dynkin series \eqref{dynkin} when $|\lambda \eta^2|<1$ implies absolute convergence of $T^{\text{pure}}_{\eta}(\lambda)$ and $T^{\text{cross}}_{\eta}(\lambda)$ when $|\lambda \eta^2| < 1$. 

Since $\Lambda$ is a diagonal matrix, taking powers of $\Lambda$ is equivalent with taking the same powers element-wise. Thus, the summation of matrices in \eqref{eq_E.3_matrixseries} is equivalent with summing each element along the diagonal separately. Furthermore, since the terms in \eqref{eq_E.3_matrixseries} are precisely those in \eqref{eq_E.3_univarseries} and thereby equal to the respective ones in \eqref{eq_E.3_closedTs}, the closed form should equal to the application of $T_{\eta}^{\text{pure}}$ and $T_{\eta}^{\text{cross}}$ on the diagonal elements separately. Finally, note that since $T^{\text{cross}}_{\eta}(\lambda)=|\lambda| T^{\text{pure}}_{\eta}(\lambda)$ holds for both $\lambda>0$ and $\lambda < 0$, we could simplify $T_{\eta}^{\text{cross}}(\Lambda)$ as $|\Lambda|T_{\eta}^{\text{pure}}(\Lambda)$. Thus, the closed-form is proven.

Finally, following the convergence criteria in Theorem \ref{Thm_Quadratic_MH_multivar}, the convergence criteria here should be:
$$
\lambda_{\max} (|\Lambda|) \eta^2 = \sigma_{\max}(BC) < 1.
$$

\noindent Since $\sigma_{\max}(BC)  \leq \sigma_{\max}(B) \sigma_{\max}(C)$, one sufficient convergence criterion is:
$$
\sigma_{\max}(B)\sigma_{\max}(C)\eta^2 < 1.
$$
\end{proof}

\section{Absolute convergence of the modified Hamiltonian}\label{Section_AbsCvg}

Much of the existing literature makes references to the convergence of the modified Hamiltonian \citep{Hairer2006,Alsallami2018,Field2003}, as from a theoretical perspective, convergence of the Dynkin series \eqref{dynkin} is necessary to define a modified Hamiltonian $\widetilde{H}_{\eta}$ in terms of the Dynkin series which is a well-defined function from $\Z$ to $\R$. In general, it is well-known that if the symplectic Euler method \eqref{init} is nonlinear, the Dynkin formula does not converge for \textit{any} $\eta$---see, e.g., \cite{Skeel}, \cite{Field2003}, or Sections IX.3 and IX.9.3 from \cite{Hairer2006}. However, numerical simulations---e.g., those shown in \cite{Wibisono2022} and Chapters I, IX, and X from \cite{Hairer2006}---suggest that for even non-smooth, closed, proper convex functions $F$ and $G$, the orbits generated by the iterations \eqref{init} lie on a closed, bounded orbit for $\eta$ sufficiently small. This at least suggests that the results shown above on the existence of closed, bounded orbits for small enough $\eta>0$ in the quadratic and logarithm cases also exist for larger classes of functions. 

That being said, the existence of closed, bounded orbits which look like the level set of some function does not necessarily imply that the BCH formula as it stands in either its Dynkin or integral form (i.e., without renormalization or analytic continuation) converges to a function whose level sets match with those orbits. In fact, there are direct counter-examples to the latter when $F$ is quadratic and $G$ is a polynomial with order higher than 3, and $d=1$. Numerical simulations---e.g., those in Figure \ref{fig:examples}---imply the existence of bounded, closed orbits for small enough $\eta$ and any initial condition throughout $\R\times\R$, even though \cite{Suris1989,Field2003} show that the BCH formula does not converge in such cases for \textit{any} $\eta>0$ when $G$ is a polynomial of degree $\geq 3$. Note that the smaller circles in the first plot of Figure \ref{fig:examples} of the same color are suborbital trajectories traced out on smaller portions of the phase space than the larger orbit for a strict subset of the iterations. For instance, for the trajectory in blue corresponding to $\left(p_{0},q_{0}\right)=\left(2,3\right)$, there are $26$ smaller circles. Thus, if we consider the dynamics of every 26th iterate of symplectic Euler \eqref{init}, then these are still periodic and lie on one of the smaller circles.

\begin{figure}[H]
\centering
  \begin{subfigure}{.3\textwidth}
    \includegraphics[width=\textwidth]{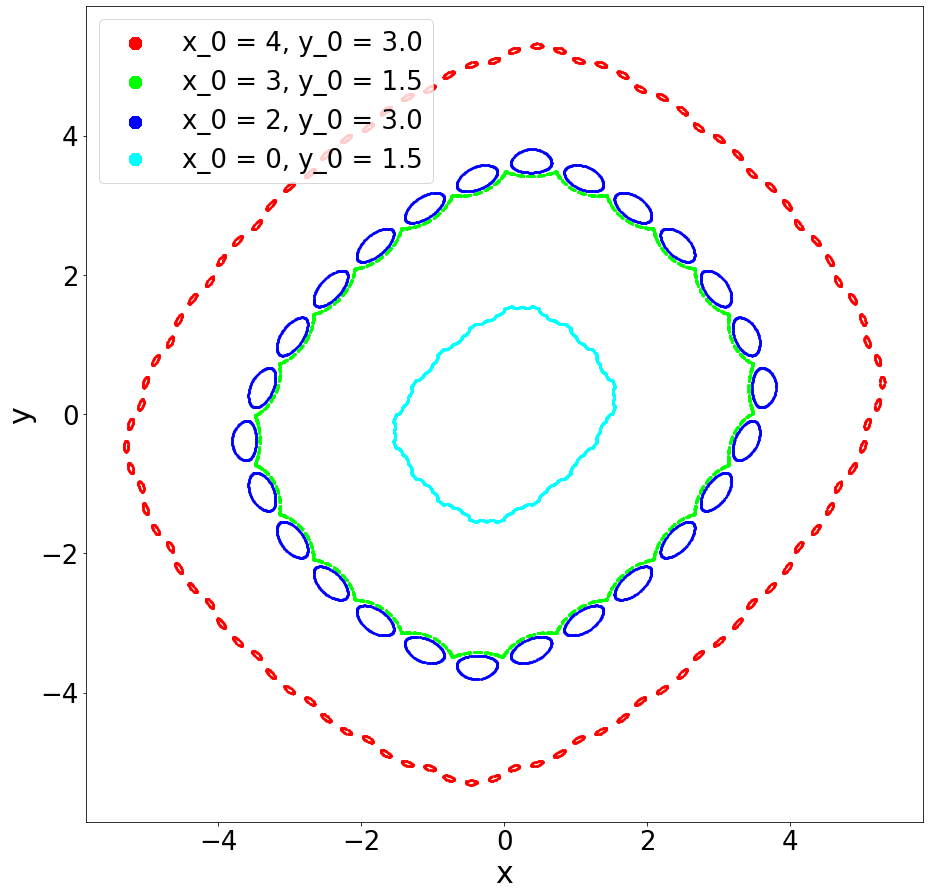}
    \caption{\footnotesize $H(x, y) = x^{1.5}+y^{1.5}, \eta=1$}
  \end{subfigure}
  \begin{subfigure}{.3\textwidth}
    \includegraphics[width=1\textwidth]{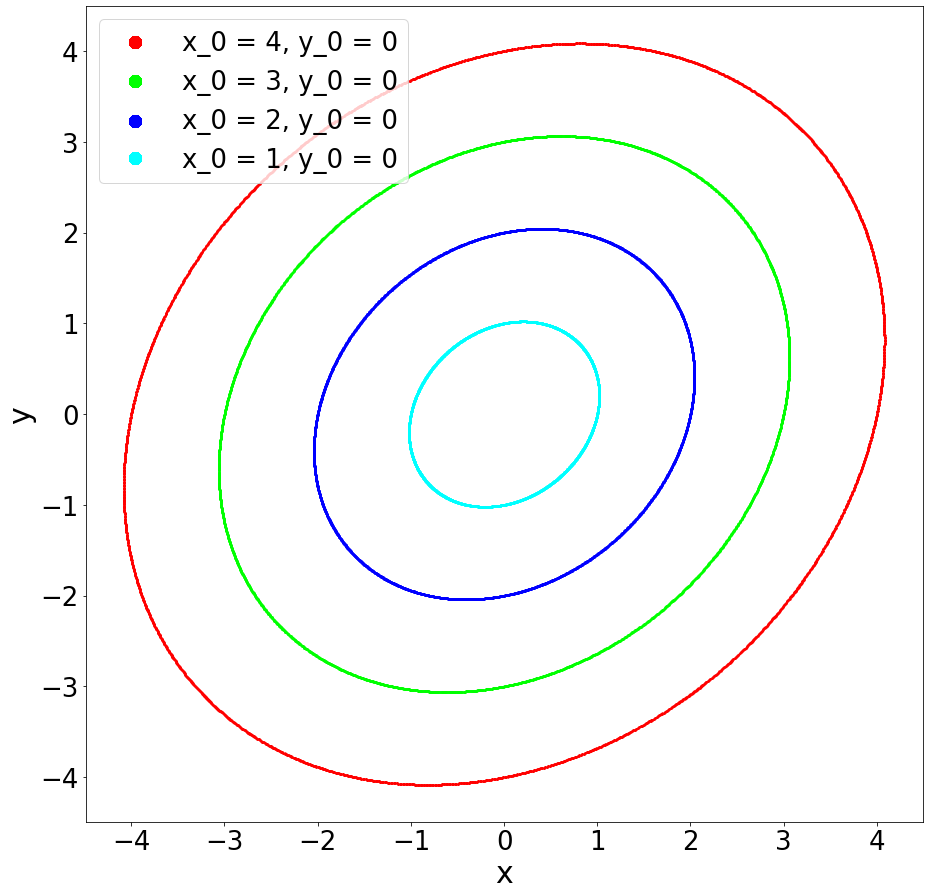}
    \caption{\footnotesize $H(x, y) = x^{2}+y^{2}, \eta=1$}
  \end{subfigure}
   \begin{subfigure}{.3\textwidth}
    \includegraphics[width=1\textwidth]{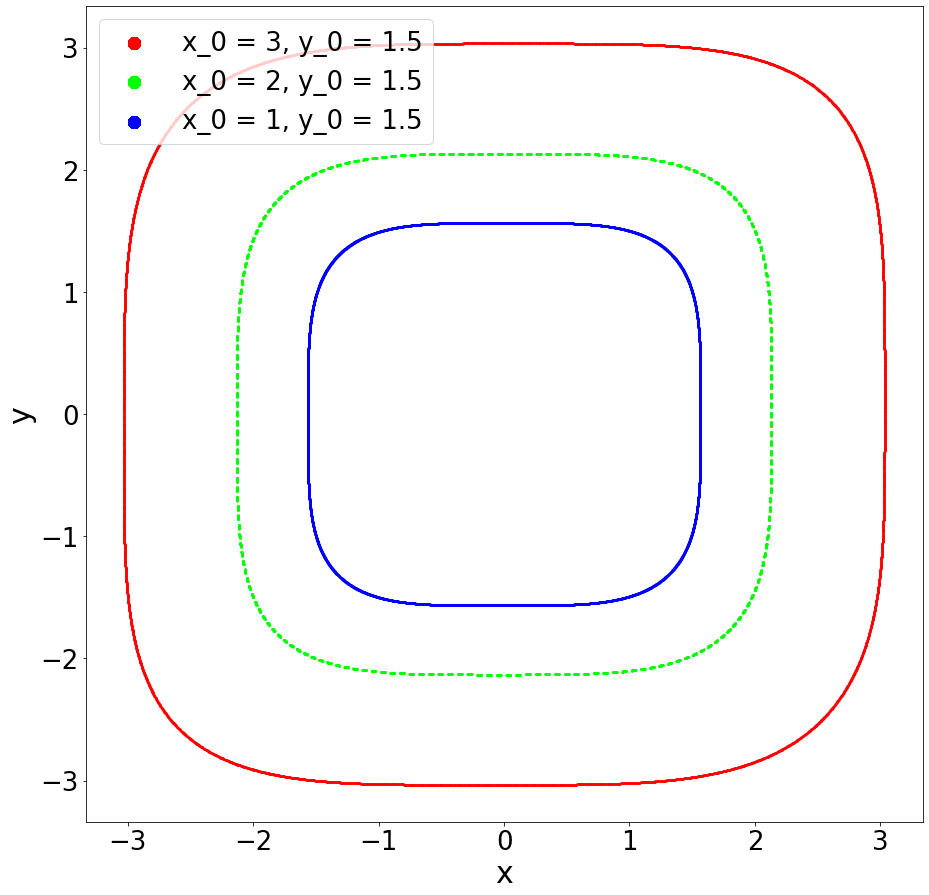}
    \caption{\footnotesize $H(x, y) = x^{4}+y^{4}, \eta=1$}
  \end{subfigure}
    \caption{\footnotesize Trajectories of iterations symplectic Euler discretizations of the Hamiltonian flows plotted in scatter plot. 
Different colors correspond to different initial positions.\\
(a) $H(x, y) = x^{1.5}+y^{1.5}, \eta=1$ (b) $H(x, y) = x^{2}+y^{2}, \eta=1$ (c) $H(x, y) = x^{4}+y^{4}, \eta=1$}
    \label{fig:examples}
\end{figure}

At least to the authors' knowledge, no prior study other than \cite{Suris1989,Field2003} has tried to study which general classes of functions $F$ and $G$ will or will not lead to a convergent modified Hamiltonian, expressed with the Dynkin formula. But more generally, in the field of Lie algebra, there are numerous studies which state and prove sufficient criteria for the Dynkin formula to converge for two elements $x,y$ in a normed Lie algebra $\left(L,\lnorm\cdot\rnorm\right)$ \citep{Blanes2004,Biagi2018,Suzuki1977,Casas}. These studies require a coercivity bound of the form $\lnorm \left[x,y\right]\rnorm\leq \mu\lnorm x\rnorm\lnorm y\rnorm$ for some $\mu\geq 0$. To apply these results directly to the Poisson algebra over $C^{\infty}\left(\Z\right)$, we would require at least something of the form $\lnorm\left\{F,G\right\}\rnorm=\lnorm\nabla_{q}F\cdot\nabla_{p}G-\nabla_{p}F\cdot\nabla_{q}G\rnorm\leq\mu\lnorm F\rnorm\lnorm G\rnorm$ for some suitable function norm $\lnorm\cdot\rnorm$ and $\mu\geq 0$, e.g., some $L^{p}\left(\Z\right)$ norm or Hölder space norm \citep{evans10,Lax2002-qt,Krylov1996-kk}. However, this effectively amounts to enforcing a reverse Poincaré inequality over the function space in question, or equivalently, that the derivative map is bounded over the function space in question, which we would only generally expect to be true over a finite-dimensional function space \citep{Lax2002-qt,Reed1981-vt}. There are a few known examples of this, e.g., the quadratic case discussed in Section \ref{Quadratic} and Appendix \ref{MHQC_BrutalCalc}. In such cases, it would be most straightforward to find a basis for the \textit{finite}-dimensional function space $\mathcal{F}$ such that the Poisson bracket can be expressed as a matrix acting on this basis.

To avoid the aforementioned issues in imposing that the Poisson bracket is bounded (or similarly, that the derivative operator is bounded) over the underlying function space, we derive criteria for convergence that rely on imposing growth conditions on the IPBs of $F$ and $G$ instead of directly bounding the Poisson bracket or derivative operator. Furthermore, we assume absolute convergence, since doing so both allows us to reorder the terms in the Dynkin series \eqref{dynkin} as needed \citep{Ross2013} and allows us to bound the absolute value of the IPBs \eqref{por} in each term, as shown in the start of Appendix \ref{appC}. 

We consider the following two growth assumptions on the IPBs in the Dynkin series \eqref{dynkin}:
\begin{enumerate}
    \item $\left|\left\{G^{r_{1}}F^{s_{1}}G^{r_{2}}F^{s_{2}}\cdots G^{r_{n}}F^{s_{n}}\right\}\right|$ is bounded above by a function of $k\coloneqq r_{1}+\dots+r_{n}+s_{1}+\dots+s_{n}$ only, and
    \item $\left|\left\{G^{r_{1}}F^{s_{1}}G^{r_{2}}F^{s_{2}}\cdots G^{r_{n}}F^{s_{n}}\right\}\right|$ is bounded above by the product of a function of $k$ and $\prod_{i=1}^{n}{r_{i}!s_{i}!}$.
\end{enumerate}
In either case, the bound may also depend on $\left(p,q\right)$, in which case $\widetilde{H}_{\eta}$ is defined as a pointwise limit function at each $\left(p,q\right)$ separately but is not necessarily continuous. However, if the bound holds uniformly for all $\left(p,q\right)\in \Z$, then by the Weierstrass M-test \citep{Ross2013}, $\widetilde{H}_{\eta}$ is also continuous on all of $\Z$.

\paragraph{Case 1}

In this case, we assume a uniform bound on the IPBs which is a function of the ``order'' of the IPB---i.e., we assume that 
\begin{gather}
    \left|\left\{G^{r_{1}}F^{s_{1}}G^{r_{2}}F^{s_{2}}\cdots G^{r_{n}}F^{s_{n}}\right\}\right|\leq A\left(r_{1}+\dots+r_{n}+s_{1}+\dots+s_{n}\right)\label{right}
\end{gather}
for all $n\in\mathbb{Z}_{+}$ and pairs $r_{i}+s_{i}>0$ for $i=1,\dots,n$, where $A:\mathbb{Z}_{+}\rightarrow\R_{\geq 0}$. Doing so in general gives Lemma \ref{case1_general}, from which we can derive two corollaries upon further assumptions for the function $A$ which use explicitly verifiable bounds on the size of $\eta$.

\begin{theorem}\label{port1}
    Suppose that there exist mappings $b,c:\Z\rightarrow\R_{\geq 0}$ such that
\begin{multline}
    \left|\left\{G^{r_{1}}F^{s_{1}}G^{r_{2}}F^{s_{2}}\cdots G^{r_{n}}F^{s_{n}}\right\}\left(p,q\right)\right|\leq \\b\left(p,q\right)\left(r_{1}+\dots+r_{n}+s_{1}+\dots+s_{n}\right)^{2}c\left(p,q\right)^{r_{1}+\dots+r_{n}+s_{1}+\dots+s_{n}-1}\label{bost}
\end{multline}
for all $n\in\mathbb{Z}_{+}$, $\left(p,q\right)\in \Z$, and integer pairs $r_{i}+s_{i}>0$ for $i=1,\dots,n$. Then, for all $\left(p,q\right)\in \Z$ such that $\eta<\log 2/\left(2c\left(p,q\right)\right)$,
\begin{gather}
    \left|\widetilde{H}_{\eta}\left(p,q\right)\right|\leq \frac{2b\left(p,q\right)e^{2\eta c\left(p,q\right)}}{2-e^{2\eta c\left(p,q\right)}}.\label{boundedd}
\end{gather}
\end{theorem}

\begin{proof}
    Replace $A$ with \eqref{bost} in Lemma \ref{case1_general}. Then, we note that the exponential generating function of the Fubini numbers is \citep{Quaintance2015}
$$
\sum_{n=0}^{\infty}{a_{n}\frac{x^{n}}{n!}}=\frac{1}{2-e^{x}},
$$
which converges only when $x<\log 2$. Simplifying the result, we get \eqref{boundedd}. 
\end{proof}

\begin{theorem}\label{port2}
    Suppose that there exist mappings $b,c:\Z\rightarrow\R_{\geq 0}$ and some $t\in\R$ such that
\begin{multline}    
\left|\left\{G^{r_{1}}F^{s_{1}}G^{r_{2}}F^{s_{2}}\cdots G^{r_{n}}F^{s_{n}}\right\}\left(p,q\right)\right|\leq \\
b\left(p,q\right)\left(r_{1}+\dots+r_{n}+s_{1}+\dots+s_{n}\right)^{2-t}c\left(p,q\right)^{r_{1}+\dots+r_{n}+s_{1}+\dots+s_{n}}\label{bosted}
\end{multline}
for all $n\in\mathbb{Z}_{+}$, $\left(p,q\right)\in \Z$, and integer pairs $r_{i}+s_{i}>0$ for $i=1,\dots,n$. Then, for all $\left(p,q\right)\in\Z$ such that $\eta<\log 2/\left(2c\left(p,q\right)\right)$, $\widetilde{H}_{\eta}\left(p,q\right)$ is well-defined and bounded, and moreover, $\widetilde{H}_{\eta}\left(p,q\right)$ is continuous on every open subset of $\Z$ where the upper bound \eqref{bosted} holds uniformly for all $\left(p,q\right)$ in that open set. Moreover, the following bound holds for all $\left(p,q\right)\in\Z$:
\begin{gather}
    \left|\widetilde{H}_{\eta}\left(p,q\right)\right|\leq\frac{2b\left(p,q\right)}{\eta}\left(\textrm{Li}_{t}\left(\frac{2\eta c\left(p,q\right)}{\log 2}\right)-\eta\right),\nonumber
\end{gather}
where $\textrm{Li}_{t}$ is the polylogarithm of order $t$ \citep{poly,Abramowitz1965}.
\end{theorem}

\begin{proof}
    We replace $A$ with \eqref{bosted} in Lemma \ref{case1_general}, thereby giving us
\begin{gather} 
    \left|\widetilde{H}_{\eta}\left(p,q\right)\right|\leq b\left(p,q\right)\left(4\sum_{k=1}^{\infty}{\frac{\eta^{k-1}2^{k-1}c\left(p,q\right)^{k}a_{k-1}}{k^{\top}\left(k-1\right)!}}-2\right).\label{well}
\end{gather}
We have the following upper bound for the Fubini numbers from \cite{Zou2017}:
\begin{gather}
    a_{k-1}<\frac{\left(k-1\right)!}{\left(\log 2\right)^{k}},\label{bdd}
\end{gather}
and using \eqref{bdd} to bound \eqref{well} from above,
\begin{gather}
    \left|\widetilde{H}_{\eta}\left(p,q\right)\right|\leq2b\left(p,q\right)\left(\frac{1}{\eta}\sum_{k=1}^{\infty}{\frac{1}{k^{\top}}\left(\frac{2\eta c\left(p,q\right)}{\log 2}\right)^{k}}-1\right).\label{polylog}
\end{gather}
By the definition of polylogarithm \citep{poly,Abramowitz1965}, the result follows.
\end{proof}

\paragraph{Case 2}

We now assume that 
\begin{gather}
    \left|\left\{G^{r_{1}}F^{s_{1}}G^{r_{2}}F^{s_{2}}\cdots G^{r_{n}}F^{s_{n}}\right\}\right|\leq B\left(r_{1}+\dots+r_{n}+s_{1}+\dots+s_{n}\right)\prod_{i=1}^{n}{r_{i}!s_{i}!}\label{trombone}
\end{gather}
for all $n\in\mathbb{Z}_{+}$ and pairs $r_{i}+s_{i}>0$ for $i=1,\dots,n$, where $B:\mathbb{Z}_{+}\rightarrow\R_{\geq 0}$. Doing so gives the more general convergence criteria stated as Lemma \ref{case2_general}, and we can also derive a corollary from Lemma \ref{case2_general} upon further assumptions on the function $B$.

\begin{theorem}\label{eps}
    Suppose that there exist mappings $b,c:\Z\rightarrow\R_{\geq 0}$ and some $t\in\R$ such that
\begin{multline}
    \left|\left\{G^{r_{1}}F^{s_{1}}G^{r_{2}}F^{s_{2}}\cdots G^{r_{n}}F^{s_{n}}\right\}\left(p,q\right)\right|\leq \\ b\left(p,q\right)\left(r_{1}+\dots+r_{n}+s_{1}+\dots+s_{n}\right)^{1-t}c\left(p,q\right)^{r_{1}+\dots+r_{n}+s_{1}+\dots+s_{n}}\prod_{i=1}^{n}{r_{i}!s_{i}!}\label{boosted}
\end{multline}
for all $n\in\mathbb{Z}_{+}$, $\left(p,q\right)\in \Z$, and integer pairs $r_{i}+s_{i}>0$ for $i=1,\dots,n$. Then, for all $\left(p,q\right)\in\Z$ such that $\eta<\left(2-\sqrt{3}\right)/c\left(p,q\right)$, $\widetilde{H}_{\eta}\left(p,q\right)$ is well-defined and bounded, and moreover, $\widetilde{H}_{\eta}$ is continuous on every open subset of $\Z$ where the upper bound \eqref{boosted} holds uniformly for all $\left(p,q\right)$ in that open set. Moreover, the following bound holds for all $\left(p,q\right)\in\Z$:
\begin{gather}
    \left|\widetilde{H}_{\eta}\left(p,q\right)\right|\leq \frac{b\left(p,q\right)}{\eta}\textrm{Li}_{t}\left(\eta\left(2+\sqrt{3}\right)c\left(p,q\right)\right),\nonumber
\end{gather}
where $\textrm{Li}_{t}$ is the polylogarithm of order $t$.
\end{theorem}

\subsection{Proof of the two convergence criteria}\label{appC}

For each case, we re-sum the modified Hamiltonian \eqref{dynkin} such that the outermost summation sums over consecutive orders of $\eta$. Doing so also gives each order-by-order correction in $\eta$ to the modified Hamiltonian and a way to truncate the modified Hamiltonian at a given order. 

Using the notations introduced in Section \ref{MH_more} and the double summation identity $\sum_{n=m}^{N}{\sum_{k=n}^{N}{a_{nk}}}=\sum_{k=m}^{N}{\sum_{n=m}^{k}{a_{nk}}}$, for any $m,N\in\mathbb{Z}$ with $m\leq N$ \citep{Graham1994}, we have
\begin{align}
    \widetilde{H}_{\eta}^{\left(N-1\right)}\left(p,q\right)
    &=\sum_{n=1}^{\infty}{\frac{\left(-1\right)^{n-1}}{n}\sum_{k=n}^{N}{\sum_{\substack{ r_{1}+s_{1}>0 \\ \vdots \\ r_{n}+s_{n}>0 \\ r_{1}+\dots+r_{n}+s_{1}+\dots+s_{n}=k}}{\frac{\eta^{k-1}\left\{G^{r_{1}}F^{s_{1}}G^{r_{2}}F^{s_{2}}\cdots G^{r_{n}}F^{s_{n}}\right\}}{k\prod_{i=1}^{n}{r_{i}!s_{i}!}}}}} \nonumber\\
    &=\sum_{k=1}^{N}{\frac{\eta^{k-1}}{k}\sum_{n=1}^{k}{\frac{\left(-1\right)^{n-1}}{n}\sum_{\substack{ r_{1}+s_{1}>0 \\ \vdots \\ r_{n}+s_{n}>0 \\ r_{1}+\dots+r_{n}+s_{1}+\dots+s_{n}=k}}{\frac{\left\{G^{r_{1}}F^{s_{1}}G^{r_{2}}F^{s_{2}}\cdots G^{r_{n}}F^{s_{n}}\right\}}{\prod_{i=1}^{n}{r_{i}!s_{i}!}}}}}.\label{best}
\end{align}
Hence, taking $N\rightarrow\infty$ in \eqref{best} and assuming absolute convergence, \eqref{dynkin} is equal to
\begin{gather}
    \widetilde{H}_{\eta}\left(p,q\right)=\sum_{k=1}^{\infty}{\frac{\eta^{k-1}}{k}\sum_{n=1}^{k}{\frac{\left(-1\right)^{n-1}}{n}\sum_{\substack{ r_{1}+s_{1}>0 \\ \vdots \\ r_{n}+s_{n}>0 \\ r_{1}+\dots+r_{n}+s_{1}+\dots+s_{n}=k}}{\frac{\left\{G^{r_{1}}F^{s_{1}}G^{r_{2}}F^{s_{2}}\cdots G^{r_{n}}F^{s_{n}}\right\}}{\prod_{i=1}^{n}{r_{i}!s_{i}!}}}}}.\label{bast}
\end{gather}
Since the steps above are reversible, in practice, once establishing absolute convergence of \eqref{bast}, we can reorder the terms back to establish absolute convergence for the original modified Hamiltonian \eqref{dynkin}. Moreover, by looking at \eqref{best}, we observe that the order $\left(N-1\right)$ correction in the modified Hamiltonian is 
\begin{gather}
    \frac{\eta^{N-1}}{N}\sum_{n=1}^{N}{\frac{\left(-1\right)^{n-1}}{n}\sum_{\substack{ r_{1}+s_{1}>0 \\ \vdots \\ r_{n}+s_{n}>0 \\ r_{1}+\dots+r_{n}+s_{1}+\dots+s_{n}=N}}{\frac{\left\{G^{r_{1}}F^{s_{1}}G^{r_{2}}F^{s_{2}}\cdots G^{r_{n}}F^{s_{n}}\right\}}{\prod_{i=1}^{n}{r_{i}!s_{i}!}}}}.\label{termorder}
\end{gather}

By applying the triangle inequality to \eqref{bast}, we start out with
\begin{gather}
    \left|\widetilde{H}_{\eta}\left(p,q\right)\right|\leq\sum_{k=1}^{\infty}{\frac{\eta^{k-1}}{k}\sum_{n=1}^{k}{\frac{1}{n}\sum_{\substack{ r_{1}+s_{1}>0 \\ \vdots \\ r_{n}+s_{n}>0 \\ r_{1}+\dots+r_{n}+s_{1}+\dots+s_{n}=k}}{\frac{\left|\left\{G^{r_{1}}F^{s_{1}}G^{r_{2}}F^{s_{2}}\cdots G^{r_{n}}F^{s_{n}}\right\}\right|}{\prod_{i=1}^{n}{r_{i}!s_{i}!}}}}}\label{tritri}
\end{gather}
in either case. From here, the game is to bound the inner two sums (i.e., \eqref{termorder}) as tightly as is reasonably possible while still giving us a tractable expression, and then check the outermost sum for convergence.

\paragraph{Case 1}

Given the upper bound \eqref{right} in case 1, \eqref{tritri} can be bounded from above like
\begin{gather}
    \left|\widetilde{H}_{\eta}\left(p,q\right)\right|\leq\sum_{k=1}^{\infty}{\frac{\eta^{k-1}A\left(k\right)}{k}\sum_{n=1}^{k}{\frac{1}{n}\sum_{\substack{ r_{1}+s_{1}>0 \\ \vdots \\ r_{n}+s_{n}>0 \\ r_{1}+\dots+r_{n}+s_{1}+\dots+s_{n}=k}}{\frac{1}{\prod_{i=1}^{n}{r_{i}!s_{i}!}}}}}.\label{hearye}
\end{gather}
Rather surprisingly, we can express the innermost sum in \eqref{hearye} in closed form, as formalized in the lemma below:

\begin{lemma}\label{lemma_1}
For all $n,k\in\mathbb{Z}_{+}$ such that $k\geq n$, 
\begin{gather}
    \sum_{\substack{ r_{1}+s_{1}>0 \\ \vdots \\ r_{n}+s_{n}>0 \\ r_{1}+\dots+r_{n}+s_{1}+\dots+s_{n}=k}}{\frac{k!}{\prod_{i=1}^{n}{r_{i}!s_{i}!}}}=n!2^{k}\stirlingtwo{k}{n},\label{BOO}
\end{gather}
where $\stirlingtwo{k}{n}$ is a Stirling number of the second kind---i.e., the number of ways to partition a set of $k$ objects into $n$ non-empty subsets \citep{Graham1994}.
\end{lemma}

\begin{proof}
    By the residue theorem \citep{Lang1999}, $\oint_{\left|z\right|=1}{\frac{1}{z^{1+p}}\frac{dz}{2\pi i}}=\delta_{p0}$, where $\delta_{ij}$ is the Kronecker delta \citep{Kronecker}. Thus, using properties of the Kronecker delta, we can simplify the left side of \eqref{BOO} as follows:
\begin{gather}
    \sum_{\substack{ r_{1}+s_{1}>0 \\ \vdots \\ r_{n}+s_{n}>0 \\ r_{1}+\dots+r_{n}+s_{1}+\dots+s_{n}=k}}{\frac{k!}{\prod_{m=1}^{n}{r_{m}!s_{m}!}}}=k!\sum_{\substack{ r_{1}+s_{1}\geq 1 \\ \vdots \\ r_{n}+s_{n}\geq 1 \\ r_{1}+\dots+r_{n}+s_{1}+\dots+s_{n}=k}}{\frac{1}{\prod_{m=1}^{n}{r_{m}!s_{m}!}}}\nonumber\\=k!\sum_{\substack{ r_{1}+s_{1}\geq 1 \\ \vdots \\ r_{n}+s_{n}\geq 1}}{\frac{\delta_{\left(k-\sum_{m=1}^{n}{\left(r_{m}+s_{m}\right)}\right)0}}{\prod_{m=1}^{n}{r_{m}!s_{m}!}}}
    =k!\sum_{\substack{ r_{j}+s_{j}\geq 1 \\ j=1,\dots,n}}{\frac{1}{\prod_{m=1}^{n}{r_{m}!s_{m}!}}\oint_{\left|z\right|=1}{\frac{1}{z^{k+1-\sum_{m=1}^{n}{\left(r_{m}+s_{m}\right)}}}\frac{dz}{2\pi i}}}\nonumber\\
    =k!\oint_{\left|z\right|=1}{\frac{1}{z^{k+1}}\sum_{\substack{ r_{j}+s_{j}\geq 1 \\ j=1,\dots,n}}{\prod_{m=1}^{n}{\frac{z^{r_{m}}z^{s_{m}}}{r_{m}!s_{m}!}}}\frac{dz}{2\pi i}}=k!\oint_{\left|z\right|=1}{\frac{1}{z^{k+1}}\left(\sum_{r+s\geq 1}{\frac{z^{r}z^{s}}{r!s!}}\right)^{n}\frac{dz}{2\pi i}}.\label{residue}
\end{gather}
The sum in \eqref{residue} can be simplified as follows:
\begin{gather}
    \sum_{r+s\geq 1}{\frac{z^{r}z^{s}}{r!s!}}=\sum_{r+s=1}^{\infty}{\frac{z^{r+s}}{r!s!}}=\sum_{j=1}^{\infty}{\sum_{r+s=j}{\frac{z^{j}}{r!s!}}}=\sum_{j=1}^{\infty}{\frac{z^{j}}{j!}\sum_{r+s=j}{\frac{j!}{r!s!}}}=\sum_{j=1}^{\infty}{\frac{z^{j}}{j!}\sum_{s=0}^{j}{\binom{j}{s}}}.\label{binom}
\end{gather}
Then, using the binomial theorem and the exponential series $e^{x}=\sum_{j=0}^{\infty}{x^{j}/j!}$ \citep{Abramowitz1965} to simplify \eqref{binom}, 
\begin{gather}
    \sum_{j=1}^{\infty}{\frac{z^{j}}{j!}2^{j}}=\sum_{j=1}^{\infty}{\frac{\left(2z\right)^{j}}{j!}}=e^{2z}-\frac{\left(2z\right)^{0}}{0!}=e^{2z}-1.\label{simped}
\end{gather}
Substituting \eqref{simped} into \eqref{residue} and applying Cauchy's differentiation formula \citep{Lang1999}, we get that 
\begin{gather}
    \sum_{\substack{ r_{1}+s_{1}>0 \\ \vdots \\ r_{n}+s_{n}>0 \\ r_{1}+\dots+r_{n}+s_{1}+\dots+s_{n}=k}}{\frac{k!}{\prod_{m=1}^{n}{r_{m}!s_{m}!}}}=\frac{k!}{2\pi i}\oint_{\left|z\right|=1}{\frac{\left(e^{2z}-1\right)^{n}}{z^{k+1}}dz}=\frac{d^{(k)}}{dz^{(k)}}\left[\left(e^{2z}-1\right)^{n}\right]_{z=0}.\label{backthen}
\end{gather}
Finally, using the binomial theorem, exponential series, and then the definition of Stirling numbers of the second kind \citep{Graham1994} to expand \eqref{backthen},
\begin{align*}
    \left(e^{2z}-1\right)^{n}&=\sum_{l=0}^{n}{\binom{n}{l}\left(e^{2z}\right)^{n-l}\left(-1\right)^{l}}=\sum_{l=0}^{n}{\binom{n}{l}e^{2z\left(n-l\right)}\left(-1\right)^{l}}.
\end{align*}
This implies for all $k > 0$:
\begin{align*}    
    \frac{d^{(k)}}{dz^{(k)}}\left[\left(e^{2z}-1\right)^{n}\right]_{z=0}&=\sum_{l=0}^{n}{\binom{n}{l}\frac{d^{(k)}}{dz^{(k)}}\left.e^{2z\left(n-l\right)}\right|_{z=0}\left(-1\right)^{l}} \\
    &=\sum_{l=0}^{n}{\binom{n}{l}2^{k}\left(n-l\right)^{k}\left.e^{2z\left(n-l\right)}\right|_{z=0}\left(-1\right)^{l}} \\ 
    &=n!2^{k}\frac{1}{n!}\sum_{l=0}^{n}{\left(-1\right)^{l}\binom{n}{l}\left(n-l\right)^{k}} \\
    &=n!2^{k}\stirlingtwo{k}{n},
\end{align*}
as we wanted to show. 
\end{proof}

 With Lemma \ref{lemma_1} proven, we can now proceed with the proof of Lemma \ref{case1_general}.

\begin{lemma}\label{case1_general}
    Suppose that there exists some $A:\mathbb{Z}_{+}\times\Z\rightarrow\R_{\geq 0}$ such that $$\left|\left\{G^{r_{1}}F^{s_{1}}G^{r_{2}}F^{s_{2}}\cdots G^{r_{n}}F^{s_{n}}\right\}\left(p,q\right)\right|\leq A\left(r_{1}+\dots+r_{n}+s_{1}+\dots+s_{n},p,q\right)$$ for all $n\in\mathbb{Z}_{+}$, $\left(p,q\right)\in \Z$, and integer pairs $r_{i}+s_{i}>0$ for $i=1,\dots,n$. Then, for all $\left(p,q\right)\in\Z$,
\begin{gather}
    \left|\widetilde{H}_{\eta}\left(p,q\right)\right|\leq\sum_{k=1}^{\infty}{\frac{\eta^{k-1}2^{k+1}A\left(k,p,q\right)a_{k-1}}{k^{2}\left(k-1\right)!}}-2A\left(1,p,q\right),
\end{gather}
where $a_{k-1}$ is the $(k-1)$-th Fubini number \citep{Zou2017,KeresknyiBalogh2021} (also known as ordered Bell number \citep{tanny_1975}) for all $k\in\mathbb{Z}_{+}$.
\end{lemma}

\begin{proof} Using Lemma \ref{lemma_1} and properties of Stirling numbers of the second kind \citep{Graham1994} to simplify \eqref{hearye},
\begin{align}
    \left|\widetilde{H}_{\eta}\left(p,q\right)\right|&\leq\sum_{k=1}^{\infty}{\frac{\eta^{k-1}A\left(k\right)}{k!k}\sum_{n=1}^{k}{\frac{n!2^{k}}{n}\stirlingtwo{k}{n}}} \nonumber \\
    &=\sum_{k=1}^{\infty}{\frac{\eta^{k-1}2^{k}A\left(k\right)}{k!k}\left(\left(k-1\right)!\stirlingtwo{k}{k}+\sum_{n=1}^{k-1}{\left(n-1\right)!\stirlingtwo{k}{n}}\right)}\nonumber\\
    &\stackrel{(i)}{=}\sum_{k=1}^{\infty}{\frac{\eta^{k-1}2^{k}A\left(k\right)}{k!k}\left[\left(k-1\right)!+\sum_{n=1}^{k-1}{\left(n-1\right)!\left(n\stirlingtwo{k-1}{n}+\stirlingtwo{k-1}{n-1}\right)}\right]}\nonumber\\
    &=\sum_{k=1}^{\infty}{\frac{\eta^{k-1}2^{k}A\left(k\right)}{k!k}\left(\sum_{n=1}^{k-1}{n!\stirlingtwo{k-1}{n}}+\sum_{n=1}^{k}{\left(n-1\right)!\stirlingtwo{k-1}{n-1}}\right)}\nonumber\\
    &\stackrel{(ii)}{=} \sum_{k=1}^{\infty}{\frac{\eta^{k-1}2^{k}A\left(k\right)}{k!k}\left(2\sum_{m=0}^{k-1}{m!\stirlingtwo{k-1}{m}}-\delta_{(k-1)0}\right)}.\label{recurrence}
\end{align}
In the above, the step $(i)$ follows from the identities $\stirlingtwo{k}{k}=1$ and $\stirlingtwo{k}{n}=n\stirlingtwo{k-1}{n}+\stirlingtwo{k-1}{n-1}$;
while the step $(ii)$ follows from $\stirlingtwo{k-1}{0}=\delta_{(k-1)0}$.
Finally, the Fubini numbers $a_{n}$ can be computed from the Stirling numbers of the second kind via the formula \citep{Quaintance2015,KeresknyiBalogh2021}
\begin{gather}
    a_{n}=\sum_{m=0}^{n}{m!\stirlingtwo{n}{m}}\quad \forall n\in\N_{0}.\label{Fubini}
\end{gather}
Substituting \eqref{Fubini} into \eqref{recurrence}, we are done.
\end{proof}

\paragraph{Case 2}

\begin{lemma}\label{case2_general}
    Suppose that there exists some $B:\mathbb{Z}_{+}\times\Z\rightarrow\R_{\geq 0}$ such that $$\left|\left\{G^{r_{1}}F^{s_{1}}G^{r_{2}}F^{s_{2}}\cdots G^{r_{n}}F^{s_{n}}\right\}\left(p,q\right)\right|\leq B\left(r_{1}+\dots+r_{n}+s_{1}+\dots+s_{n},p,q\right)\prod_{i=1}^{n}{r_{i}!s_{i}!}$$ for all $n\in\mathbb{Z}_{+}$, $\left(p,q\right)\in \Z$, and integer pairs $r_{i}+s_{i}>0$ for $i=1,\dots,n$. Then, for all $\left(p,q\right)\in\Z$ such that
$$
\sum_{k=1}^{\infty}{\frac{\left(\eta\left(2+\sqrt{3}\right)\right)^{k}B\left(k,p,q\right)}{k^{2}}}<\infty,
$$
the modified Hamiltonian $\widetilde{H}_{\eta}\left(p,q\right)$ converges. Moreover, the following bounds hold for all $\left(p,q\right)\in\Z$:
$$
\left|\widetilde{H}_{\eta}\left(p,q\right)\right|\leq \sum_{k=1}^{\infty}{\frac{\eta^{k-1}B\left(k,p,q\right)}{k^{2}}\sum_{n=1}^{k}{\prod_{m=0}^{n-1}{\frac{2\left(k^{2}-m^{2}\right)}{n^{2}-m^{2}}}}}\leq \sum_{k=1}^{\infty}{\frac{\eta^{k-1}B\left(k,p,q\right)}{k}\left(2+\sqrt{3}\right)^{k}}.
$$
\end{lemma}

Before proving Lemma \ref{case2_general}, we need to show a few more properties. Assuming the upper bound \eqref{trombone}, \eqref{tritri} can be bounded above like
\begin{gather}
    \left|\widetilde{H}_{\eta}\left(p,q\right)\right|\leq\sum_{k=1}^{\infty}{\frac{\eta^{k-1}B\left(k\right)}{k}\sum_{n=1}^{k}{\frac{1}{n}\sum_{\substack{ r_{1}+s_{1}>0 \\ \vdots \\ r_{n}+s_{n}>0 \\ r_{1}+\dots+r_{n}+s_{1}+\dots+s_{n}=k}}{1}}}.\label{wells}
\end{gather}
Simplifying the innermost sum in \eqref{wells} is equivalent with a combinatorics problem: counting the number of $r_{1},\dots,r_{n},s_{1},\dots,s_{n}\in\N_{0}$ such that $r_{1}+s_{1},\dots,r_{n}+s_{n}>0$ and $r_{1}+\dots+r_{n}+s_{1}+\dots+s_{n}=k$. Analogously, this is equivalent with counting the number of ways that we can fit $k$ indistinguishable balls into $n$ ordered red and $n$ ordered blue bins such that at least one ball ends up in the first red or blue bin, at least one ball ends up in the second red or blue bin, and so on. In turn, this is equal to the number of ways we can fit $n$ out of $k$ total indistinguishable balls into $n$ ordered red and $n$ ordered blue bins such that at least one ball ends up in the first red or blue bin, at least one ball ends up in the second red or blue bin, and so on, multiplied by the number of ways we can fit the remaining balls into the bins. The first of these quantities is $2^{n}$, as for each pair of red and blue bins, we can place one ball into either of the two bins, and there are a total of $n$ pairs of red and blue bins. Hence, the second of these is the number of ways to put $\left(k-n\right)$ unlabeled balls into $2n$ distinct bins---i.e., \citep{Feller1968}
\begin{gather}
    \binom{k-n+2n-1}{2n-1}=\binom{k+n-1}{2n-1}=\frac{\left(k+n-1\right)!}{\left(k-n\right)!\left(2n-1\right)!},\label{theremin}
\end{gather}
but we can further rewrite \eqref{theremin} as the following:
\begin{gather}
    \frac{\left(k+n-1\right)\left(k+n-2\right)\cdots\left(k-n+1\right)\left(k-n\right)!}{\left(k-n\right)!\left(2n-1\right)!}
    =\frac{nk^{2}\left(k^{2}-1^{2}\right)\cdots\left(k^{2}-\left(n-1\right)^{2}\right)}{kn^{2}\left(n^{2}-1^{2}\right)\cdot\left(n^{2}-\left(n-1\right)^{2}\right)}.\label{centralfact}
\end{gather}
Thus, \eqref{wells} can be simplified to
\begin{align}
    \left|\widetilde{H}_{\eta}\left(p,q\right)\right|&\leq\sum_{k=1}^{\infty}{\frac{\eta^{k-1}B\left(k\right)}{k^{2}}\sum_{n=1}^{k}{2^{n}\frac{\left(k^{2}-0^{2}\right)\left(k^{2}-1^{2}\right)\cdots\left(k^{2}-\left(n-1\right)^{2}\right)}{\left(n^{2}-0^{2}\right)\left(n^{2}-1^{2}\right)\cdot\left(n^{2}-\left(n-1\right)^{2}\right)}}}\nonumber\\
    &=\sum_{k=1}^{\infty}{\frac{\eta^{k-1}B\left(k\right)}{k^{2}}\sum_{n=1}^{k}{\prod_{m=0}^{n-1}{\frac{2\left(k^{2}-m^{2}\right)}{n^{2}-m^{2}}}}}.\label{sumprod}
\end{align}
As far as the authors know, the inner sum and product in \eqref{sumprod} cannot be expressed in closed form. 
However, we can use Riemann sums and properties of the LogSumExp function~\citep{Ghaoui} to produce an upper bound which is saturated as $k\rightarrow\infty$ as follows:

\begin{lemma}\label{backseat} For every $r>0$,
\begin{gather}
    \lim_{k\rightarrow\infty}{\frac{\sum_{n=1}^{k}{\prod_{m=0}^{n-1}{\frac{r\left(k^{2}-m^{2}\right)}{n^{2}-m^{2}}}}}{\left(\frac{2+r+\sqrt{r\left(4+r\right)}}{2}\right)^{k}}}=1,\label{backus}
\end{gather}
or equivalently,
\begin{gather}
    \lim_{k\rightarrow\infty}{\frac{1}{k}\log\left(\sum_{n=1}^{k}{\prod_{m=0}^{n-1}{\frac{r\left(k^{2}-m^{2}\right)}{n^{2}-m^{2}}}}\right)}=\log\left(\frac{2+r+\sqrt{r\left(4+r\right)}}{2}\right).\label{more}
\end{gather}
Furthermore, we have the following upper bound for all $k\in\N$: 
\begin{gather}
    \sum_{n=1}^{k}{\prod_{m=0}^{n-1}{\frac{r\left(k^{2}-m^{2}\right)}{n^{2}-m^{2}}}}\leq k\left(\frac{2+r+\sqrt{r\left(4+r\right)}}{2}\right)^{k}.\label{MOREE}
\end{gather}
\end{lemma}

\begin{proof} The equivalence between \eqref{backus} and \eqref{more} and the implication from
\begin{align}
    &\max\left\{\int_{0}^{\max\left\{0,\left \lfloor{kx_{+}}\right \rfloor-1\right\}/k}{\log\left(\frac{r\left(1-y^{2}\right)}{\left(\max\left\{\frac{1}{k},\frac{\left \lfloor{kx_{+}}\right \rfloor}{k}\right\}\right)^{2}-y^{2}}\right)dy}-\frac{1}{k}\log\left(\frac{\left(\max\left\{\frac{1}{k},\frac{\left \lfloor{kx_{+}}\right \rfloor}{k}\right\}\right)^{2}}{r}\right), \right. \nonumber\\ &\hspace{2.5cm}\left. \int_{0}^{\left(\left \lceil{kx_{+}}\right \rceil-1\right)/k}{\log\left(\frac{r\left(1-y^{2}\right)}{\left(\frac{\left \lceil{kx_{+}}\right \rceil}{k}\right)^{2}-y^{2}}\right)dy}-\frac{1}{k}\log\left(\frac{\left(\frac{\left \lceil{kx_{+}}\right \rceil}{k}\right)^{2}}{r}\right)\right\} \nonumber\\ \leq &\frac{1}{k}\log\left(\sum_{n=1}^{k}{\prod_{m=0}^{n-1}{\frac{r\left(k^{2}-m^{2}\right)}{n^{2}-m^{2}}}}\right)\nonumber\\ \leq &\max\left\{\int_{0}^{\max\left\{\frac{1}{k},\frac{\left \lfloor{kx_{+}}\right \rfloor}{k}\right\}}{\log\left(\frac{r\left(1-y^{2}\right)}{\left(\max\left\{\frac{1}{k},\frac{\left \lfloor{kx_{+}}\right \rfloor}{k}\right\}\right)^{2}-y^{2}}\right)dy},\int_{0}^{\frac{\left \lceil{kx_{+}}\right \rceil}{k}}{\log\left(\frac{r\left(1-y^{2}\right)}{\left(\frac{\left \lceil{kx_{+}}\right \rceil}{k}\right)^{2}-y^{2}}\right)dy}\right\}+\frac{\log k}{k}
    \nonumber\\ \leq &\log\left(\frac{2+r+\sqrt{r\left(4+r\right)}}{2}\right)+\frac{\log k}{k},\label{MORE}
\end{align}
where $x_{+}=\sqrt{\frac{r}{4+r}}$, to \eqref{MOREE} follow from sequential continuity and arithmetic properties of the logarithm \citep{Ross2013}. Hence, it suffices to prove \eqref{more} and \eqref{MORE}; at first, we show the latter. For any $x_{1},\dots,x_{k}\in\R$, we have the following lower and upper bounds \citep{Ghaoui}: 
\begin{gather}
    \max\left\{x_{1},\dots,x_{k}\right\}\leq \log\left(\exp\left(x_{1}\right)+\dots+\exp\left(x_{k}\right)\right)\leq \max\left\{x_{1},\dots,x_{k}\right\}+\log k,\label{LSE}
\end{gather}
In particular, if we let $x_{n}\coloneqq \sum_{m=0}^{n-1}{\log\left(\frac{r\left(k^{2}-m^{2}\right)}{n^{2}-m^{2}}\right)}$ for $n=1,\dots,k$ and divide both sides by $1/k$, then \eqref{LSE} implies that
\begin{gather}
    \max_{n=1,\dots,k}\sum_{m=0}^{n-1}{\frac{1}{k}\log\left(\frac{r\left(k^{2}-m^{2}\right)}{n^{2}-m^{2}}\right)}\leq \frac{1}{k}\log\left(\sum_{n=1}^{k}{\prod_{m=0}^{n-1}{\frac{r\left(k^{2}-m^{2}\right)}{n^{2}-m^{2}}}}\right)\nonumber\\ \leq \max_{n=1,\dots,k}\sum_{m=0}^{n-1}{\frac{1}{k}\log\left(\frac{r\left(k^{2}-m^{2}\right)}{n^{2}-m^{2}}\right)}+\frac{\log k}{k}.\label{face}
\end{gather}
The reader might recognize the sum on the leftmost or rightmost sides of \eqref{face} as a Riemann sum for the following improper integral:
\begin{gather}
    \int_{0}^{x}{\log\left(\frac{r\left(1-y^{2}\right)}{x^{2}-y^{2}}\right)dy},\label{improper}
\end{gather}
where $x\coloneqq n/k$. Thus, we work along this vein of thought. At first, the integrand of \eqref{improper} is increasing in $y$ for $0\leq y\leq x\leq 1$. To show this, note that
\begin{gather}
    \frac{\partial}{\partial y}\log\left(\frac{r\left(1-y^{2}\right)}{x^{2}-y^{2}}\right)=\frac{2y\left(1-x^{2}\right)}{\left(1-y^{2}\right)\left(x^{2}-y^{2}\right)},\label{deriv}
\end{gather}
and whenever $0<x<1$, $0<y<1$, and $\left|y\right|<\left|x\right|$, \eqref{deriv} is positive. Thus, for any $n,m\in\N_{0}$ such that $0\leq m\leq n-1$,
\begin{align}
    \frac{1}{k}\log\left(\frac{r\left(k^{2}-m^{2}\right)}{n^{2}-m^{2}}\right)&=\int_{m/k}^{\left(m+1\right)/k}{\log\left(\frac{r\left(k^{2}-m^{2}\right)}{n^{2}-m^{2}}\right)dy}\leq \int_{m/k}^{\left(m+1\right)/k}{\log\left(\frac{r\left(1-y^{2}\right)}{\left(\frac{n}{k}\right)^{2}-y^{2}}\right)dy}\nonumber\\
    \Rightarrow \sum_{m=0}^{n-1}{\frac{1}{k}\log\left(\frac{r\left(k^{2}-m^{2}\right)}{n^{2}-m^{2}}\right)}&\leq\sum_{m=0}^{n-1}{\int_{m/k}^{\left(m+1\right)/k}{\log\left(\frac{r\left(1-y^{2}\right)}{\left(\frac{n}{k}\right)^{2}-y^{2}}\right)dy}}
    =\int_{0}^{x}{\log\left(\frac{r\left(1-y^{2}\right)}{x^{2}-y^{2}}\right)dy}.\label{improp1}
\end{align}
Furthermore, if we instead have that $1\leq m\leq n$, then
\begin{align}
    \frac{1}{k}\log\left(\frac{r\left(k^{2}-m^{2}\right)}{n^{2}-m^{2}}\right)&=\int_{\left(m-1\right)/k}^{m/k}{\log\left(\frac{r\left(k^{2}-m^{2}\right)}{n^{2}-m^{2}}\right)dy}\geq \int_{\left(m-1\right)/k}^{m/k}{\log\left(\frac{r\left(1-y^{2}\right)}{\left(\frac{n}{k}\right)^{2}-y^{2}}\right)dy}\nonumber\\
    \Rightarrow \sum_{m=1}^{n-1}{\frac{1}{k}\log\left(\frac{r\left(k^{2}-m^{2}\right)}{n^{2}-m^{2}}\right)}&\geq\sum_{m=1}^{n-1}{\int_{\left(m-1\right)/k}^{m/k}{\log\left(\frac{r\left(1-y^{2}\right)}{\left(\frac{n}{k}\right)^{2}-y^{2}}\right)dy}}
    =\int_{0}^{x-1/k}{\log\left(\frac{r\left(1-y^{2}\right)}{x^{2}-y^{2}}\right)dy}\nonumber\\
    \Rightarrow \sum_{m=0}^{n-1}{\frac{1}{k}\log\left(\frac{r\left(k^{2}-m^{2}\right)}{n^{2}-m^{2}}\right)}&\geq \int_{0}^{x-1/k}{\log\left(\frac{r\left(1-y^{2}\right)}{x^{2}-y^{2}}\right)dy}-\frac{1}{k}\log\left(\frac{x^{2}}{r}\right).\label{lower}
\end{align}
Since \eqref{improp1} and \eqref{lower} hold for all $n=1,\dots,k$ (or equivalently, for all $x=1/k,\dots,1$), \eqref{face}, \eqref{improp1}, and \eqref{lower} imply that
\begin{multline}    
    \max_{x=1/k,\dots,1}\left\{\int_{0}^{x-1/k}{\log\left(\frac{r\left(1-y^{2}\right)}{x^{2}-y^{2}}\right)dy}-\frac{1}{k}\log\left(\frac{x^{2}}{r}\right)\right\}\leq \frac{1}{k}\log\left(\sum_{n=1}^{k}{\prod_{m=0}^{n-1}{\frac{r\left(k^{2}-m^{2}\right)}{n^{2}-m^{2}}}}\right)\\ 
    \leq \max_{x=1/k,\dots,1}\left\{\int_{0}^{x}{\log\left(\frac{r\left(1-y^{2}\right)}{x^{2}-y^{2}}\right)dy}\right\}+\frac{\log k}{k}.\label{twobounds}
\end{multline}
We will see that the leftmost and rightmost sides of \eqref{twobounds} are close for $k$ large; there is no known closed-form expression for the leftmost side. We can maximize the rightmost side and then use the location of the maximum to approximate the leftmost side---i.e., for all $z=1/k,\dots,1$,
\begin{multline}
    \int_{0}^{z-1/k}{\log\left(\frac{r\left(1-y^{2}\right)}{z^{2}-y^{2}}\right)dy}-\frac{1}{k}\log\left(\frac{z^{2}}{r}\right)\\
    \leq \max_{x=1/k,\dots,1}\left\{\int_{0}^{x-1/k}{\log\left(\frac{r\left(1-y^{2}\right)}{x^{2}-y^{2}}\right)dy}-\frac{1}{k}\log\left(\frac{x^{2}}{r}\right)\right\}\label{bastast}
\end{multline}
Thus, to further simplify \eqref{twobounds}, we should compute $\int_{0}^{x}{\log\left(\frac{r\left(1-y^{2}\right)}{x^{2}-y^{2}}\right)dy}$, maximize the resultant expression over \textit{all} $x\in\left[0,1\right]$, then see how this implies the maximum over $x=1/k,\dots,1$. The derivative of $f(x)=\int_{0}^{x}{\log\left(\frac{r\left(1-y^{2}\right)}{x^{2}-y^{2}}\right)dy}$ (with respect to $x$) is $f'(x)=\log\left(\frac{r\left(1-x^{2}\right)}{4x^{2}}\right)$ which 1) has two zeroes at $x_{\pm}=\pm\sqrt{\frac{r}{4+r}}$ and 2) second derivative $f''(x)=\frac{2}{x^{3}-x}$.  Hence, since $f''(x)<0$ for $0<x<1$ and $f''(x)>0$ for $-1<x<0$, it is clear that $f$ is maximized at $x_{+}=\sqrt{\frac{r}{4+r}}<1$ on the domain $\left[0,1\right]$. Moreover, since $f$ is always concave down on $\left[0,1\right]$, the maximum on the ``grid'' $\left\{1/k,\dots,1\right\}$ will be either at $\frac{\left \lceil{kx_{+}}\right \rceil}{k}$ or $\max\left\{\frac{1}{k},\frac{\left \lfloor{kx_{+}}\right \rfloor}{k}\right\}$, since $x_{+}\in\left[\frac{\left \lfloor{kx_{+}}\right \rfloor}{k},\frac{\left \lceil{kx_{+}}\right \rceil}{k}\right]$---i.e., one of the two multiples of $1/k$ adjacent to $x_{+}$. Thus, \eqref{twobounds} and \eqref{bastast} imply that
\begin{gather}
    \max\left\{\int_{0}^{\max\left\{0,\left \lfloor{kx_{+}}\right \rfloor-1\right\}/k}{\log\left(\frac{r\left(1-y^{2}\right)}{\left(\max\left\{\frac{1}{k},\frac{\left \lfloor{kx_{+}}\right \rfloor}{k}\right\}\right)^{2}-y^{2}}\right)dy}-\frac{1}{k}\log\left(\frac{\left(\max\left\{\frac{1}{k},\frac{\left \lfloor{kx_{+}}\right \rfloor}{k}\right\}\right)^{2}}{r}\right),\right.\nonumber\\ \left.\int_{0}^{\left(\left \lceil{kx_{+}}\right \rceil-1\right)/k}{\log\left(\frac{r\left(1-y^{2}\right)}{\left(\frac{\left \lceil{kx_{+}}\right \rceil}{k}\right)^{2}-y^{2}}\right)dy}-\frac{1}{k}\log\left(\frac{\left(\frac{\left \lceil{kx_{+}}\right \rceil}{k}\right)^{2}}{r}\right)\right\}\nonumber
    \end{gather}
    \begin{align}
    &\leq \frac{1}{k}\log\left(\sum_{n=1}^{k}{\prod_{m=0}^{n-1}{\frac{r\left(k^{2}-m^{2}\right)}{n^{2}-m^{2}}}}\right)\nonumber\\ &\leq \max\left\{\int_{0}^{\max\left\{\frac{1}{k},\frac{\left \lfloor{kx_{+}}\right \rfloor}{k}\right\}}{\log\left(\frac{r\left(1-y^{2}\right)}{\left(\max\left\{\frac{1}{k},\frac{\left \lfloor{kx_{+}}\right \rfloor}{k}\right\}\right)^{2}-y^{2}}\right)dy},\right.\nonumber\\ &\hspace{1.5cm}\left.\int_{0}^{\frac{\left \lceil{kx_{+}}\right \rceil}{k}}{\log\left(\frac{r\left(1-y^{2}\right)}{\left(\frac{\left \lceil{kx_{+}}\right \rceil}{k}\right)^{2}-y^{2}}\right)dy}\right\}
    +\frac{\log k}{k},\label{busta}
\end{align}
which establishes part of \eqref{MORE}. Furthermore, since $\lim_{k\rightarrow\infty}{\max\left\{\frac{1}{k},\frac{\left \lfloor{kx_{+}}\right \rfloor}{k}\right\}}=\lim_{k\rightarrow\infty}{\frac{\left \lceil{kx_{+}}\right \rceil}{k}}=x_{+}$, taking limits of \eqref{busta} gives 
\begin{align}
    \int_{0}^{x_{+}}{\log\left(\frac{r\left(1-y^{2}\right)}{x_{+}^{2}-y^{2}}\right)dy}-\frac{1}{k}\log\left(\frac{x_{+}^{2}}{r}\right)&\leq \frac{1}{k}\log\left(\sum_{n=1}^{k}{\prod_{m=0}^{n-1}{\frac{r\left(k^{2}-m^{2}\right)}{n^{2}-m^{2}}}}\right)\nonumber\\ &\leq \int_{0}^{x_{+}}{\log\left(\frac{r\left(1-y^{2}\right)}{x_{+}^{2}-y^{2}}\right)dy}+\frac{\log k}{k}\label{mess}
\end{align}
for sufficiently large $k$. With help from Wolfram Mathematica \citep{Mathematica}, the integral in \eqref{mess} evaluates to the following:
\begin{gather}
    \log\left(\frac{2+r+\sqrt{r(4+r)}}{2}\right).\label{lim}
\end{gather}
Finally, since $\lim_{k\rightarrow\infty}{\frac{1}{k}\log\left(\frac{x_{+}^{2}}{r}\right)}=\lim_{k\rightarrow\infty}{+\frac{\log k}{k}}=0$, by the squeeze theorem \citep{Ross2013}, \eqref{mess} and \eqref{lim} imply \eqref{more}.

It remains to show the last inequality in \eqref{MORE}. Note that $\frac{\left \lceil{kx_{+}}\right \rceil}{k}$ and $\max\left\{\frac{1}{k},\frac{\left \lfloor{kx_{+}}\right \rfloor}{k}\right\}$ are ``close'' to $x_{+}$, but $f\left(\frac{\left \lceil{kx_{+}}\right \rceil}{k}\right)$ and $f\left(\max\left\{\frac{1}{k},\frac{\left \lfloor{kx_{+}}\right \rfloor}{k}\right\}\right)$ can still only be \textit{at most} $f\left(x_{+}\right)$. Hence, \eqref{busta} implies that
\begin{gather}
    \frac{1}{k}\log\left(\sum_{n=1}^{k}{\prod_{m=0}^{n-1}{\frac{r\left(k^{2}-m^{2}\right)}{n^{2}-m^{2}}}}\right)\leq\log\left(\frac{2+r+\sqrt{r(4+r)}}{2}\right)+\frac{\log k}{k},\nonumber
\end{gather}
as we wanted to show. \end{proof}

\subsubsection{Proof of Lemma~\ref{case2_general}}
Lemma \ref{backseat} gives us a simple way to analyze the convergence and boundedness of the inner sum and product within \eqref{sumprod}. Hence, we have the proof of Lemma~\ref{case2_general}:

\medskip

\begin{proof}[Proof of Lemma~\ref{case2_general}]
    The equation \eqref{backus} from Lemma \ref{backseat} implies by the limit comparison test \citep{Ross2013} that the convergence of \eqref{sumprod} is equivalent with the convergence of
\begin{gather}
    \sum_{k=1}^{\infty}{\frac{\left(\eta\left(2+\sqrt{3}\right)\right)^{k}B\left(k\right)}{k^{2}}}.\nonumber
\end{gather}
Moreover, \eqref{MOREE} from Lemma \ref{backseat} can be used to estimate \eqref{sumprod} as follows:
\begin{gather}
    \sum_{k=1}^{\infty}{\frac{\eta^{k-1}B\left(k\right)}{k^{2}}\sum_{n=1}^{k}{\prod_{m=0}^{n-1}{\frac{2\left(k^{2}-m^{2}\right)}{n^{2}-m^{2}}}}}\leq \sum_{k=1}^{\infty}{\frac{\eta^{k-1}B\left(k\right)}{k}\left(2+\sqrt{3}\right)^{k}},\nonumber
\end{gather}
thereby proving Lemma \ref{case2_general}.
\end{proof}

\subsection{Applications of the convergence criteria}\label{Appendix_ConvergenceCriteria}

We now explore some basic examples of functions $F$ and $G$ and see if the criteria above confirm an absolutely convergent modified Hamiltonian, expressed with Dynkin series in those cases. That being said, we invite readers to look for other functions $F$ and $G$ which satisfy (or fail to satisfy) the convergence criteria outlined in Theorems \ref{port1}, \ref{port2}, and \ref{eps}, or more generally, those from Lemmas \ref{case2_general} and \ref{case1_general}. 

\subsubsection{Quadratic case} 

In Section \ref{Quadratic} and Appendix \ref{MHQC_BrutalCalc}, we computed the closed-form of the modified Hamiltonian, expressed with the Dynkin series when $f$ and $g$ are both quadratic functions. The convergence criteria we proposed in Appendix \ref{Section_AbsCvg} are also satisfied.

\begin{proposition}
    When $F=ap^2, G=dq^2$, we have that
    \begin{gather}
        \left|\left\{G^{r_{1}}F^{s_{1}}G^{r_{2}}F^{s_{2}}\cdots G^{r_{n}}F^{s_{n}}\right\}\left(p,q\right)\right|\leq
        \max\{p^2, q^2, pq\}
        \left(\max\{4a, 4d\}\right)^{r_{1}+\dots+r_{n}+s_{1}+\dots+s_{n}}
        \label{quadcase_boosted}        
    \end{gather}
    Thus, all convergence criteria in Appendix \ref{Section_AbsCvg} are satisfied.
\end{proposition}

\begin{proof}
    We show \eqref{quadcase_boosted} by induction. First, for the base case, it is trivial that
    \begin{gather}
        F=ap^2 \leq \max\{p^2, q^2, pq\}
        \left(\max\{4a, 4d\}\right)^{1}, 
        G=dq^2 \leq \max\{p^2, q^2, pq\}
        \left(\max\{4a, 4d\}\right)^{1}.\nonumber
    \end{gather}
    We now suppose that
    \begin{gather}
        \left|\left\{
        G^{r_{1}}F^{s_{1}}G^{r_{2}}F^{s_{2}}\cdots G^{r_{n}}F^{s_{n}}
        \right\}\left(p,q\right)\right|
        \leq \max\{p^2, q^2, pq\}
        \left(\max\{4a, 4d\}\right)^{r_{1}+\dots+r_{n}+s_{1}+\dots+s_{n}}.  \nonumber   
    \end{gather}
    We may apply $F$ or $G$ to the IPBs. We start with the case of $F$.
    \begin{align}
        \left|\left\{ 
        G^{r_{1}}F^{s_{1}}G^{r_{2}}F^{s_{2}}\cdots G^{r_{n}}F^{s_{n}+1}
        \right\}\left(p,q\right)\right| &= 
        \left|\left\{ \left\{ 
        G^{r_{1}}F^{s_{1}}G^{r_{2}}F^{s_{2}}\cdots G^{r_{n}}F^{s_{n}}
        \right\},F \right\}
        \left(p,q\right)\right| \nonumber \\
        &\leq \left|
        \max\{p^2, q^2, pq\}
        \left(\max\{4a, 4d\}\right)^{r_{1}+\dots+r_{n}+s_{1}+\dots+s_{n}}\cdot 4a
        \right| \nonumber \\
        &\leq \left|
        \max\{p^2, q^2, pq\}
        \left(\max\{4a, 4d\}\right)^{r_{1}+\dots+r_{n}+s_{1}+\dots+s_{n}+1}
        \right|.\nonumber
    \end{align}
    Similarly, for the case of $G$, we have:
    \begin{align}
        \left|\left\{ 
        G^{r_{1}}F^{s_{1}}G^{r_{2}}F^{s_{2}}\cdots G^{r_{n}}F^{s_{n}}G
        \right\}\left(p,q\right)\right| &= 
        \left|\left\{ \left\{ 
        G^{r_{1}}F^{s_{1}}G^{r_{2}}F^{s_{2}}\cdots G^{r_{n}}F^{s_{n}}
        \right\},G \right\}
        \left(p,q\right)\right| \nonumber \\
        &\leq \left|
        \max\{p^2, q^2, pq\}
        \left(\max\{4a, 4d\}\right)^{r_{1}+\dots+r_{n}+s_{1}+\dots+s_{n}}\cdot 4d
        \right| \nonumber \\
        &\leq \left|
        \max\{p^2, q^2, pq\}
        \left(\max\{4a, 4d\}\right)^{r_{1}+\dots+r_{n}+s_{1}+\dots+s_{n}+1}
        \right|.\nonumber
    \end{align}
Hence, using \eqref{quadcase_boosted} in tandem with any of the lemmas or theorems from Appendix \ref{Section_AbsCvg}, the modified Hamiltonian for quadratic case converges absolutely for sufficiently small, nonzero $\eta$.

For a discussion about the conserved quantity predicted by $\widetilde{H}_{\eta}$ in this case, see Section \ref{Quadratic} or Appendix \ref{MHQC_BrutalCalc}.
\end{proof}

\subsubsection{Quartic case} 

We now present a case in which our lemma does not hold: $F=p^4, G=q^4$.
\begin{proposition}
    When $r_1=s_1=\cdots=r_n=s_n=1$, 
    \begin{gather}
        \left|\left\{ 
        G^{r_{1}}F^{s_{1}}G^{r_{2}}F^{s_{2}}\cdots G^{r_{n}}F^{s_{n}}
        \right\}\right| = (n+1)(2n)!4^{2n-1}p^{2n+1}q^{2n+1}.\label{111}
    \end{gather}
    Thus, none of the convergence criteria in Appendix \ref{Section_AbsCvg} can be satisfied in this case.
\end{proposition}
\begin{proof}
    We show this by induction. The formula is true for $n=1$, for $|\{p^4, q^4\}|=16p^3q^3=2 \cdot 2 \cdot 4 \cdot p^3 \cdot p^3$. Suppose that the formula is true for $n$. We calculate the $n+1$-case:
    \begin{align}
        |\{\{(n+1)(2n)!4^{2n-1}p^{2n+1}q^{2n+1}, G\}, F\}| 
        &= |(n+1)(2n)!(2n+1)(2n+4)4^{2n+1}p^{2n+3}q^{2n+3}| \nonumber \\
        &= |(n+2)(2n)!(2n+1)(2n+2)4^{2n+1}p^{2n+3}q^{2n+3}| \nonumber \\
        &= \left((n+1)+1\right)2(n+1)!4^{2(n+1)-1}p^{2(n+1)+1}q^{2(n+1)+1}.\nonumber
    \end{align}
    Hence, by induction, we have shown \eqref{111}. But since even just the terms with IPBs that have $r_1=s_1=\cdots=r_n=s_n=1$ grow with $(2n)!=\left(r_{1}+\dots+r_{n}+s_{1}+\dots+s_{n}\right)!$, none of the convergence criteria in Appendix \ref{Section_AbsCvg} can be satisfied.
\end{proof}
It is not known whether the modified Hamiltonian $\widetilde{H}_{\eta}$ actually converges in this case, although Figure \ref{fig:examples} suggests that it \textit{could} converge for sufficiently small $\eta>0$ (or at least that a conserved quantity exists in that case).

\subsubsection{Logarithmic case} 

To further demonstrate the limitations of the current convergence criteria derived in this paper, we calculate the case where $F=\log(a+p), G=\log(b+q)$. 
\begin{proposition}
    When $F=\log(a+p), G=\log(b+q)$, \\
    \begin{gather}
        \left|\left\{ 
        G^{r_{1}}F^{s_{1}}G^{r_{2}}F^{s_{2}}\cdots G^{r_{n}}F^{s_{n}}
        \right\}\right|=\frac{(S-2)!}{(a+p)^{S-1}(b+q)^{S-1}}, \text{ where $S=\sum_{i=1}^n s_i+\sum_{i=1}^n r_i$.}\nonumber
    \end{gather}
    Yet again, we claim that none of the convergence criteria in Appendix \ref{Section_AbsCvg} are satisfied.
\end{proposition}

\begin{proof}
    As in the quadratic and quartic cases, we show this by induction. Note that $|\{F, G\}|=\frac{1}{(a+p)(b+q)}$, and hence the case where $s_1=r_1=1$ is verified. Assume now that the formula is true for some $n\in\N$, in which case
    \begin{align}
        \left|\left\{ 
        G^{r_{1}}F^{s_{1}}G^{r_{2}}F^{s_{2}}\cdots G^{r_{n}}F^{s_{n}+1}
        \right\}\right| 
        &= \frac{(S-1)!}{(a+p)^{S}(b+q)^{S}}.\nonumber
    \end{align} The case in which we apply $G$ to the formula is very similar. However, the Fubini numbers $a_{k-1}\approx (k-1)!/\left(2\left(\log 2\right)^{k}\right)$ \citep{tanny_1975}, and hence, a factorial increase of the IPBs with $k$ is too fast for the sums in either Lemma \ref{case1_general} or Lemma \ref{case2_general} to converge.
\end{proof}

However, there is still a known, \textit{absolutely} convergent modified Hamiltonian for sufficiently small $\eta>0$. As shown in \cite{Field2003} and restated in Appendix~\ref{master_list}, the MH is of the form $F\left(\left(\alpha+p\right)\left(\beta+q\right)\right)$ for an original Hamiltonian of the form $H\left(p,q\right)=\log\left(\alpha+p\right)+\log\left(\beta+q\right)$. This shows that these convergence criteria we derived, despite requiring sophisticated and careful techniques to derive, are sufficient but still \textit{not} necessary to ensure that the MH converges.

\section{List of known modified Hamiltonians}\label{master_list}

In this section, we review known choices of $F$ and $G$ such that when the symplectic Euler integrator \eqref{init} is applied to $F+G$, the iterations will conserve a modified Hamiltonian exactly. We remark that the results discussed here might not be exhaustive---these are only the cases that the authors are aware of in the literature. Here, we summarize the original Hamiltonian system and the modified Hamiltonian conserved by symplectic Euler \eqref{init} when applied to each respective system.

\paragraph{The quadratic case.} The quadratic case where $(p, q) \in \R^{d \times d}$ and $H(p, q)=p^{\top}Bp+q^{\top}Cq$ and $BC$ is diagonalizable is well-studied in this paper. Refer to Theorem \ref{Thm_Quadratic_MH_multivar} or Appendix \ref{MHQC_BrutalCalc} for details.

\paragraph{The logarithmic case.} In \citep{Field2003}, the authors study the log case where $p, q \in \R$ and 
$$H(p, q)=\log\left(\alpha+p\right)+\log\left(\beta+q\right).$$ 
In this case, when $p, q$ evolve under the symplectic Euler method \eqref{init}, one can check the quantity $L(p,q) \coloneqq  \log\left(\alpha+p\right)+\log\left(\beta+q\right)$ is exactly conserved.

By using an action-angle argument, \cite{Field2003} derive the following closed form of the modified Hamiltonian:
\begin{gather}
    \widetilde{H}_{\eta}(p, q) = \log\left(1-\frac{\eta}{L(p,q)}\right)-\frac{\eta}{L(p,q)}\log\left(1-\frac{\eta}{L(p,q)}\right).\nonumber
\end{gather}

\paragraph{The lattice KdV equation case.} In \citep{Alsallami2018}, the authors study the symplectic Euler integration of the following Hamiltonian system governed by lattice KdV equations: $p, q \in \R$ and $\epsilon$ is a constant in $\R$, 
$$H(p, q)=\epsilon \log\left(\epsilon^2-p^2\right)+\epsilon  \log\left(\epsilon^2-q^2\right).$$ 
In this case, when $p, q$ evolve under the symplectic Euler method~\eqref{init}, the quantity 
$$\mathcal{I}=p^2q^2-\epsilon^2(p^2+q^2)-2\epsilon \eta pq$$
is exactly conserved. Note that $\eta$ in the above equation corresponds to the stepsize in the symplectic Euler method ~\eqref{init}. By using an action-angle argument, \cite{Alsallami2018} derive the following closed form of the modified Hamiltonian:

$$
\tilde H_\eta(\mathcal{P})=\int^{\mathcal{P}} \int_0^{\frac{2 \epsilon^2 \eta \sqrt{-\mathcal{P}^{\prime}}}{\epsilon^4+\mathcal{P}^{\prime}}} \frac{1}{2 \sqrt{\eta^2 \epsilon^2 q^2-\left(\epsilon^2-q^2\right)\left(\epsilon^2 q^2+\mathcal{P}^{\prime}\right)}} d q d \mathcal{P}^{\prime},
$$
where they use $\mathcal{P}$ to denote $p^2q^2-\epsilon^2p^2-\epsilon^2q^2$. Note that the outer integral is an indefinite integral and the inner integral is a definite integral starting from $0$.

\paragraph{The Suris maps.} \cite{Suris1989} summarizes all possible Hamiltonian systems in $\R^{1 \times 1}$ with $F(p)=\frac{p^2}{2}$ such that there exists a conserved quantity to the iterations of symplectic Euler \eqref{init}. \cite{Suris1989} also assumes that the conserved quantity, $\Hat{H}(p, q)$, must be of the form $\Hat{H}_0(p, q)+\eta \Hat{H}_1(p, q)$, where both $\Hat{H}_0(p, q)$ and $\Hat{H}_1(p, q)$ are not functions of $\eta$. 

Note that the existence of a conserved quantity is a necessary condition for the convergence of the Dynkin series \eqref{dynkin} for $\widetilde{H}_{\eta}$, since $\widetilde{H}_{\eta}$ is itself a conserved quantity by Theorem \ref{conserved_Hamiltonian}. To show this, suppose that $\widetilde{H}_{\eta}$ is convergent yet there does not exist a conserved quantity for the iterations of symplectic Euler \eqref{init}. Then, we arrive at a contradiction by Theorem 2.

Under these assumptions, a conserved quantity $\Hat{H}$ exists if and only if $G(q)$ in the decomposition $H(p, q) = F(p)+G(q)$ satisfies the following conditions, and takes the listed form. In the expressions listed below, $A, B, C, D, E$ are arbitrary constants in $\R$.

\begin{align}
F(p)=\frac{p^2}{2}, \quad G(q, \eta) &= -\int_x \frac{A+Bx+Cx^2+Dx^3}{\eta-\eta^2(E+\frac{C}{3}x+\frac{D}{2}x^2)}dx, \nonumber \\
\Hat{H}(p, q) &= \frac{1}{2}(p-q)^2-\eta\left[-\frac{1}{2} A(p+q)-\frac{1}{2} B pq-\frac{1}{6} C pq(p+q)-\frac{1}{4} D p^2 q^2-\frac{1}{2} E(p-q)^2\right];\nonumber
\end{align}

\begin{align}
F(p)=\frac{p^2}{2},\quad G(q, \eta) &= -\int_x \frac{2}{\omega \eta^2} \arctan \left(\frac{(\omega \eta / 2)(A \sin \omega x+B \cos \omega x+C \sin 2 \omega x+D \cos 2 \omega x)}{1-(\omega \eta / 2)(A \cos \omega x-B \sin \omega x+C \cos 2 \omega x-D \sin 2 \omega x+E)}\right)dx, \nonumber \\
\Hat{H} &=\frac{1-\cos \omega(p-q)}{\omega^2}+\frac{\eta}{2 \omega}\left[A(\cos \omega p+\cos \omega q)-B(\sin \omega p+\sin \omega q)\right.\nonumber\\ &\left.
\quad +C \cos \omega(p+q)-D \sin \omega(p+q)+E \cos \omega(p-q)\right];\nonumber
\end{align}

\begin{align}
F(p)=\frac{p^2}{2}, \quad G(q, \eta) &= -\int_x \frac{1}{\alpha \eta^2} \ln \left(\frac{1+\alpha \eta(B \exp (-\alpha x)+D \exp (-2 \alpha x)-E)}{1-\alpha \eta(A \exp (\alpha x)+C \exp (2 \alpha x)+E)}\right) dx, \nonumber\\
\Hat{H} &= \frac{\cosh \alpha(p-q)-1}{\alpha^2}+ \frac{\eta}{2 \alpha}
\Big[-A(\exp (\alpha p)+\exp (\alpha q))+B(\exp (-\alpha p)+\exp (-\alpha q)) \nonumber \\
& -C \exp (\alpha(p+q))+D \exp (-\alpha(p+q))-2 E \cosh \alpha(p-q) \Big].\nonumber
\end{align}

\section{Partial proof of Conjecture \ref{Conjecture_MH_General}} \label{Appendix_Thm5}

We start by taking the Taylor series of the $N$th modification $H_N$:
\begin{gather}
  H_N(p_{n+1}, q_{n+1}) = H_N(p_{n}, q_{n}) + \sum_{i=1}^{\infty} \frac{1}{i!}\Big( (p_{n+1}-p_{n})\cdot\nabla_{p}+(q_{n+1}-q_{n})\cdot\nabla_{q} \Big)^iH_N(p_{n}, q_{n}),  \label{bin}
\end{gather}
where the gradients are to be applied to $H_{N}$. Equivalently, by using the binomial theorem to expand \eqref{bin}, we have
\begin{multline}
    H_{N}(p_{n+1}, q_{n+1}) - H_N(p_{n}, q_{n}) = \\ 
    \sum_{i=1}^{\infty}\frac{1}{i!}\left\{\sum_{j=0}^{i} \binom{i}{j} \nabla_{p}^{j}\nabla_{q}^{i-j}H_N(p_{n}, q_{n}) \left[\left(p_{n+1}-p_{n}\right)^{\otimes j}, \left(q_{n+1}-q_{n}\right)^{\otimes\left(i-j\right)}\right ]\right\},\label{general_kModification_Taylor_1} 
\end{multline}
where it is implied that the dimensions corresponding to $\nabla_{p}^{j}$ multiply with those corresponding to $p_{n+1}-p_{n}$, and those for $\nabla_{q}^{i-j}$ multiply with those for $q_{n+1}-q_{n}$.\footnote{This is suggested by the matching indices in $j$ and $i-j$ and the respectively similar ordering of $\nabla_{p}^{j}\nabla_{q}^{i-j}$ with $\left[\left(p_{n+1}-p_{n}\right)^{\otimes j}, \left(q_{n+1}-q_{n}\right)^{\otimes\left(i-j\right)}\right ]$.} 
We proceed by plugging the iterates \eqref{init} of symplectic Euler into \eqref{bin}:
\begin{gather}
    H_N(p_{n+1}, q_{n+1}) -H_N(p_{n}, q_{n}) = \\
    \sum_{i=1}^{\infty} \frac{\eta^i}{i!}\Big( -\nabla_{q}G\Big|_{q=q_{n}}\cdot\nabla_{p}+\nabla_{p}F\Big|_{p=p_{n+1}}\cdot\nabla_{q} \Big)^iH_N(p_{n}, q_{n}),\label{something_else}
\end{gather}
or alternatively,
\begin{multline}
    H_{N}(p_{n+1}, q_{n+1}) - H_N(p_{n}, q_{n})=\\ 
    \sum_{i=1}^{\infty}\frac{\eta^{i}}{i!}\left\{\sum_{j=0}^{i}\left(-1\right)^{j} \binom{i}{j}\left[ 
   \nabla_{p}^{j}\nabla_{q}^{i-j}H_N(p_{n}, q_{n}),\nabla_{q}G(q_{n})^{\otimes j},\nabla_{p} F(p_{n+1})^{\otimes\left(i-j\right)}\right ]\right\}.\label{general_kModification_Taylor_2} 
\end{multline}

The fact that $\nabla_{p}F$ is evaluated at $p_{n+1}$ requires us to also approximate $F(p_{n+1})$ with Taylor series:
\begin{align}
F(p_{n+1}) 
&= F\left(p_{n}-\eta \nabla_{q}G(q_{n})\right)  \notag \\
&=\left.F\left(p_{n}\right)
-\eta\left.\left[\nabla_{p}F,\nabla_{q}G\right]\right|_{p_{n},q_{n}}
+\frac{\eta^2}{2!} \left.\left[\nabla_{p}^{2}F,\nabla_{q}G^{\otimes 2}\right]\right|_{p_{n},q_{n}}
-\frac{\eta^3}{3!}\left[\nabla_{p}^{3}F,\nabla_{q}G^{\otimes 3}\right]\right|_{p_{n},q_{n}}
+\cdots \notag \\
&=F\left(p_{n}\right)
+\eta\left.\{F,G\}\right|_{p_{n},q_{n}}
+\frac{\eta^2}{2!}\left.\{\{F,G\},G\}\right|_{p_{n},q_{n}}
+\frac{\eta^3}{3!}\left.\{\{\{F,G\}, G\}, G\}\right|_{p_{n},q_{n}}
+\cdots \notag \\
&=\left.e^{\eta\{\cdot,  G\}}F\right|_{p_{n},q_{n}}. \label{F'_Taylor}
\end{align}
We proceed by plugging \eqref{F'_Taylor} into \eqref{something_else}:
\begin{gather}
H_N(p_{n+1}, q_{n+1})- H_N(p_{n}, q_{n})=\left.\sum_{i=1}^{\infty} \frac{\eta^i}{i!}\left[-\nabla_{q}G\left(q_{n}\right)\cdot\nabla_{p}+ \left.\nabla_{p}\left(e^{\eta\{\cdot,  G\}}F\right)\right|_{p_{n},q_{n}}\cdot\nabla_{q} \right]^iH_N\right|_{p_{n},q_{n}}.\label{product}
\end{gather}
For \eqref{product}, we note that the evaluations at $p_{n},q_{n}$ both inside and outside the brackets are intentional. The reader should interpret this notation in the sense that, when the differential operator inside the brackets is repeatedly applied to $H_{N}$, the derivatives with respect to $p$ and $q$ do \textit{not} apply to $\nabla_{q}G\left(q_{n}\right)$ or $\left.\nabla_{p}\left(e^{\eta\{\cdot,  G\}}F\right)\right|_{p_{n},q_{n}}$, as one would normally expect from the product rule. 

We proceed by rewriting the terms contained in $\widetilde{H}_N(p_{n+1}, q_{n+1})-\widetilde{H}_N(p_n, q_n)$ as the sum of $N+1$ Taylor series. Using the exponential operator series expansion, we can further rewrite \eqref{product} in the following form:
\begin{gather}
    H_N(p_{n+1}, q_{n+1})-H_N(p_{n}, q_{n})=\nonumber\\ \left.\left\{\exp\left[\eta\left(-\nabla_{q} G\left(q_{n}\right)\cdot\nabla_{p}+\left.\nabla_{p} \left(\exp\left(\eta\left\{\cdot,G\right\}\right)F\right)\right|_{p_{n},q_{n}}\cdot\nabla_{q}\right)\right]-1\right\}H_{N}\right|_{p_{n}, q_{n}}.\label{general_kModification_Taylor_4}
\end{gather}
Then, by taking a Taylor series of \eqref{general_kModification_Taylor_4} about $\eta=0$, the $\mathcal{O}\left(\eta^{j}\right)$ term in \eqref{general_kModification_Taylor_4} is 
\begin{gather}
    C_{j,k}\coloneqq \Omega_j(H_k)\quad \forall j\in\N, k\in\N_{0},\label{Cjk}
\end{gather}
where $\Omega_{j}:C^{r+j}\left(\Z,\R\right)\rightarrow C^{r}\left(\Z,\R\right)$ ($r\in\N_{0}$) is the differential operator
\begin{gather}
    \Omega_j\left(f\right) \coloneqq  \frac{1}{j!} \left.\frac{\partial^{(j)}}{\partial\eta^{(j)}}
    \left\{\exp\left[\eta\left(-\nabla_{q} G\left(q_{n}\right)\cdot\nabla_{p}+\left.\nabla_{p} \left(\exp\left(\eta\left\{\cdot,G\right\}\right)F\right)\right|_{p_{n},q_{n}}\cdot\nabla_{q}\right)\right]-1\right\}\right|_{\eta=0} f\label{Omega}
\end{gather}
for all $f\in C^{r+j}\left(\Z,\R\right)$ and $j\in\N$. 

To get the conservation error of the $k$th-order truncated modified Hamiltonian, $\widetilde{H}_{\eta}^{(k)}$, we combine the first $k$ modifications and truncate the Taylor series at $\mathcal{O}\left(\eta^{k+1}\right)$ using the Lagrange form of the remainder from Taylor's theorem. Hence, we have the following lemma.

\begin{lemma}{\label{lemma_error_N+2}}
    Suppose that the domains $\mathcal{P}$ and $\mathcal{Q}$ are convex and closed, and assume that $F$ and $G$ are sufficiently smooth. If the conditions
    \begin{gather}
    \sum_{j=0}^{m-1}\left.C_{m-j,j}\right|_{p,q}= 0, \quad \text{ for all } ~ 1 \leq m \leq N+1,
    \label{sum_to_0_condition}
\end{gather} 
and $|C_{m-j+1, j}| \leq L$ hold for all $(p,q)\in \Z$, then we have that $ \widetilde{H}_{\eta}^{(N)}(p_{n+1}, q_{n+1})-\widetilde{H}_{\eta}^{(N)}(p_{n}, q_{n})\leq\eta^{N+2}(N+1)L$.
\end{lemma}

\begin{proof}
    Since $\P$ and $\Q$ are convex and closed, $\left(p_{n},q_{n}\right)$ remains within $\Z$ for all $n$, and we can apply the single-variable Taylor's theorem with Lagrange remainder \citep[Theorem~31.3]{Ross2013} to expand \eqref{general_kModification_Taylor_4} in the stepsize $\eta$ about $\eta=0$, for any fixed truncation order $j+k=N$. Doing so gives 
    \begin{align}
    \left(H_0(p_{n+1}, q_{n+1})-H_0(p_n, q_n)\right) &= \left[\eta C_{1,0}+\eta^2 C_{2,0}+\dots+\eta^{N+1}C_{N+1,0}\right]_{z=z_{n}} +\eta^{N+2}\left.C_{N+2,0}\right|_{z=\zeta_{0}}\notag \\
    \eta^{1}\left(H_1(p_{n+1}, q_{n+1})-H_1(p_n, q_n)\right) &= \left[\eta^2 C_{1,1}+\eta^3 C_{2,1}+\dots+\eta^{N+1}C_{N,1}\right]_{z=z_{n}}+\eta^{N+2}\left.C_{N+1,1}\right|_{z=\zeta_{1}}\notag \\
    \eta^{2}\left(H_2(p_{n+1}, q_{n+1})-H_2(p_n, q_n)\right) &= \left[\eta^3 C_{1,2}+\eta^4 C_{2,2}+\dots+\eta^{N+1}C_{N-1,2}\right]_{z=z_{n}}+\eta^{N+2}\left.C_{N,2}\right|_{z=\zeta_{2}}\notag \\
    \vdots \qquad &= \qquad \vdots \notag\\
    \eta^{N}\left(H_N(p_{n+1}, q_{n+1})-H_N(p_n, q_n)\right) &= \eta^{N+1} \left.C_{1,N}\right|_{z=z_{n}}+\eta^{N+2}\left.C_{2,N}\right|_{z=\zeta_{N}}
    \label{upper_triangular_MH_trunc}
    \end{align}
    for the order-$N$ truncated modified Hamiltonian, where $\zeta_{0},\zeta_{1},\dots,\zeta_{N}$ lie on the convex line segment $L_{n}\coloneqq \left\{z\in\R^{d}\times\R^{d}:z=\left(1-\lambda\right)z_{n+1}+\lambda z_{n}, 0\leq\lambda\leq 1\right\}$ between $z_{n}$ and $z_{n+1}$.
    Thus, \begin{align}
        \left| \widetilde{H}_{\eta}^{(N)}(p_{n+1}, q_{n+1})-\widetilde{H}_{\eta}^{(N)}(p_{n}, q_{n}) \right| &= \left| \sum_{m=1}^{N+1} \eta^m\sum_{j=0}^{m-1}\left.C_{m-j, j}\right|_{p_{n},q_{n}} +\eta^{N+2}\sum_{m=0}^{N}C_{N+2-m, m}|_{z=\zeta_{m}} \right| \nonumber \\
        &= \left| \sum_{m=1}^{N+1} 0 +\eta^{N+2}\sum_{m=0}^{N}C_{N+2-m, m}|_{z=\zeta_{m}} \right| \nonumber \\
        &\leq \eta^{N+2}\sum_{m=0}^{N}\left|C_{N+2-m, m}|_{z=\zeta_{m}} \right| \nonumber \\
        &\leq \eta^{N+2}(N+1)L. \label{coefs}
    \end{align}
\end{proof}

It now remains to show that the sum along each diagonal in \eqref{upper_triangular_MH_trunc} actually cancels to zero except for that which corresponds to the last coefficient at order $\eta^{N+2}$. To partially answer this question, in the next two subsections, we show that the following lemma holds:

\begin{lemma}\label{Coefficients_Cancellation} 
    For $N=0, 1, 2, 3$, if $F:\P\times\R$ and $G:\Q\times\R$ have derivatives up to order $N+2$ on their respective domains, then
    we have that $\sum_{j=0}^i C_{i+1-j,j} = 0$ for all $i \leq N$ on all of $\Z$. Moreover, for $N=0, 1, \dots,10$, when $d=1$ and $F,G$ have derivatives up to order $N+2$ on their respective domains, then we have 
    $\sum_{j=0}^i C_{i+1-j,j} = 0$ for all $i \leq N$ on $\Z$. 
\end{lemma}

Once Lemma \ref{Coefficients_Cancellation} has been proven, we could use that and Lemma \ref{lemma_error_N+2} to show that $ \widetilde{H}_{\eta}^{(N)}(p_{n+1}, q_{n+1})-\widetilde{H}_{\eta}^{(N)}(p_{n}, q_{n})=\mathcal{O}(\eta^{N+2})$. However, at least up to $j+k=5$ (i.e., up to $N=3$), we will show that each $C_{j,k}$ can be expressed as a linear combination of products of higher-order derivatives of $F$ and $G$ such that, for each product in this linear combination, the sum of the orders of derivatives in the product adds up to $2\left(j+k\right)=2N+4$, no individual derivative is of higher order than $j+k=N+2$, and there are $N+3$ of such terms. Thus, we can show Conjecture \ref{Conjecture_MH_General} for (in theory) arbitrarily high $N\in\N_{0}$ by assuming that $F$ and $G$ are both $L$-smooth of orders $1, \dots, N+2$, since in that case, \eqref{coefs} can be simplified to 
\begin{gather}
    \left|\widetilde{H}_{\eta}^{\left(N\right)}\left(p_{n+1},q_{n+1}\right)-\widetilde{H}_{\eta}^{\left(N\right)}\left(p_{n},q_{n}\right)\right|\leq \Phi\left(N\right)L^{N+3}\eta^{N+2}\label{awesome}
\end{gather}
for some function $\Phi:\N_{0}\rightarrow\mathbb{Q}$ to be determined. 

Using the aforementioned steps, we show how to verify \eqref{awesome} for orders $N \in \{0,1,2,3\}$ and for arbitrary $d$ in Appendix \ref{SymbolicTaylor}  by direct computation.
We also show how this can be verified by computer for orders $N \in \{0,1,\dots,10\}$ and $d=1$ in Appendix \ref{CompTaylor}. 

\subsection{Verification up to \texorpdfstring{$N=3$}{N=3} via direct computation}\label{SymbolicTaylor}

We start by computing $\Omega_{j}$ for $j=1,\dots,5$, since all will be used when checking Conjecture \ref{Conjecture_MH_General} up to $N=3$.

\paragraph{Calculation of \texorpdfstring{$\Omega_j$}{Omegaj} up to \texorpdfstring{$j=5$}{j=5}}
\begin{align}
\Omega_1(\cdot) &= 
-\nabla_p(\cdot)\nabla_qG 
+\nabla_q(\cdot)\nabla_pF\\
\Omega_2(\cdot) &= 
\nabla_q(\cdot) \nabla_p\{F, G\}
+\frac{1}{2}\nabla_{pp}(\cdot)\left[\nabla_qG^{\otimes 2}\right]
-\nabla_{pq}(\cdot)[\nabla_qG, \nabla_pF]
+\frac{1}{2}\nabla_{qq}(\cdot)\left[\nabla_pF^{\otimes 2}\right]\\
\Omega_3(\cdot)&=
\frac{1}{2}\nabla_q(\cdot) \nabla_{ppp}F[\nabla_qG^{\otimes 2}]
+\nabla_{pq}(\cdot)[\nabla_qG, \nabla_{pp}F\nabla_qG]-\nabla_{qq}(\cdot)[\nabla_pF, \nabla_{pp}F\nabla_qG]
\nonumber \\
&\quad -\frac{1}{6}\nabla_{ppp}(\cdot)[\nabla_q(G)^{\otimes 3}]
+\frac{1}{2}\nabla_{ppq}(\cdot)[\nabla_qG^{\otimes 2}, \nabla_pF]
\nonumber \\
&\quad -\frac{1}{2}\nabla_{pqq}(\cdot)[\nabla_qG, \nabla_pF^{\otimes 2}]
+\frac{1}{6}\nabla_{qqq}(\cdot)[\nabla_p(F)^{\otimes 3}]\\
\Omega_4(\cdot)&=\frac{1}{24}\nabla_{pppp}(\cdot)[\nabla_qG^{\otimes 4}]-\frac{1}{6}\nabla_{pppq}(\cdot)[\nabla_qG^{\otimes 3}, \nabla_pF]-\frac{1}{2}\nabla_{ppq}(\cdot)[\nabla_qG^{\otimes 2}, \nabla_{pp}F\nabla_qG]\nonumber\\ &\quad +\frac{1}{4}\nabla_{ppqq}(\cdot)[\nabla_qG^{\otimes 2}, \nabla_pF^{\otimes 2}]-\frac{1}{2}\nabla_{pq}(\cdot)[\nabla_qG, \nabla_{ppp}F[\nabla_qG^{\otimes 2}]]\nonumber\\ &\quad +\nabla_{pqq}(\cdot)[\nabla_qG, \nabla_pF, \nabla_{pp}F\nabla_qG]-\frac{1}{6}\nabla_{pqqq}(\cdot)[\nabla_qG, \nabla_pF^{\otimes 3}]\nonumber\\ &\quad -
\frac{1}{6}\nabla_q(\cdot)\nabla_{pppp}F[\nabla_qG^{\otimes 3}]+\frac{1}{2}\nabla_{qq}(\cdot)[\nabla_{p}F, \nabla_{ppp}{F}[\nabla_{q}G^{\otimes 2}]]\nonumber \\
&\quad -\frac{1}{2}\nabla_{qqq}(\cdot)[\nabla_pF^{\otimes 2}, \nabla_{pp}F\nabla_qG]+\frac{1}{2}\nabla_{qq}(\cdot)[(\nabla_{pp}F\nabla_qG)^{\otimes 2}]+\frac{1}{24}\nabla_{qqqq}(\cdot)[\nabla_pF^{\otimes 4}]
\end{align}

\begin{align}
\Omega_5(\cdot)&=-\frac{1}{120}\nabla_{ppppp}\left(\cdot\right)\left[\nabla_{q}G^{\otimes 5}\right]+\frac{1}{24}\nabla_{ppppq}\left(\cdot\right)\left[\nabla_{q}G^{\otimes 4},\nabla_{p}F\right]+\frac{1}{6}\nabla_{pppq}\left(\cdot\right)\left[\nabla_{q}G^{\otimes 3},\nabla_{pp}F\nabla_{q}G\right]\nonumber\\
&\quad-\frac{1}{12}\nabla_{pppqq}\left(\cdot\right)\left[\nabla_{q}G^{\otimes 3},\nabla_{p}F^{\otimes 2}\right]+\frac{1}{4}\nabla_{ppq}\left(\cdot\right)\left[\nabla_{q}G^{\otimes 2},\nabla_{ppp}F\left[\nabla_{q}G^{\otimes 2}\right]\right]\nonumber\\
&\quad-\frac{1}{2}\nabla_{ppqq}\left(\cdot\right)\left[\nabla_{q}G^{\otimes 2},\nabla_{p}F,\nabla_{pp}F\nabla_{q}G\right]+\frac{1}{12}\nabla_{ppqqq}\left(\cdot\right)\left[\nabla_{q}G^{\otimes 2},\nabla_{p}F^{\otimes 3}\right]\nonumber\\
&\quad+\frac{1}{6}\nabla_{pq}\left(\cdot\right)\left[\nabla_{q}G,\nabla_{pppp}F\left[\nabla_{q}G^{\otimes 3}\right]\right]-\frac{1}{2}\nabla_{pqq}\left(\cdot\right)\left[\nabla_{q}G,\nabla_{p}F,\nabla_{ppp}F\left[\nabla_{q}G^{\otimes 2}\right]\right]\nonumber\\ &\quad+\frac{1}{2}\nabla_{pqqq}\left(\cdot\right)\left[\nabla_{q}G,\nabla_{p}F^{\otimes 2},\nabla_{pp}F\nabla_{q}G\right]-\frac{1}{2}\nabla_{pqq}\left(\cdot\right)\left[\nabla_{q}G,\left(\nabla_{pp}F\nabla_{q}G\right)^{\otimes 2}\right]\nonumber\\ &\quad-\frac{1}{24}\nabla_{pqqqq}\left(\cdot\right)\left[\nabla_{q}G,\nabla_{p}F^{\otimes 4}\right]+\frac{1}{24}\nabla_{q}\left(\cdot\right)\nabla_{ppppp}F\left[\nabla_{q}G^{\otimes 4}\right]\nonumber\\
&\quad-\frac{1}{6}\nabla_{qq}\left(\cdot\right)\left[\nabla_{p}F,\nabla_{pppp}F\left[\nabla_{q}G^{\otimes 3}\right]\right]+\frac{1}{4}\nabla_{qqq}\left(\cdot\right)\left[\nabla_{p}F^{\otimes 2},\nabla_{ppp}F\left[\nabla_{q}G^{\otimes 2}\right]\right]\nonumber\\
&\quad-\frac{1}{6}\nabla_{qqqq}\left(\cdot\right)\left[\nabla_{p}F^{\otimes 3},\nabla_{pp}F\nabla_{q}G\right]+\frac{1}{2}\nabla_{qqq}\left(\cdot\right)\left[\nabla_{p}F,\left(\nabla_{pp}F\nabla_{q}G\right)^{\otimes 2}\right]\nonumber\\
&\quad-\frac{1}{2}\nabla_{qq}\left(\cdot\right)\left[\nabla_{pp}F\nabla_{q}G,\nabla_{ppp}F\left[\nabla_{q}G^{\otimes 2}\right]\right]+\frac{1}{120}\nabla_{qqqqq}\left(\cdot\right)\left[\nabla_{p}F^{\otimes 5}\right].
\end{align}

\paragraph{For $N=0$:}
\noindent 
Note $C_{1,0}$ equals to $\Omega_1(H_{0})$. Hence,
\begin{align}
    C_{1,0} &= \Omega_{1}(H_{0}) \nonumber \\
    &= -\nabla_qG\nabla_p(H_0)+\nabla_pF\nabla_q(H_0) \nonumber \\
    &= -\nabla_qG\nabla_pF+\nabla_pF\nabla_qG\nonumber \\
    &=0. \label{C01}
\end{align} Thus, the case where $N=0$ is verified. And in this case, we do not need to get an upper bound on the remainder, as that will be taken care of when dealing with when $N=1$.

\paragraph{For $N=1$:}
\noindent 
We first calculate $\Omega_2(H_0)$ and $\Omega_1(H_1)$:
\begin{align}
    \Omega_2(H_0) 
&= \nabla_q(H_0) \nabla_p\{F, G\}+\frac{1}{2}\nabla_{pp}(H_0)\left[\nabla_qG^{\otimes 2}\right]
-\nabla_{pq}(H_0)[\nabla_qG, \nabla_pF]
+\frac{1}{2}\nabla_{qq}(H_0)\left[\nabla_pF^{\otimes 2}\right]\nonumber \\
&=-\frac{1}{2}\nabla_{pp}F\left[\nabla_qG^{\otimes 2}\right]+\frac{1}{2}\nabla_{qq}G\left[\nabla_pF^{\otimes 2}\right]\label{C02},\\
    \Omega_1(H_1) &= -\nabla_p(H_1)\nabla_qG +\nabla_q(H_1)\nabla_pF 
    =\frac{1}{2}\nabla_{pp}F[\nabla_qG^{\otimes 2}]
    -\frac{1}{2}\nabla_{qq}G[\nabla_pF^{\otimes 2}]\label{C11}.
\end{align}

 We now sum \eqref{C02} and \eqref{C11} up:
\begin{align}
    \Omega_2(H_0)+\Omega_1(H_1) &= -\frac{1}{2}\nabla_{pp}F\left[\nabla_qG^{\otimes 2}\right]+\frac{1}{2}\nabla_{qq}G\left[\nabla_pF^{\otimes 2}\right] \nonumber \\
    &\quad +\frac{1}{2}\nabla_{pp}F[\nabla_qG^{\otimes 2}]
    -\frac{1}{2}\nabla_{qq}G[\nabla_pF^{\otimes 2}]\nonumber \\
    &=0.
\end{align}
Thus, the case where $N=1$ is verified. Furthermore, to get an upper bound on the $N=0$ case, assuming that $F$ and $G$ are $L$-smooth of orders $1$ and $2$, \eqref{coefs} gives us 
\begin{align*}    
    \left|\widetilde{H}_{\eta}^{\left(0\right)}\left(p_{n+1},q_{n+1}\right)-\widetilde{H}_{\eta}^{\left(0\right)}\left(p_{n},q_{n}\right)\right|
    & \leq \eta^{2}\max_{z\in L_{n}}\left|\Omega_{2}\left(H_{0}\right)\left(z\right)\right| \\ 
    &\le \frac{\eta^{2}}{2}\left(\max\left|\nabla_{pp}F\left[\nabla_qG^{\otimes 2}\right]\right|+\max\left|\nabla_{qq}G\left[\nabla_pF^{\otimes 2}\right]\right|\right) \\
    &\le L^{3}\eta^{2}.
\end{align*}
The second inequality above follows from triangle inequality, and the last from Cauchy-Schwarz inequality.

\paragraph{For $N=2$:}
\noindent
We first calculate $\Omega_3(H_0)$, $\Omega_2(H_1)$ and $\Omega_1(H_2)$:
\begin{align}
    \Omega_3(H_0) &= \frac{1}{2}\nabla_q(H_0) \nabla_{ppp}F[\nabla_qG^{\otimes 2}]+\nabla_{pq}(H_0)[\nabla_qG, \nabla_{pp}F\nabla_qG]\nonumber \\
&
\quad -\nabla_{qq}(H_0)[\nabla_pF, \nabla_{pp}F\nabla_qG]
-\frac{1}{6}\nabla_{ppp}(H_0)[\nabla_q(G)^{\otimes 3}]
\nonumber \\
&\quad +\frac{1}{2}\nabla_{ppq}(H_0)[\nabla_qG^{\otimes 2}, \nabla_pF]
-\frac{1}{2}\nabla_{pqq}(H_0)[\nabla_pG, \nabla_pF^{\otimes 2}]
+\frac{1}{6}\nabla_{qqq}(H_0)[\nabla_p(F)^{\otimes 3}] \nonumber \\
&= \frac{1}{2}\nabla_qG \nabla_{ppp}F[\nabla_qG^{\otimes 2}]-0
-\nabla_{qq}G[\nabla_pF, \nabla_{pp}F\nabla_qG]-\frac{1}{6}\nabla_{ppp}F[\nabla_q(G)^{\otimes 3}]\nonumber\\ 
&\quad +0+0+\frac{1}{6}\nabla_{qqq}G[\nabla_p(F)^{\otimes 3}] \nonumber \\
&=\frac{1}{3}\nabla_{ppp}F[\nabla_qG^{\otimes 3}]-\nabla_pF\nabla_{qq}G\nabla_{pp}F\nabla_qG+\frac{1}{6}\nabla_{qqq}G[\nabla_p(F)^{\otimes 3}],\label{C03}\\
    \Omega_2(H_1) &= 
\nabla_q(H_1) \nabla_p\{F, G\}
+\frac{1}{2}\nabla_{pp}(H_1)\left[\nabla_qG^{\otimes 2}\right]
\nonumber \\&\quad -\nabla_{pq}(H_1)[\nabla_qG, \nabla_pF]
+\frac{1}{2}\nabla_{qq}(H_1)\left[\nabla_pF^{\otimes 2}\right] \nonumber \\
&=\frac{1}{2}\nabla_pF\nabla_{qq}G\nabla_{pp}F\nabla_qG -\frac{1}{4}\nabla_{ppp}F[\nabla_qG^{\otimes 3}]
\nonumber \\
&\quad +\frac{1}{2}\nabla_qG\nabla_{pp}F\nabla_{qq}G\nabla_pF
-\frac{1}{4}\nabla_{qqq}G[\nabla_pF^{\otimes 3}] \nonumber \\
&=-\frac{1}{4}\nabla_{ppp}F[\nabla_qG^{\otimes 3}]
+\nabla_qG\nabla_{pp}F\nabla_{qq}G\nabla_pF
-\frac{1}{4}\nabla_{qqq}G[\nabla_pF^{\otimes 3}],\label{C12}
\end{align}

\begin{align}
    \Omega_1(H_2) &= -\nabla_p(H_2)\nabla_qG+\nabla_q(H_2)\nabla_pF \nonumber \\
    &=-\frac{1}{12}\nabla_{ppp}F[\nabla_{q}G^{\otimes 3}]-\frac{1}{6}\nabla_pF\nabla_{qq}G\nabla_{pp}F\nabla_qG+\frac{1}{6}\nabla_pF\nabla_{qq}G\nabla_{pp}F\nabla_qG+\frac{1}{12}\nabla_{qqq}G[\nabla_pF^{\otimes 3}]\nonumber \\
    &=-\frac{1}{12}\nabla_{ppp}F[\nabla_{q}G^{\otimes 3}]+\frac{1}{12}\nabla_{qqq}G[\nabla_pF^{\otimes 3}].\label{C21}
\end{align}

We now sum \eqref{C03}, \eqref{C12} and \eqref{C21}:
\begin{align}
    \Omega_3(H_0)+\Omega_2(H_1)+\Omega_1(H_2) &= \frac{1}{3}\nabla_{ppp}F[\nabla_qG^{\otimes 3}]-\nabla_pF\nabla_{qq}G\nabla_{pp}F\nabla_qG+\frac{1}{6}\nabla_{qqq}G[\nabla_p(F)^{\otimes 3}]\nonumber \\
    &\quad-\frac{1}{4}\nabla_{ppp}F[\nabla_qG^{\otimes 3}]
+\nabla_qG\nabla_{pp}F\nabla_{qq}G\nabla_pF
-\frac{1}{4}\nabla_{qqq}G[\nabla_pF^{\otimes 3}] \nonumber \\
&\quad-\frac{1}{12}\nabla_{ppp}F[\nabla_{q}G^{\otimes 3}]+\frac{1}{12}\nabla_{qqq}G[\nabla_pF^{\otimes 3}] \nonumber \\
&=0
\end{align}

Thus, the case where $N=2$ is verified. Moreover, to get an upper bound on the $N=1$ case, assuming that $F$ and $G$ are $L$-smooth of orders $1$, $2$, and $3$, \eqref{coefs} gives us 
\begin{align*}
    \left|\widetilde{H}_{\eta}^{\left(1\right)}\left(p_{n+1},q_{n+1}\right)-\widetilde{H}_{\eta}^{\left(1\right)}\left(p_{n},q_{n}\right)\right|
    &\leq\eta^{3}\left(\max_{z\in L_{n}}\left|\Omega_{3}\left(H_{0}\right)\left(z\right)\right|+\max_{z\in L_{n}}\left|\Omega_{2}\left(H_{1}\right)\left(z\right)\right|\right) \\
    &\leq\eta^{3}L^{4}\left(\frac{1}{3}+1+\frac{1}{6}+\frac{1}{4}+1+\frac{1}{4}\right) \\
    &\leq3L^{4}\eta^{3}.
\end{align*}

\paragraph{For $N=3$:}
\noindent 
We first calculate $\Omega_4(H_0)$, $\Omega_3(H_1)$, $\Omega_2(H_2)$ and $\Omega_1(H_3)$:
\begin{align}
    \Omega_4(H_0)&=\frac{1}{24}\nabla_{pppp}(H_{0})[\nabla_qG^{\otimes 4}]-\frac{1}{6}\nabla_{pppq}(H_{0})[\nabla_qG^{\otimes 3}, \nabla_pF]-\frac{1}{2}\nabla_{ppq}(H_{0})[\nabla_qG^{\otimes 2}, \nabla_{pp}F\nabla_qG]\nonumber\\ 
    &\quad+\frac{1}{4}\nabla_{ppqq}(H_{0})[\nabla_qG^{\otimes 2}, \nabla_pF^{\otimes 2}]-\frac{1}{2}\nabla_{pq}(H_{0})[\nabla_qG, \nabla_{ppp}F[\nabla_qG^{\otimes 2}]]\nonumber\\ 
    &\quad+\nabla_{pqq}(H_{0})[\nabla_qG, \nabla_pF, \nabla_{pp}F\nabla_qG]-\frac{1}{6}\nabla_{pqqq}(H_{0})[\nabla_qG, \nabla_pF^{\otimes 3}]\nonumber\\ 
    &\quad-\frac{1}{6}\nabla_q(H_{0})\nabla_{pppp}F[\nabla_qG^{\otimes 3}]+\frac{1}{2}\nabla_{qq}(H_{0})[\nabla_{p}F, \nabla_{ppp}{F}[\nabla_{q}G^{\otimes 2}]]\nonumber \\
&\quad-\frac{1}{2}\nabla_{qqq}(H_{0})[\nabla_pF^{\otimes 2}, \nabla_{pp}F\nabla_qG]+\frac{1}{2}\nabla_{qq}(H_{0})[(\nabla_{pp}F\nabla_qG)^{\otimes 2}]+\frac{1}{24}\nabla_{qqqq}(H_{0})[\nabla_pF^{\otimes 4}]\nonumber\\
&=\frac{1}{24}\nabla_{pppp}F[\nabla_qG^{\otimes 4}]+0-0+0-0+0-0-
\frac{1}{6}\nabla_{q}G\nabla_{pppp}F[\nabla_qG^{\otimes 3}]\nonumber \\ 
&\quad+\frac{1}{2}\nabla_{qq}G[\nabla_{p}F, \nabla_{ppp}{F}[\nabla_{q}G^{\otimes 2}]]-\frac{1}{2}\nabla_{qqq}G[\nabla_pF^{\otimes 2}, \nabla_{pp}F\nabla_qG]\nonumber\\ &\quad+\frac{1}{2}\nabla_{qq}G[(\nabla_{pp}F\nabla_qG)^{\otimes 2}]+\frac{1}{24}\nabla_{qqqq}G[\nabla_pF^{\otimes 4}]\nonumber\\
&=-
\frac{1}{8}\nabla_{pppp}F[\nabla_qG^{\otimes 4}]-\frac{1}{2}\nabla_{qqq}G[\nabla_pF^{\otimes 2}, \nabla_{pp}F\nabla_qG]+\frac{1}{24}\nabla_{qqqq}G[\nabla_pF^{\otimes 4}]\nonumber\\ 
&\quad+\frac{1}{2}\nabla_{qq}G[(\nabla_{pp}F\nabla_qG)^{\otimes 2}]+\frac{1}{2}\nabla_{qq}G[\nabla_{p}F, \nabla_{ppp}{F}[\nabla_{q}G^{\otimes 2}]],\label{C04}
\end{align}

\begin{align}
    \Omega_3(H_1) &=
\frac{1}{2}\nabla_q(H_1) \nabla_{ppp}F[\nabla_qG^{\otimes 2}]+\nabla_{pq}(H_1)[\nabla_qG, \nabla_{pp}F\nabla_qG]
-\nabla_{qq}(H_1)[\nabla_pF, \nabla_{pp}F\nabla_qG]
\nonumber \\
&\quad-\frac{1}{6}\nabla_{ppp}(H_1)[\nabla_qG^{\otimes 3}]
+\frac{1}{2}\nabla_{ppq}(H_1)[\nabla_qG^{\otimes 2}, \nabla_pF]
\nonumber \\
&\quad-\frac{1}{2}\nabla_{pqq}(H_1)[\nabla_qG, \nabla_pF^{\otimes 2}]
+\frac{1}{6}\nabla_{qqq}(H_1)[\nabla_pF^{\otimes 3}]\nonumber \\
&=-\frac{1}{4}\nabla_pF\nabla_{qq}G\nabla_{ppp}F[\nabla_qG^{\otimes 2}]-\frac{1}{2}\nabla_qG\nabla_{pp}F\nabla_{qq}G\nabla_{pp}F\nabla_qG
+\frac{1}{2}[\nabla_pF^{\otimes 2}]\nabla_{qqq}G\nabla_{pp}F\nabla_qG \nonumber \\
&\quad+\frac{1}{12}\nabla_{pppp}F[\nabla_qG^{\otimes 4}]
-\frac{1}{4}[\nabla_qG^{\otimes 2}]\nabla_{ppp}F\nabla_{qq}G\nabla_pF
\nonumber \\
&\quad+\frac{1}{4}\nabla_qG\nabla_{pp}F\nabla_{qqq}G[\nabla_pF^{\otimes 2}]
-\frac{1}{12}\nabla_{qqqq}G[\nabla_pF^{\otimes 4}] \nonumber \\
&=\frac{1}{12}\nabla_{pppp}F[\nabla_qG^{\otimes 4}]-\frac{1}{2}[\nabla_qG^{\otimes 2}]\nabla_{ppp}F\nabla_{qq}G\nabla_pF-\frac{1}{2}\nabla_qG\nabla_{pp}F\nabla_{qq}G\nabla_{pp}F\nabla_qG\nonumber \\
&\quad +\frac{3}{4}[\nabla_pF^{\otimes 2}]\nabla_{qqq}G\nabla_{pp}F\nabla_qG-\frac{1}{12}\nabla_{qqqq}G[\nabla_pF^{\otimes 4}], \label{C13}
\end{align}

\begin{align}
\Omega_2(H_2) &= -\nabla_q(H_2) \nabla_{pp}F\nabla_qG\nonumber \\
&\quad+\frac{1}{2}\nabla_{pp}(H_2)\left[\nabla_qG^{\otimes 2}\right]
-\nabla_{pq}(H_2)[\nabla_qG, \nabla_pF]
+\frac{1}{2}\nabla_{qq}(H_2)\left[\nabla_pF^{\otimes 2}\right]\nonumber \\
&=-\frac{1}{6}\nabla_qG\nabla_{pp}F\nabla_{qq}G\nabla_{pp}F\nabla_qG-\frac{1}{12}[\nabla_pF^{\otimes 2}]\nabla_{qqq}G\nabla_{pp}F\nabla_qG \nonumber \\ 
&\quad+\frac{1}{24}\nabla_{pppp}F[\nabla_qG^{\otimes 4}]+\frac{1}{12}\nabla_{p}F\nabla_{qq}G\nabla_{ppp}F[\nabla_qG^{\otimes 2}]+\frac{1}{12}\nabla_{q}G\nabla_{pp}F\nabla_{qq}G\nabla_{pp}F\nabla_{q}G\nonumber \\
&\quad-\frac{1}{6}[\nabla_{q}G^{\otimes 2}]\nabla_{ppp}F\nabla_{qq}G\nabla_{p}F -\frac{1}{6}[\nabla_{p}F^{\otimes 2}]\nabla_{qqq}G\nabla_{pp}F\nabla_{q}G \nonumber \\
&\quad+\frac{1}{12}\nabla_qG\nabla_{pp}F\nabla_{qqq}G[\nabla_pF^{\otimes 2}] +\frac{1}{12}\nabla_pF\nabla_{qq}G\nabla_{pp}F\nabla_{qq}G\nabla_{p}F +\frac{1}{24}\nabla_{qqqq}G[\nabla_{p}F^{\otimes 4}]\nonumber \\ 
&=\frac{1}{24}\nabla_{pppp}F[\nabla_qG^{\otimes 4}]-\frac{1}{12}\nabla_{p}F\nabla_{qq}G\nabla_{ppp}F[\nabla_qG^{\otimes 2}]+\frac{1}{12}\nabla_pF\nabla_{qq}G\nabla_{pp}F\nabla_{qq}G\nabla_{p}F\nonumber \\
&\quad-\frac{1}{12}\nabla_{q}G\nabla_{pp}F\nabla_{qq}G\nabla_{pp}F\nabla_{q}G-\frac{1}{6}[\nabla_{p}F^{\otimes 2}]\nabla_{qqq}G\nabla_{pp}F\nabla_{q}G+\frac{1}{24}\nabla_{qqqq}G[\nabla_{p}F^{\otimes 4}], \label{C22}\\
    \Omega_1(H_3) &=
-\nabla_qG \nabla_p(H_3)+\nabla_pF\nabla_q(H_3) \nonumber \\
&=\frac{1}{12}\nabla_qG\nabla_{pp}F\nabla_{qq}G\nabla_{pp}F\nabla_{q}G+\frac{1}{12}[\nabla_qG^{\otimes 2}]\nabla_{ppp}F\nabla_{qq}G\nabla_pF\nonumber \\
&\quad -\frac{1}{12}\nabla_pF\nabla_{qq}G\nabla_{pp}F\nabla_{qq}G\nabla_{p}F-\frac{1}{12}\nabla_{q}G\nabla_{pp}F\nabla_{qqq}G[\nabla_{p}F^{\otimes 2}] \nonumber \\
&=\frac{1}{12}[\nabla_qG^{\otimes 2}]\nabla_{ppp}F\nabla_{qq}G\nabla_pF+\frac{1}{12}\nabla_qG\nabla_{pp}F\nabla_{qq}G\nabla_{pp}F\nabla_{q}\nonumber \\
&\quad -\frac{1}{12}\nabla_pF\nabla_{qq}G\nabla_{pp}F\nabla_{qq}G\nabla_{p}F-\frac{1}{12}\nabla_{q}G\nabla_{pp}F\nabla_{qqq}G[\nabla_{p}F^{\otimes 2}].\label{C31}
\end{align}

We now sum \eqref{C04}, \eqref{C13}, \eqref{C22} and \eqref{C31}:
\begin{align}
&\Omega_4(H_0)+\Omega_3(H_1)+\Omega_2(H_2)+\Omega_1(H_3) = 
\nonumber \\
&=\left(-\frac{1}{8}+\frac{1}{12}+\frac{1}{24}\right)\nabla_{pppp}F[\nabla_qG^{\otimes 4}]
+\left(\frac{1}{2}-\frac{1}{2}+\frac{1}{12}-\frac{1}{12}\right)\nabla_pF\nabla_{qq}G\nabla_{ppp}F[\nabla_qG^{\otimes 2}] \nonumber \\
&\quad +\left(\frac{1}{2}-\frac{1}{2}-\frac{1}{12}+\frac{1}{12}\right)\nabla_qG\nabla_{pp}F\nabla_{qq}G\nabla_{pp}F\nabla_{q}G
+\left(\frac{1}{12}-\frac{1}{12}\right)\nabla_pF\nabla_{qq}G\nabla_{pp}F\nabla_{qq}G\nabla_{p}F \nonumber \\
&\quad +\left(-\frac{1}{2}+\frac{3}{4}-\frac{1}{6}-\frac{1}{12}\right)[\nabla_pF^{\otimes 2}]\nabla_{qqq}G\nabla_{pp}F\nabla_qG+\left(\frac{1}{24}-\frac{1}{12}+\frac{1}{24}\right)\nabla_{qqqq}G[\nabla_pF^{\otimes 4}] \nonumber \\
&=0.
\end{align}
Thus, the case where $N=3$ is verified. To get an upper bound on the $N=2$ case, assuming that $F$ and $G$ are $L$-smooth of orders $1,\dots,4$, \eqref{coefs} implies
\begin{align*}
    \left|\widetilde{H}_{\eta}^{\left(2\right)}\left(p_{n+1},q_{n+1}\right)-\widetilde{H}_{\eta}^{\left(2\right)}\left(p_{n},q_{n}\right)\right|& \leq\eta^{4}\left(\max_{z\in L_{n}}\left|\Omega_{4}\left(H_{0}\right)\left(z\right)\right|+\max_{z\in L_{n}}\left|\Omega_{3}\left(H_{1}\right)\left(z\right)\right|+\max_{z\in L_{n}}\left|\Omega_{2}\left(H_{2}\right)\left(z\right)\right|\right)\\
    & \leq 
    \frac{49}{12}L^{5}\eta^{4}.
\end{align*}

\paragraph{For $N=4$ (abridged):}
Since we are only showing Conjecture \ref{Conjecture_MH_General} up to $N=3$, we do not need to show that the diagonals of \eqref{upper_triangular_MH_trunc} cancel at $N=4$. However, we still need to find the remainder for $N=3$, which requires computing $\Omega_{5}\left(H_{0}\right)$, $\Omega_{4}\left(H_{1}\right)$, $\Omega_{3}\left(H_{2}\right)$, and $\Omega_{2}\left(H_{3}\right)$.

\begin{align}
    \Omega_5(H_0)&=-\frac{1}{120}\nabla_{ppppp}\left(H_{0}\right)\left[\nabla_{q}G^{\otimes 5}\right]+\frac{1}{24}\nabla_{ppppq}\left(H_{0}\right)\left[\nabla_{q}G^{\otimes 4},\nabla_{p}F\right]+\frac{1}{120}\nabla_{qqqqq}\left(H_{0}\right)\left[\nabla_{p}F^{\otimes 5}\right]\nonumber\\
&\quad+\frac{1}{6}\nabla_{pppq}\left(H_{0}\right)\left[\nabla_{q}G^{\otimes 3},\nabla_{pp}F\nabla_{q}G\right]-\frac{1}{12}\nabla_{pppqq}\left(H_{0}\right)\left[\nabla_{q}G^{\otimes 3},\nabla_{p}F^{\otimes 2}\right]\nonumber\\
&\quad+\frac{1}{4}\nabla_{ppq}\left(H_{0}\right)\left[\nabla_{q}G^{\otimes 2},\nabla_{ppp}F\left[\nabla_{q}G^{\otimes 2}\right]\right]-\frac{1}{2}\nabla_{ppqq}\left(H_{0}\right)\left[\nabla_{q}G^{\otimes 2},\nabla_{p}F,\nabla_{pp}F\nabla_{q}G\right]\nonumber\\
&\quad+\frac{1}{12}\nabla_{ppqqq}\left(H_{0}\right)\left[\nabla_{q}G^{\otimes 2},\nabla_{p}F^{\otimes 3}\right]+\frac{1}{6}\nabla_{pq}\left(H_{0}\right)\left[\nabla_{q}G,\nabla_{pppp}F\left[\nabla_{q}G^{\otimes 3}\right]\right]\nonumber\\ 
&\quad-\frac{1}{2}\nabla_{pqq}\left(H_{0}\right)\left[\nabla_{q}G,\nabla_{p}F,\nabla_{ppp}F\left[\nabla_{q}G^{\otimes 2}\right]\right]+\frac{1}{2}\nabla_{pqqq}\left(H_{0}\right)\left[\nabla_{q}G,\nabla_{p}F^{\otimes 2},\nabla_{pp}F\nabla_{q}G\right]\nonumber\\
&\quad-\frac{1}{2}\nabla_{pqq}\left(H_{0}\right)\left[\nabla_{q}G,\left(\nabla_{pp}F\nabla_{q}G\right)^{\otimes 2}\right]-\frac{1}{24}\nabla_{pqqqq}\left(H_{0}\right)\left[\nabla_{q}G,\nabla_{p}F^{\otimes 4}\right]\nonumber\\
&\quad+\frac{1}{24}\nabla_{q}\left(H_{0}\right)\nabla_{ppppp}F\left[\nabla_{q}G^{\otimes 4}\right]-\frac{1}{6}\nabla_{qq}\left(H_{0}\right)\left[\nabla_{p}F,\nabla_{pppp}F\left[\nabla_{q}G^{\otimes 3}\right]\right]\nonumber\\
&\quad+\frac{1}{4}\nabla_{qqq}\left(H_{0}\right)\left[\nabla_{p}F^{\otimes 2},\nabla_{ppp}F\left[\nabla_{q}G^{\otimes 2}\right]\right]-\frac{1}{6}\nabla_{qqqq}\left(H_{0}\right)\left[\nabla_{p}F^{\otimes 3},\nabla_{pp}F\nabla_{q}G\right]\nonumber\\
&\quad+\frac{1}{2}\nabla_{qqq}\left(H_{0}\right)\left[\nabla_{p}F,\left(\nabla_{pp}F\nabla_{q}G\right)^{\otimes 2}\right]-\frac{1}{2}\nabla_{qq}\left(H_{0}\right)\left[\nabla_{pp}F\nabla_{q}G,\nabla_{ppp}F\left[\nabla_{q}G^{\otimes 2}\right]\right]\nonumber\\
&=-\frac{1}{120}\nabla_{ppppp}F\left[\nabla_{q}G^{\otimes 5}\right]+0+\frac{1}{120}\nabla_{qqqqq}G\left[\nabla_{p}F^{\otimes 5}\right]+0-0+0-0+0+0-0+0\nonumber\\
&\quad-0-0+\frac{1}{24}\nabla_{q}G\nabla_{ppppp}F\left[\nabla_{q}G^{\otimes 4}\right]-\frac{1}{6}\nabla_{qq}G\left[\nabla_{p}F,\nabla_{pppp}F\left[\nabla_{q}G^{\otimes 3}\right]\right]\nonumber\\
&\quad+\frac{1}{4}\nabla_{qqq}G\left[\nabla_{p}F^{\otimes 2},\nabla_{ppp}F\left[\nabla_{q}G^{\otimes 2}\right]\right]-\frac{1}{6}\nabla_{qqqq}G\left[\nabla_{p}F^{\otimes 3},\nabla_{pp}F\nabla_{q}G\right]\nonumber\\
&\quad+\frac{1}{2}\nabla_{qqq}G\left[\nabla_{p}F,\left(\nabla_{pp}F\nabla_{q}G\right)^{\otimes 2}\right]-\frac{1}{2}\nabla_{qq}G\left[\nabla_{pp}F\nabla_{q}G,\nabla_{ppp}F\left[\nabla_{q}G^{\otimes 2}\right]\right]\nonumber\\
&=\frac{1}{30}\nabla_{ppppp}F\left[\nabla_{q}G^{\otimes 5}\right]+\frac{1}{120}\nabla_{qqqqq}G\left[\nabla_{p}F^{\otimes 5}\right]-\frac{1}{6}\nabla_{qq}G\left[\nabla_{p}F,\nabla_{pppp}F\left[\nabla_{q}G^{\otimes 3}\right]\right]\nonumber\\
&\quad+\frac{1}{4}\nabla_{qqq}G\left[\nabla_{p}F^{\otimes 2},\nabla_{ppp}F\left[\nabla_{q}G^{\otimes 2}\right]\right]-\frac{1}{6}\nabla_{qqqq}G\left[\nabla_{p}F^{\otimes 3},\nabla_{pp}F\nabla_{q}G\right]\nonumber\\
&\quad+\frac{1}{2}\nabla_{qqq}G\left[\nabla_{p}F,\left(\nabla_{pp}F\nabla_{q}G\right)^{\otimes 2}\right]-\frac{1}{2}\nabla_{qq}G\left[\nabla_{pp}F\nabla_{q}G,\nabla_{ppp}F\left[\nabla_{q}G^{\otimes 2}\right]\right],\label{C50}
\end{align}

\begin{align}
    \Omega_4(H_1)&=\frac{1}{24}\nabla_{pppp}(H_{1})[\nabla_qG^{\otimes 4}]-\frac{1}{6}\nabla_{pppq}(H_{1})[\nabla_qG^{\otimes 3}, \nabla_pF]-\frac{1}{2}\nabla_{ppq}(H_{1})[\nabla_qG^{\otimes 2}, \nabla_{pp}F\nabla_qG]\nonumber\\ 
    &\quad+\frac{1}{4}\nabla_{ppqq}(H_{1})[\nabla_qG^{\otimes 2}, \nabla_pF^{\otimes 2}]-\frac{1}{2}\nabla_{pq}(H_{1})[\nabla_qG, \nabla_{ppp}F[\nabla_qG^{\otimes 2}]]\nonumber\\ &\quad+\nabla_{pqq}(H_{1})[\nabla_qG, \nabla_pF, \nabla_{pp}F\nabla_qG]-\frac{1}{6}\nabla_{pqqq}(H_{1})[\nabla_qG, \nabla_pF^{\otimes 3}]\nonumber\\ 
    &\quad-\frac{1}{6}\nabla_q(H_{1})\nabla_{pppp}F[\nabla_qG^{\otimes 3}]+\frac{1}{2}\nabla_{qq}(H_{1})[\nabla_{p}F, \nabla_{ppp}{F}[\nabla_{q}G^{\otimes 2}]]\nonumber \\
&\quad-\frac{1}{2}\nabla_{qqq}(H_{1})[\nabla_pF^{\otimes 2}, \nabla_{pp}F\nabla_qG]+\frac{1}{2}\nabla_{qq}(H_{1})[(\nabla_{pp}F\nabla_qG)^{\otimes 2}]+\frac{1}{24}\nabla_{qqqq}(H_{1})[\nabla_pF^{\otimes 4}]\nonumber\\
&=-\frac{1}{48}\nabla_{ppppp}F[\nabla_qG^{\otimes 5}]+\frac{1}{12}\nabla_{pppp}F[\nabla_qG^{\otimes 3}, \nabla_{qq}G\nabla_pF]+\frac{1}{4}\nabla_{ppp}F[\nabla_qG^{\otimes 2},\nabla_{qq}G\nabla_{pp}F\nabla_qG]\nonumber\\ 
&\quad-\frac{1}{8}\nabla_{ppp}F[\nabla_qG^{\otimes 2},\nabla_{qqq}G\left[\nabla_pF^{\otimes 2}\right]]+\frac{1}{4}\nabla_{pp}F[\nabla_qG,\nabla_{qq}G\nabla_{ppp}F[\nabla_qG^{\otimes 2}]]\nonumber\\ &\quad-\frac{1}{2}\nabla_{qqq}G[\nabla_{pp}F\nabla_qG, \nabla_pF, \nabla_{pp}F\nabla_qG]+\frac{1}{12}\nabla_{qqqq}G[\nabla_{pp}F\nabla_qG, \nabla_pF^{\otimes 3}]\nonumber\\ 
&\quad+\frac{1}{12}\nabla_{pppp}F[\nabla_{qq}G\nabla_{p}F,\nabla_qG^{\otimes 3}]-\frac{1}{4}\nabla_{qqq}G[\nabla_{p}F^{\otimes 2}, \nabla_{ppp}{F}[\nabla_{q}G^{\otimes 2}]]\nonumber \\
&\quad+\frac{1}{4}\nabla_{qqqq}G[\nabla_pF^{\otimes 3}, \nabla_{pp}F\nabla_qG]-\frac{1}{4}\nabla_{qqq}G[\nabla_{p}F,(\nabla_{pp}F\nabla_qG)^{\otimes 2}]-\frac{1}{48}\nabla_{qqqqq}G[\nabla_pF^{\otimes 5}]\nonumber\\
&=-\frac{1}{48}\nabla_{ppppp}F[\nabla_qG^{\otimes 5}]+\frac{1}{6}\nabla_{pppp}F[\nabla_qG^{\otimes 3}, \nabla_{qq}G\nabla_pF]\nonumber\\ 
&\quad+\frac{1}{2}\nabla_{ppp}F[\nabla_qG^{\otimes 2},\nabla_{qq}G\nabla_{pp}F\nabla_qG]-\frac{3}{8}\nabla_{ppp}F[\nabla_qG^{\otimes 2},\nabla_{qqq}G\left[\nabla_pF^{\otimes 2}\right]]\nonumber\\ 
&\quad-\frac{3}{4}\nabla_{qqq}G[\nabla_pF,\left(\nabla_{pp}F\nabla_qG\right)^{\otimes 2}]+\frac{1}{3}\nabla_{qqqq}G[\nabla_pF^{\otimes 3}, \nabla_{pp}F\nabla_qG]-\frac{1}{48}\nabla_{qqqqq}G[\nabla_pF^{\otimes 5}],\label{C41}
\end{align}

\begin{align}
    \Omega_3(H_2)&=\frac{1}{2}\nabla_q\left(H_{2}\right) \nabla_{ppp}F[\nabla_qG^{\otimes 2}]+\nabla_{pq}\left(H_{2}\right)[\nabla_qG, \nabla_{pp}F\nabla_qG]-\nabla_{qq}\left(H_{2}\right)[\nabla_pF, \nabla_{pp}F\nabla_qG]
\nonumber \\
&\quad-\frac{1}{6}\nabla_{ppp}\left(H_{2}\right)[\nabla_q(G)^{\otimes 3}]
+\frac{1}{2}\nabla_{ppq}\left(H_{2}\right)[\nabla_qG^{\otimes 2}, \nabla_pF]
\nonumber \\
&\quad-\frac{1}{2}\nabla_{pqq}\left(H_{2}\right)[\nabla_qG, \nabla_pF^{\otimes 2}]
+\frac{1}{6}\nabla_{qqq}\left(H_{2}\right)[\nabla_p(F)^{\otimes 3}]\nonumber\\
&=\frac{1}{12}\nabla_{ppp}F\left[\nabla_{qq}G\nabla_{pp}F\nabla_qG,\nabla_qG^{\otimes 2}\right]+\frac{1}{24}\nabla_{ppp}F\left[\nabla_{qqq}G\left[\nabla_pF^{\otimes 2}\right],\nabla_qG^{\otimes 2}\right]\nonumber\\ &\quad+\frac{1}{6}\nabla_{ppp}F\left[\nabla_{qq}G\nabla_{pp}F\nabla_qG,\nabla_qG^{\otimes 2}\right]+\frac{1}{6}\nabla_{qqq}G\left[\nabla_pF,\left(\nabla_{pp}F\nabla_qG\right)^{\otimes 2}\right]\nonumber\\ &\quad-\frac{1}{6}\nabla_{qqq}G\left[\nabla_pF,\left(\nabla_{pp}F\nabla_qG\right)^{\otimes 2}\right]-\frac{1}{6}\nabla_{qq}G\nabla_{pp}F\nabla_{qq}G[\nabla_pF, \nabla_{pp}F\nabla_qG]\nonumber \\
&\quad-\frac{1}{12}\nabla_{qqqq}G\left[\nabla_pF^{\otimes 3}, \nabla_{pp}F\nabla_qG\right]
-\frac{1}{72}\nabla_{ppppp}F\left[\nabla_qG^{\otimes 5}\right]-\frac{1}{36}\nabla_{pppp}F\left[\nabla_{qq}G\nabla_pF,\nabla_qG^{\otimes 3}\right]\nonumber\\&\quad-\frac{1}{12}\nabla_{ppp}F\left[\nabla_{qq}G\nabla_{pp}F\nabla_qG,\nabla_qG^{\otimes 2}\right]
+\frac{1}{12}\nabla_{pppp}F\left[\nabla_qG^{\otimes 3}, \nabla_{qq}G\nabla_pF\right]\nonumber \\
&\quad+\frac{1}{12}\nabla_{qqq}G\left[\left(\nabla_{pp}F\nabla_qG\right)^{\otimes 2},\nabla_pF\right]+\frac{1}{12}\nabla_{qqq}G\left[\nabla_{ppp}F\left[\nabla_qG^{\otimes 2}\right],\nabla_pF^{\otimes 2}\right]
\nonumber \\
&\quad-\frac{1}{12}\nabla_{qqq}G\left[\nabla_{ppp}F\left[\nabla_qG^{\otimes 2}\right], \nabla_pF^{\otimes 2}]\right]-\frac{1}{12}\nabla_{ppp}F\left[\nabla_qG,\left(\nabla_{qq}G\nabla_pF\right)^{\otimes 2}\right]\nonumber\\ &\quad-\frac{1}{12}\nabla_{qqqq}G\left[\nabla_{pp}F\nabla_qG,\nabla_pF^{\otimes 3}\right]
+\frac{1}{36}\nabla_{qqqq}G\left[\nabla_{pp}F\nabla_qG,\nabla_pF^{\otimes 3}\right]\nonumber\\ &\quad+\frac{1}{12}\nabla_{qqq}G\left[\nabla_{pp}F\nabla_{qq}G\nabla_pF,\nabla_pF^{\otimes 2}\right]+\frac{1}{72}\nabla_{qqqqq}G\left[\nabla_pF^{\otimes 5}\right]\nonumber\\
&=\frac{1}{6}\nabla_{ppp}F\left[\nabla_{qq}G\nabla_{pp}F\nabla_qG,\nabla_qG^{\otimes 2}\right]+\frac{1}{24}\nabla_{ppp}F\left[\nabla_{qqq}G\left[\nabla_pF^{\otimes 2}\right],\nabla_qG^{\otimes 2}\right]\nonumber\\ &\quad-\frac{1}{6}\nabla_{qq}G\nabla_{pp}F\nabla_{qq}G[\nabla_pF, \nabla_{pp}F\nabla_qG]\nonumber \\
&\quad-\frac{5}{36}\nabla_{qqqq}G\left[\nabla_pF^{\otimes 3}, \nabla_{pp}F\nabla_qG\right]
-\frac{1}{72}\nabla_{ppppp}F\left[\nabla_qG^{\otimes 5}\right]\nonumber\\&\quad
+\frac{1}{18}\nabla_{pppp}F\left[\nabla_qG^{\otimes 3}, \nabla_{qq}G\nabla_pF\right]\nonumber \\
&\quad+\frac{1}{12}\nabla_{qqq}G\left[\left(\nabla_{pp}F\nabla_qG\right)^{\otimes 2},\nabla_pF\right]-\frac{1}{12}\nabla_{ppp}F\left[\nabla_qG,\left(\nabla_{qq}G\nabla_pF\right)^{\otimes 2}\right]\nonumber\\ &\quad+\frac{1}{12}\nabla_{qqq}G\left[\nabla_{pp}F\nabla_{qq}G\nabla_pF,\nabla_pF^{\otimes 2}\right]+\frac{1}{72}\nabla_{qqqqq}G\left[\nabla_pF^{\otimes 5}\right],\label{C32}
\end{align}

\begin{align}
    \Omega_2(H_3)&=\nabla_q\left(H_{3}\right) \nabla_p\{F, G\}
+\frac{1}{2}\nabla_{pp}\left(H_{3}\right)\left[\nabla_qG^{\otimes 2}\right]
\nonumber \\ &\quad-\nabla_{pq}\left(H_{3}\right)[\nabla_qG, \nabla_pF]
+\frac{1}{2}\nabla_{qq}\left(H_{3}\right)\left[\nabla_pF^{\otimes 2}\right]\nonumber\\
&=\frac{1}{12}\nabla_{qqq}G\left[\nabla_pF,\left(\nabla_{pp}F\nabla_qG\right)^{\otimes 2}\right]+\frac{1}{12}\nabla_{qq}G\nabla_{pp}F\nabla_{qq}G\left[\nabla_pF,\nabla_{pp}F\nabla_{q}G\right]
\nonumber\\ &\quad-\frac{1}{8}\nabla_{ppp}F\left[\nabla_{qq}G\nabla_{pp}F\nabla_qG,\nabla_qG^{\otimes 2}\right]-\frac{1}{24}\nabla_{pppp}F\left[\nabla_{qq}G\nabla_pF,\nabla_qG^{\otimes 3}\right]
\nonumber \\ &\quad+\frac{1}{12}\nabla_{qqq}G\left[\left(\nabla_{pp}F\nabla_qG\right)^{\otimes 2},\nabla_pF\right]+\frac{1}{12}\nabla_{qqq}G\left[\nabla_{ppp}F\left[\nabla_qG^{\otimes 2}\right], \nabla_pF^{\otimes 2}\right]\nonumber\\ &\quad+\frac{1}{12}\nabla_{pp}F\nabla_{qq}G\nabla_{pp}F\nabla_{qq}G[\nabla_qG, \nabla_pF]+\frac{1}{12}\nabla_{ppp}F\left[\nabla_qG, \left(\nabla_{qq}G\nabla_pF\right)^{\otimes 2}\right]
\nonumber\\ &\quad-\frac{1}{24}\nabla_{qqqq}G\left[\nabla_{pp}F\nabla_qG,\nabla_pF^{\otimes 3}\right]-\frac{1}{8}\nabla_{qqq}G\left[\nabla_{pp}F\nabla_{qq}G\nabla_pF,\nabla_pF^{\otimes 2}\right]\nonumber\\
&=\frac{1}{6}\nabla_{qqq}G\left[\left(\nabla_{pp}F\nabla_qG\right)^{\otimes 2},\nabla_pF\right]+\frac{1}{6}\nabla_{qq}G\nabla_{pp}F\nabla_{qq}G\left[\nabla_pF,\nabla_{pp}F\nabla_{q}G\right]
\nonumber\\ &\quad-\frac{1}{8}\nabla_{ppp}F\left[\nabla_{qq}G\nabla_{pp}F\nabla_qG,\nabla_qG^{\otimes 2}\right]-\frac{1}{24}\nabla_{pppp}F\left[\nabla_{qq}G\nabla_pF,\nabla_qG^{\otimes 3}\right]
\nonumber \\ &\quad+\frac{1}{12}\nabla_{ppp}F\left[\nabla_qG, \left(\nabla_{qq}G\nabla_pF\right)^{\otimes 2}\right]+\frac{1}{12}\nabla_{qqq}G\left[\nabla_{ppp}F\left[\nabla_qG^{\otimes 2}\right], \nabla_pF^{\otimes 2}\right]
\nonumber\\ &\quad-\frac{1}{24}\nabla_{qqqq}G\left[\nabla_{pp}F\nabla_qG,\nabla_pF^{\otimes 3}\right]-\frac{1}{8}\nabla_{qqq}G\left[\nabla_{pp}F\nabla_{qq}G\nabla_pF,\nabla_pF^{\otimes 2}\right].\label{C23}
\end{align}
Hence, from \eqref{C50}, \eqref{C41}, \eqref{C32}, and \eqref{C23}, we can read off the upper bound on the $N=3$ case as follows, assuming that $F$ and $G$ are $L$-smooth of orders $1,\dots,5$:
\begin{align}
    &\left|\widetilde{H}_{\eta}^{\left(3\right)}\left(p_{n+1},q_{n+1}\right)-\widetilde{H}_{\eta}^{\left(3\right)}\left(p_{n},q_{n}\right)\right|\nonumber\\ &\leq\eta^{5}\left(\max_{z\in L_{n}}\left|\Omega_{5}\left(H_{0}\right)\left(z\right)\right|+\max_{z\in L_{n}}\left|\Omega_{4}\left(H_{1}\right)\left(z\right)\right|+\max_{z\in L_{n}}\left|\Omega_{3}\left(H_{2}\right)\left(z\right)\right|+\max_{z\in L_{n}}\left|\Omega_{2}\left(H_{3}\right)\left(z\right)\right|\right)\nonumber\\
    &\leq 
    \eta^{5}L^{6}\left(\frac{13}{8}+\frac{13}{6}+\frac{61}{72}+\frac{5}{6}\right)=\frac{197}{36}\eta^{5}L^{6}. \nonumber
\end{align}
In particular, from our work above, we have confirmed that $\Phi\left(0\right)=1$, $\Phi\left(1\right)=3$, $\Phi\left(2\right)=49/12$, and $\Phi\left(3\right)=197/36$ for arbitrary dimension $d\in\N$.

\subsection{Verification up to \texorpdfstring{$N=10$}{N=10} via symbolic calculation of the coefficients}\label{CompTaylor}

Unlike in Appendix \ref{SymbolicTaylor}, the following approach only works for $d=1$ dimensions since SymPy \citep{SymPy} does not currently have support for higher-order derivative tensors. To compute the successive terms of the modified Hamiltonian in SymPy, we implement the recursive form of the BCH formula as stated in \cite{Casas} and Section 2.15 in \cite{Varadarajan1984}, and proved as Lemma 2.15.3 in \cite{Varadarajan1984}. When applied to the setting of symplectic Euler and Alternating Mirror Descent, we recover the following for any $N\in\N$:
\begin{subequations}\label{recursive1}
\begin{align}
    H_{0}&=F+G,\\
    NH_{N}&=\frac{1}{2}\left\{H_{N-1},G-F\right\}+\sum_{p=1}^{\lfloor (N-1)/2 \rfloor}{\frac{B_{2p}}{\left(2p\right)!}\left(\ad_{\widetilde{H}}^{2p}\left(F+G\right)\right)_{N}},
\end{align}
\end{subequations}
where
\begin{gather}
\left(\ad_{\widetilde{H}}^{2p}\left(F+G\right)\right)_{N}=\sum_{\substack{k_{1}+\dots+k_{2p}=N-1 \\ k_{1}\geq 1,\dots,k_{2p}\geq 1}}{\left\{\left(F+G\right)H_{k_{2p}-1}\cdots H_{k_{1}-1}\right\}}\label{recursive2}
\end{gather}
denotes the projection of $\ad_{\widetilde{H}}^{2p}\left(F+G\right)$ onto the homogeneous subspace $\mathcal{L}\left(F,G\right)_{N+1}$ consisting of degree $N+1$ elements from $\mathcal{L}\left(F,G\right)$, as defined in \eqref{Eq:FreeLieAlgebra}.

In SymPy \citep{SymPy}, we use the recursive formula above to compute the order-by-order corrections to the modified Hamiltonian $H_{k}$. We have uploaded our code onto a repository, \cite{katona2024amd-code}. These are then used to compute the values of $C_{j,k}=\Omega_j(H_k)$ for each fixed value of $j+k-1=N$ to check Conjecture \ref{Conjecture_MH_General} at each order in $N$. These are also used to compute the remainders for fixed values of $j+k=N+1$. 

We essentially use the formula \eqref{Omega} to compute the action of $\Omega_{j}$ on each $H_{k}$ as needed. However, we do this by expanding both exponentials as power series in $\eta$ up to sufficiently high order, from which we can compute the derivatives symbolically with respect to $\eta$ and then evaluate the result at $\eta=0$.

The necessary computations to check Conjecture \ref{Conjecture_MH_General} when $N=11$ can in principle be run in Python. 
However, due to the explosion in the number of terms, the program ran into memory and rounding issues when trying to cancel out all of the fractions in the resultant expressions and show that the terms vanished.
We leave an investigation of a more efficient program for the verification of the above conjecture as future work.

\subsection{Concluding the proof}\label{concluding}

In the prior subsections of Appendix \ref{Appendix_Thm5}, we confirmed \eqref{awesome} for orders $N \in \{0,1,2,3\}$ by direct computation for arbitrry dimension $d$; 
we also confirmed the cases $N \in \{0,1,\dots,10\}$ by computer for dimension $d=1$. However, \eqref{awesome} only bounds the difference in $\widetilde{H}_{\eta}^{(N)}$ after one iteration while Conjecture \ref{Conjecture_MH_General} gives the same for multiple iterations. To derive a bound after $k$ iterations, we apply \eqref{awesome} recursively and then the triangle inequality as follows:
\begin{align*}
    \left|\widetilde{H}_{\eta}^{\left(N\right)}\left(p_{k},q_{k}\right)-\widetilde{H}_{\eta}^{\left(N\right)}\left(p_{0},q_{0}\right)\right|
    &\le \sum_{j=0}^{k-1}{\left|\widetilde{H}_{\eta}^{\left(N\right)}\left(p_{j+1},q_{j+1}\right)-\widetilde{H}_{\eta}^{\left(N\right)}\left(p_{j},q_{j}\right)\right|}
    \notag \\ &\leq k\Phi\left(N\right)L^{N+3}\eta^{N+2},
\end{align*}
where in the first inequality above we use triangle inequality. This concludes the proof. \quad \boxed{}

\section{Characterization of \texorpdfstring{$\Phi$}{Phi}}\label{Phi}

The coefficient function $\Phi:\N_{0}\rightarrow\mathbb{Q}$ from Conjecture \ref{Conjecture_MH_General} has the following known values for $N=0,1,\dots,10$:
    \vspace{0.4cm}
    {
\par\centering\resizebox{\columnwidth}{0.75cm}{
    \begin{tabular}{ |c|c|c|c|c|c|c|c|c|c|c|c| } 
     \hline
     $N$ & $0$ & $1$ & $2$ & $3$ & $4$ & $5$ & $6$ & $7$ & $8$ & $9$ & $10$ \\ 
     \hline
     $\Phi\left(N\right)$ (fraction) & $1$ & $3$ & $\frac{49}{12}$ & $\frac{197}{36}$ & $\frac{139}{20}$ & $\frac{1049}{108}$ & $\frac{850271}{60480}$ & $\frac{10046117}{453600}$ & $\frac{16791911}{453600}$ & $\frac{73961467}{1088640}$ & $\frac{31567569067}{239500800}$\\ 
     \hline 
     $\Phi\left(N\right)$ (rounded) & $1.000$ & $3.000$ & $4.083$ & $5.472$ & $6.950$ & $9.713$ & $14.059$ & $22.148$ & $37.019$ & $67.939$ & $131.806$ \\ 
     \hline 
     \end{tabular}}\par}\vspace{0.4cm}

It is not immediately obvious as to why $\Phi$ should be bounded. However, up to $N=10$, we note that $\Phi\left(N\right)=2\frac{\gamma(N)}{\sigma\left(N+2\right)}$, where $\gamma:\N\rightarrow\N$ returns an unknown sequence of positive integers 
\vspace{0.4cm}
    {
\par\centering\resizebox{\columnwidth}{0.55cm}{
    \begin{tabular}{ |c|c|c|c|c|c|c|c|c|c|c|c| } 
     \hline
     $N$ & $0$ & $1$ & $2$ & $3$ & $4$ & $5$ & $6$ & $7$ & $8$ & $9$ & $10$ \\ 
     \hline
     $\gamma\left(N\right)$ & $1$ & $18$ & $49$ & $1970$ & $5004$ & $293720$ & $850271$ & $40184468$ & $134335288$ & $16271522740$ & $63135138134$  \\ 
     \hline 
     \end{tabular}}\par}\vspace{0.4cm}
and $\sigma:\N\rightarrow\N$ returns the Hirzebruch numbers, which can be defined as follows \citep{Hirzebuch}: 
\begin{gather}
    \sigma\left(n\right)\coloneqq \prod_{\substack{2\leq p\leq n+1 \\ p\textrm{ prime}}}{p^{\lfloor \frac{n}{p-1}\rfloor}}\quad \forall n\in\N.\nonumber
\end{gather}
As proven in \cite[Corollary 9]{Bedhouche}, $\sigma\left(N\right)=\mathcal{O}\left(N^{N}\right)$ for $N$ large. Furthermore, $\gamma\left(N\right)\leq\frac{2}{3}\left(N+2\right)^{N+2}$ is true up to $N=10$ and \textit{only} saturates at $N=1$. Thus, assuming that $\mathcal{O}\left(\gamma\left(N\right)\right)=\mathcal{O}\left(N^{N}\right)$ remains true for $N>10$, it remains plausible (although hitherto unproven) that $\mathcal{O}\left(\Phi\left(N\right)\right)=\mathcal{O}\left(1\right)$ for large $N$. Hence, we suspect that $\Phi$ is monotonic increasing and bounded as claimed in Conjecture \ref{Conjecture_MH_General}.

\bibliographystyle{alpha}
\bibliography{main}

\end{document}